\documentclass[a4paper,11pt,reqno]{amsart}
\usepackage[utf8]{inputenc}
\usepackage{enumerate}
\usepackage{amsmath, amssymb,amsthm}
\usepackage{mathtools}
\usepackage{leftidx}
\usepackage{setspace}
\numberwithin{equation}{section}
\usepackage[left=2.8cm, right=2.8cm, top=2.7cm]{geometry}
\usepackage{hyperref}
\usepackage{xcolor}
\definecolor{my-black}{rgb}{0,0,0}
\definecolor{my-blue}{rgb}{0,0,0.8}
\definecolor{my-red}{rgb}{0.8,0,0} 
\definecolor{my-green}{rgb}{0,0.5,0}
\hypersetup{colorlinks,
    linkcolor	= {my-blue},
    citecolor	= {my-red},
    urlcolor		= {my-black}
}
\usepackage{graphicx}
\usepackage{tikz-cd}
\theoremstyle{plain} 

\newtheorem{lemma}{Lemma}[section]
\newtheorem{theorem}{Theorem}
\newtheorem{corollary}{Corollary}[section]
\newtheorem{proposition}{Proposition}[section]

\newtheorem*{theorem*}{Theorem} 

\theoremstyle{definition} 
\newtheorem{remark}{Remark}[section]

\newtheorem{definition}{Definition}[section]

\theoremstyle{remark}
\newtheorem*{remark-non}{Remark}

\newcommand{\R}{\mathbb{R}}

\newcommand{\C}{\mathbb{C}}
\newcommand{\N}{\mathbb{N}}

\newcommand{\Z}{\mathbb{Z}}
\newcommand{\T}{\mathbb{T}}

\allowdisplaybreaks

\title[Time-frequency localization]{Eigenfunctions, Free Boundaries, and Time-Frequency Localization}
\author{Jo{\~a}o P.G. Ramos}
\address{Instituto de Matemática Pura e Aplicada (IMPA) - Estrada Dona Castorina 110, 22460-320, Jardim Botânico, Rio de Janeiro - RJ, Brazil}
\email{joao.ramos@impa.br}

\begin{document}
\begin{abstract}
We develop an inverse theory for time--frequency localization
operators, whose central idea  that of is free-boundary problem: the localization
domain is unknown and its boundary is recovered from prescribed spectral data.
The approach is based on the principle that an eigenfunction may be
regarded as geometric data which determines a localization domain, and
prescribing it has strong consequences for the associated variational problem.

Five main results follow from this main framework. First, if \(f_0\) is a
polynomial sufficiently close to the Gaussian and \(\lambda\in(0,1)\), we
construct a real-analytic domain \(U_\lambda\)such that $f_0$ is an eigenfunction of the localization operator associated with $U_\lambda$ and with eigenvalue $\lambda$, giving a general inverse
construction of localization domains, which is the first in the literature. Second, we recover the null-set invariant
Abreu--D\"orfler characterization of disks as the only localization domains of Hermite polynomials in the simply connected case. Third, we show that simple connectivity cannot be dropped: a weighted Hele--Shaw injection produces an infinite-dimensional family of non-radial, doubly connected analytic domains with the Gaussian as an eigenfunction of the associated localization operator. Fourth, we
prove optimality of the exponent \(1/2\) in the
G\'omez--Guerra--Ramos--Tilli quantitative stability inequality, answering thus a question posed by those authors in \cite{Gomez-Guerra-Ramos-Tilli}. Finally,
we show that local maximizers of the Gaussian Faber--Krahn problem are disks, which extends the Nicola--Tilli concentration inequality to the local case as well, and, as a matter of fact, as a consequence of a Fock-space concentration--compactness profile decomposition, we are able to use this to give a new, different proof of the Nicola--Tilli theorem.
\end{abstract}
\maketitle


\medskip



\section{Introduction}

How does a function concentrate in phase space? This is the central question of time--frequency analysis, dating back to Wigner, Gabor, and Daubechies \cite{Wigner,Gabor,Daubechies}, and it admits a precise formulation in terms of the \emph{short-time Fourier transform:}
\begin{equation}\label{eq:STFT-intro}
V_\varphi f(x,\omega) \;=\; \int_{\R} f(t)\,\overline{\varphi(t-x)}\,e^{-2\pi i t \omega}\,\mathrm{d}t,
\qquad f,\varphi\in L^2(\R),
\end{equation}
 whose squared modulus is the canonical phase-space energy density of $f$. Its integral over a prescribed set $\Omega\subset\R^2$ is the portion of the signal energy captured in that time--frequency region. Maximizing this quantity therefore identifies the unit-energy signal best concentrated in the region, or, equivalently, the waveform to which a time--frequency filter supported on $\Omega$ is most responsive. This leads to the spectral analysis of the \emph{time--frequency localization operator} $\mathcal L_\Omega$ introduced by Daubechies \cite{Daubechies}; see \cite{Grochenig,Mallat,Folland,Daubechies-Paul,Ramanathan-Topiwala-Weyl,Ramanathan-Topiwala-Spectrogram,Cordero-Grochenig,Boggiatto-Cordero-Grochenig,Bayer-Grochenig,Kisil-CrossToeplitz,Nicola-Tilli-2,Ramos-Optimal-Localization,Kalaj1,Kalaj2,Haslinger-Kalaj-Vujadinovic,Kalaj-Melentijevic-Bergman,Kalaj-Ramos-Hyperbolic-Ball,Melentijevic-Hypercontractive-Bergman,Melentijevic-Stability-Bergman,Jaguzovic-Melentijevic-Sphere,Gomez-Kalaj-Melentijevic-Ramos,Kulikov,Kulikov-Hardy-Littlewood,Kulikov-Fourier-Interpolation,Kulikov-Exponential-Plunge,Kulikov-Nicola-Ortega-Tilli,Kulikov-Sharp-Plunge} for the broader theory and \cite{Lieb,abreu2021donoho} for the underlying uncertainty and large-sieve phenomena.

Among the many possible windows, we highlight the Gaussian $\varphi(x)=2^{1/4}e^{-\pi x^2}$. For this normalized window, we write
\[
\mathcal L_\Omega:=V_\varphi^*\mathbf 1_\Omega V_\varphi,
\qquad
\lambda_1(\Omega)=\lambda_{1,\varphi}(\Omega):=\|\mathcal L_\Omega\|_{\mathrm{op}}.
\]
This is an important window for applications, as it gives optimal joint time--frequency resolution and dampens eventual oscillations caused by sharp cutoffs, acting thus as a regular counterpart to an indicator function. As a model, it is equally useful because its invariance under Fourier transform and its compatibility with translations and modulations make the localization problem analytically viable. In that regard, a central aspect of the short-time Fourier transform with Gaussian window - also known as the \emph{Gabor transform} is its relationship to analytic functions through the Bargmann transform \cite{Bargmann}. By means of that identification, the localization operator $\mathcal L_\Omega$ becomes anoperator on the Bargmann--Fock space $\mathcal F^2(\C)$ of entire functions \cite{Berezin,Berger-Coburn-Symbol,Berger-Coburn,Berger-Coburn-Heat,Zhu,Grochenig,Folland,Ramos-Optimal-Localization,Kalaj1,Kalaj2,Haslinger-Kalaj-Vujadinovic}, and time--frequency questions may be thus translated into the realm of complex analysis. In this direction, two main rigidity results shape the picture of the Gaussian setting:
\begin{itemize}
\item[\textbf{(NT)}] (Nicola--Tilli \cite{Nicola-Tilli}; see also \cite{Nicola-Tilli-2,Bastianoni-Cordero-Nicola,Bastianoni-Teofanov,Ramos-Tilli}) \emph{Geometric extremality.} Among all measurable sets of a given measure, disks maximize the first localization eigenvalue $\lambda_1$.
\item[\textbf{(AD)}] (Abreu--D\"orfler \cite{Abreu-Doerfler}) \emph{Spectral rigidity.} If a Hermite function is an eigenfunction of $\mathcal L_\Omega$ on a simply connected bounded domain, then \(\Omega\) is a disk up to null sets. The simple connectivity hypothesis is essential here; see Theorem~\ref{thm:counterexample} below.
\end{itemize}
These two theorems express, from opposite directions, the same heuristic: \emph{the special role of the Gaussian among signals is a sort of analogue of the special role of the disk among phase-space sets}. This intertwining, however, has been mostly restricted to very few examples, the most prominent of which being the two mentioned above, with both geometric and spectral aspects of the theory being developed independently of one another. The theory behind (NT) and its quantitative refinement \cite{Gomez-Guerra-Ramos-Tilli} is essentially based on rearrangement and integral-geometric inequalities, with inspiration stemming from classical calculus of variations. On the other hand, the rigidity result (AD) is proved by complex-analytic methods originating from potential theory. Most importantly, neither of these perspectives makes, on its own, makes the connection between eigenfunctions and domains constructive.

\medskip

\noindent\textbf{Eigenfunctions as geometric data.} The aim of this manuscript is to further establish and strengthen that connection. The unifying theme is that, for Gaussian time--frequency localization, an eigenfunction must be regarded as \emph{geometric data}: it determines a localization domain, and prescribing it has strong consequences to its variational theory. This theme is developed in three coupled steps.

\smallskip
\noindent\emph{(i) An inverse construction.} Near the Gaussian, the eigenvalue equation $\mathcal L_U f_0=\lambda f_0$ admits a perturbative solution \emph{in the unknown domain~$U$} for arbitrarily prescribed sufficiently small finite Hermite-polynomial perturbations $f_0$ of $h_0$. We parametrize the unknown boundary conformally near the Daubechies circle and, following Cherednichenko\ \cite{Cherednichenko1973,Cherednichenko1975}, solve the resulting holomorphic boundary equation by an analytic implicit-function theorem.

\smallskip
\noindent\emph{(ii) Constructive consequences of rigidity.} Eigenfunction rigidity yields two further conceptual proofs of the (AD) result, while the inverse construction yields the optimal exponent in the quantitative stability inequality of \cite{Gomez-Guerra-Ramos-Tilli}. The counterexamples delimit rigidity from the other side: once the symbol is allowed to have a hole, no statement of that kind survives.

\smallskip
\noindent\emph{(iii) A new road to (NT) through local rigidity.} When suitably defined, local maximizers of $\lambda_1$ are themselves rigid: each is a disk up to translation. Combined with a Fock-space concentration--compactness profile decomposition, this local rigidity yields a new proof of (NT) that is independent of isoperimetric inequalities.

\medskip

The first technical pillar of the paper is the rewriting of the eigenvalue equation, after Bargmann transform, as an infinite weighted moment system on the Bargmann--Fock space:
\begin{equation}\label{eq:moment-intro}
\int_U F(z)\,\overline{e_k(z)}\,e^{-\pi|z|^2}\,\mathrm{d}A(z)
\;=\;
\lambda\,\langle F,e_k\rangle_{\mathcal F^2},
\qquad k\ge 0,
\end{equation}
where $F=\mathcal Bf\in\mathcal F^2(\C)$ and $\{e_k\}$ denotes the standard monomial basis. This is a \emph{$\lambda$-dependent weighted quadrature identity}: indeed, classical quadrature-domain theory studies analogous identities for harmonic test functions against Lebesgue measure \cite{Gustafsson,Gustafsson-Shahgholian,Shahgholian,Aharonov-Schiffer-Zalcman,Karp,Cherednichenko}, and here the Gaussian weight and the spectral parameter $\lambda$ make the system nonstandard and place it at the intersection of phase-space analysis, free-boundary problems, and inverse potential theory \cite{Isakov}.

The free-boundary aspect is another technical cornerstone of this manuscript, as the domain $U$
is not prescribed in the inverse problem, so its boundary is an unknown to be
recovered from the moment equations. Near the Daubechies disk, the boundary is
parametrized conformally and the moment system is rewritten as a holomorphic
boundary condition. The resulting domain is then written as an analytic radial graph. The local inverse theorem in Section~3 therefore gives
solvability of a weighted free-boundary problem, at least in the local regime. In the rigidity direction,
the same unknown-domain formulation produces an overdetermined potential
problem whose boundary behavior forces the domain to be a disk.

The five main results below are different manifestations of the mechanisms mentioned above. The inverse theorem proves local solvability of the weighted quadrature system when the eigenfunction is prescribed near the Gaussian. In the opposite direction, the Abreu--D\"orfler rigidity theorem shows that for Hermite data the same moment/free-boundary structure forces the symbol to be a disk. The counterexample theorem confines that mechanism to the simply connected world: for symbols with a hole, the same moment system is solved by an infinite-dimensional family of non-radial domains. The sharpness theorem uses the inverse branch to manufacture nonradial perturbations with exactly quadratic spectral deficit. Finally, the Faber--Krahn part applies the same eigenfunction-as-geometry principle in variational form: local maximizers are forced to be eigenfunction superlevel sets, local rigidity identifies them as disks, and a Fock-space profile decomposition upgrades this local statement to a new proof of the Nicola--Tilli theorem.

\subsection{Main results}\label{ssec:intro-main-results}

The first main result of this work is the inverse construction of localization domains for near Gaussians, at the Daubechies disk. 

\begin{theorem}\label{thm:perturb}
Let $\lambda\in (0,1)$ and $N\in\N$. There exists $\varepsilon_0=\varepsilon_0(\lambda,N)>0$ such that for any
\[
f_0 \;=\; h_0+\sum_{k=1}^N a_k h_k
\qquad\text{with}\qquad
\max_k|a_k|<\varepsilon_0,
\]
there exists a domain $U_\lambda\subset\R^2$, a real-analytic perturbation of a disk, with
\[
\mathcal L_{U_\lambda} f_0 \;=\; \lambda f_0 .
\]
\end{theorem}

This appears to be the first general inverse construction of localization domains beyond the highly symmetric examples. We note that the threshold $\varepsilon_0$ depends on $\lambda$, $N$, and the quantitative inverse bounds for the linearized boundary operator at the disk. 

The construction in Theorem \ref{thm:perturb} is perturbative in nature: we write \eqref{eq:moment-intro} as an identity of entire functions, via the reproducing kernel property of the Bargmann--Fock space. Applying the coefficientwise conjugate differential operator gives a weighted moment system, which we express as a holomorphic boundary condition following Cherednichenko\ \cite{Cherednichenko1973,Cherednichenko1975}. A normalized conformal parametrization turns this condition into an equation on a fixed Wiener space. Its derivative at the disk has an explicit bounded inverse, allowing us to use the Banach-space implicit-function theorem. An elementary ordinary differential equation argument then recovers the original eigenfunction identity. 

The same machinery yields a self-contained recovery of (AD) for regular domains.

\begin{theorem}\label{thm:eig-loc-U}
Let $U\subset\C$ be a bounded simply connected domain satisfying
$U=\operatorname{int}\overline U$. If a Hermite function is an eigenfunction
of $\mathcal L_U$, then $U$ is a disk.
\end{theorem}

We give two proofs of Theorem~\ref{thm:eig-loc-U} (\S\ref{sec:classical-rigidity}). The first, in the spirit of overdetermined free-boundary problems, treats the $C^1$ subcase by transporting the eigenvalue equation into the Bargmann--Fock space, applying a Paley--Wiener argument and an Ehrenpreis--Malgrange-type divisibility lemma to achieve a certain distributional differential free boundary equation, which must admit compactly supported solutions. By translating back to the world of potential theory, we are able to force the corresponding weighted logarithmic potential to be locally radial in $U$. If the boundary is not a circle, a local radial chart then contradicts the Poisson equation satisfied by that potential. 

The second proof eliminates \emph{all} differentiability and perimeter hypotheses on $\partial U$, replacing them with a holomorphic auxiliary function attached to the associated Poisson problem. It has essentially the same initial setting as the first proof: by viewing the problem as a distributional free-boundary equation, we may find a suitable compactly supported solution. We use that solution, however, to construct an auxiliary holomorphic function $F$ on $U$ such that $zF$ has real boundary values. The maximum principle makes $zF$ constant, while its explicit boundary values are a strictly monotone function of $|z|$. Thus $|z|$ is constant on $\partial U$, and $U$ must be a disk.  

We note that both new proofs use simple connectivity in a crucial way, and so does the original argument of \cite{Abreu-Doerfler}: what the potential-theoretic mechanism produces is the vanishing of a logarithmic potential on the \emph{unbounded} component of the complement, which carries no information on a hole. Our next result shows that this is not a coincidence: for multiply connected domains, the inverse problem has an infinite-dimensional solution set, and no rigidity statement of the type of Theorem~\ref{thm:eig-loc-U} can hold.

\begin{theorem}\label{thm:counterexample}
Let $D_0\subset \C$ be a bounded simply connected domain such that $0\in D_0$, such that $\partial D_0$ is a real-analytic Jordan curve, and such that $D_0$ is not a disk centered at the origin. Then there exist $t_0>0$ and a real-analytic increasing family $(D_t)_{0\le t<t_0}$ of bounded simply connected domains, with $\overline{D_s}\subset D_t$ for $0\le s<t<t_0$, such that for every $0<t<t_0$ the shell
\[
\Omega_t:=D_t\setminus\overline{D_0}
\]
is a doubly connected domain with two real-analytic boundary components, is not radial about the origin --- in particular, it is not a union of concentric annuli --- and satisfies
\[
\mathcal L_{\Omega_t}h_0=2\pi t\,h_0,
\qquad 2\pi t\in(0,1).
\]
\end{theorem}

The idea behind the proof is a weighted Hele--Shaw flow argument.  One starts the procedure at a noncircular analytic domain $D_0$ containing the origin, and the mains idea is to flow outwards at a speed proportional to the normal derivative of the Green function of the domain divided by the Gaussian weight. The reason behind it is that, when using Reynolds' transport theorem together with one integration by parts, this yields exactly the value of any harmonic function at the origin, and the moment equations then follow from conservation laws satisfied by the flow, which is the Gaussian-weighted form of Richardson's identity for Hele--Shaw flow \cite{Richardson}. Since the inner boundary $\partial D_0$ may be prescribed arbitrarily among noncircular analytic Jordan curves, the family of counterexamples is infinite-dimensional.

The fourth theorem uses Theorem~\ref{thm:perturb} to settle the question of optimal exponents in the quantitative stability estimate of \cite{Gomez-Guerra-Ramos-Tilli}. For a measurable set $U\subset\R^2$ of finite positive measure, write $\lambda_1(U)$ for the first eigenvalue of $\mathcal L_U$, $f_U$ for an $L^2$-normalized first eigenfunction, $\varphi_{z_0}$ for a Gaussian optimizer centred at $z_0\in\C$, and $\mathcal A(U)$ for the Fraenkel asymmetry of $U$. Corollary~1.4 of \cite{Gomez-Guerra-Ramos-Tilli} provides
\begin{equation}\label{eq:GGRT-bounds-intro}
\min_{z_0\in\C,\ |c|=1}\|f_U-c\,\varphi_{z_0}\|_{L^2(\R)}
\;\le\;
C e^{|U|/2}\!\left(1-\tfrac{\lambda_1(U)}{1-e^{-|U|}}\right)^{\!1/2},
\qquad
\mathcal A(U) \;\le\; K(|U|)\!\left(1-\tfrac{\lambda_1(U)}{1-e^{-|U|}}\right)^{\!1/2}.
\end{equation}

\begin{theorem}\label{thm:GGRT-sharp}
The exponent $1/2$ in both inequalities of \eqref{eq:GGRT-bounds-intro} is optimal: it cannot be replaced by any $\beta>1/2$.
\end{theorem}

The argument is direct. Theorem~\ref{thm:perturb} applied to second-order Hermite perturbations of $h_0$ produces a one-parameter family $U_\varepsilon$ of perturbed disks whose boundary deformation is governed by quadratic harmonics. A first-variation analysis based on the Reynolds transport identity --- and a sharp explicit lower bound on $|U_\varepsilon\triangle B(z_0,r)|$ valid uniformly over $(z_0,r)$ --- shows that these harmonics cannot be removed by translations of the disk. The normalized eigenfunction is also at distance comparable with $|\varepsilon|$ from the coherent-state orbit, whereas the spectral deficit is $O(\varepsilon^2)$; thus both exponents in \eqref{eq:GGRT-bounds-intro} are sharp.

The fifth main result concerns local extremizers of the Faber--Krahn problem in the Gaussian STFT setting and is a key input to the new proof of (NT) below.

\begin{theorem}\label{thm:local-maximizers}
Any local maximizer of $\lambda_{1,\varphi}$ under the constraint of having fixed positive measure is a disk up to translation and null sets.
\end{theorem}

The proof identifies any local maximizer with a superlevel set of the Bargmann--Fock density of its first eigenfunction (via a bathtub argument). It then uses the superlevel-set boundary condition and Gauss--Green to derive an overdetermined identity. By carefully selecting the perturbations in terms of holomorphic functions, a surprising fact appears: if $F$ is the Bargmann transform of the first eigenfunction of the localization operator associated to our local maximizer, then its derivative $F'$ is also an eigenfunction associated with the same eigenvalue. On the other hand, in the specific case in which we are dealing with local optimal, we are able to show that the first eigenvalue is also \emph{simple} for any local maximizer. This yields that $F$ and $F'$ coincide, up to multiplication by constants, forcing hence the domain into a disk. Since the local maximizer is shown to be open with real-analytic boundary as a \emph{consequence} of the variational equations, no boundary regularity is assumed.

Combining Theorem~\ref{thm:local-maximizers} with a Fock-space concentration--compactness profile decomposition gives a new proof of (NT). Existence is obtained by extracting nontrivial Fock profiles and showing, through homogeneity and exact norm decoupling, that every active normalized profile already maximizes the original concentration functional \(M(s):=\sup\{\lambda_1(U):|U|=s\}\); uniqueness up to translation then follows because every global maximizer is in particular a local one.

\begin{theorem}[\cite{Nicola-Tilli}]\label{thm:NT-new}
Among all measurable sets $\Omega\subset\R^2$ of fixed positive finite measure, the global maximizers of $\lambda_{1,\varphi}$ are, up to null sets, exactly the translates of a single disk.
\end{theorem}

We emphasize that this new proof of Theorem \ref{thm:NT-new} does not at all rely on isoperimetric methods. Rather than that, we stress that the main idea here is to provide a \emph{variational} framework for that result, which may be applicable in further contexts. See the discussion at the end of the manuscript. 

\subsection{Methods and central identity}\label{ssec:intro-methods}

The central object of the paper is the moment identity \eqref{eq:moment-intro}. Three groups of techniques act on it.

\smallskip
\noindent\emph{(a) Free-boundary methods on Wiener spaces.}
The moment system \eqref{eq:moment-intro} is reinterpreted as the condition that an explicit boundary function extend holomorphically inside the unknown domain. A conformal parametrization of the boundary turns this condition into a nonlinear operator on one weighted sequence space. Its derivative at the disk is computed explicitly in Fourier modes and has a bounded real-linear inverse. An analytic implicit-function theorem then produces the perturbed domain $U_\lambda$, which is subsequently written as an analytic radial graph. This step adapts Cherednichenko's analytic-density boundary formulation and perturbative conformal method \cite{Cherednichenko1973,Cherednichenko1975}; the fixed Wiener space, the disk linearization, and the analytic parameter dependence are proved here for the Gaussian weight and the spectral parameter.

\smallskip
\noindent\emph{(b) Complex analysis and divisibility.}
The coefficientwise conjugate differential operator $P^\#(\partial_w/\pi)$ converts the eigenvalue equation into the weighted moment identity used in Section~\ref{sec:perturb-proof}; the original equation is recovered by an ordinary differential equation argument with an exponential-type bound. For Hermite eigenfunction rigidity, a Paley--Wiener--Schwartz argument applied to a compactly supported distribution on $\R^2$ encodes the eigenfunction's structure as the vanishing of an entire function on the complex variety $\xi^2+\eta^2=0$. Elementary homogeneous division gives the algebraic factorization, and the Paley--Wiener division theorem supplies the compactly supported quotient needed for the free-boundary argument (Lemma~\ref{lemma:fourier-to-fboundary}).

\smallskip
\noindent\emph{(c) Variational and geometric measure theory.}
Shape derivatives are computed via Reynolds transport for $C^1$ deformation fields. The bathtub principle of Lieb--Loss \cite[Theorem~1.14]{Lieb-Loss} converts the measure-constrained variational problem into a superlevel-set problem. The superlevel-set boundary condition and Gauss--Green yield the identity used in the local analysis, while simplicity of the principal eigenvalue follows from a modulus-ratio argument and the maximum principle. Finally, a profile-decomposition theorem in $\mathcal F^2(\C)$ in the spirit of P.-L.~Lions's concentration--compactness \cite{Lions-Concentration}, adapted to the lack of translation compactness in the Fock space, drives the existence half of the proof of (NT).

The detailed list of external theorems we use --- with explicit indication of which are quoted, which are adapted, and which are reproved as lemmas --- is collected as \S\ref{subsec:main-technical-inputs} of the preliminaries.

\subsection{Relation to the literature}\label{ssec:intro-related}

\noindent\emph{Quadrature domains and inverse potential theory.}
A classical quadrature domain $D\subset\R^2$ is characterized by the fact that integration of harmonic test functions against Lebesgue measure on $D$ reduces to a finite combination of point evaluations and their derivatives; see Gustafsson--Shahgholian \cite{Gustafsson-Shahgholian} for the obstacle-problem characterization, the Gustafsson lectures \cite{Gustafsson}, and \cite{Shahgholian}. Aharonov--Schiffer--Zalcman \cite{Aharonov-Schiffer-Zalcman} and Karp \cite{Karp} address the inverse potential problem in adjacent settings. Cherednichenko\ \cite{Cherednichenko1973,Cherednichenko1975,Cherednichenko} develops the planar boundary formulation and local perturbative method adapted in Section~\ref{sec:perturb-proof}. Our moment system \eqref{eq:moment-intro} is structurally different from a classical quadrature identity in two ways: the test space is the Bargmann--Fock space of entire functions rather than all harmonic functions, and the eigenvalue $\lambda$ enters as an explicit spectral parameter. A second point of contact with that circle of ideas is Richardson's conservation law for Hele--Shaw injection \cite{Richardson} and its variable-permeability extensions \cite{Loutsenko,RossWitt}, of instrumental use in the proof of Theorem~\ref{thm:counterexample}, where they produce non-radial layers whose quadrature node lies in a hole rather than in the domain of integration.  This is exactly the configuration left uncovered by the inverse mean-value theorem of Kuran \cite{Kuran}, and it is what defeats simple-connectivity rigidity; for the topology of quadrature domains see also Lee--Makarov \cite{LeeMakarov}. 

\smallskip
\noindent\emph{Toeplitz inverse-spectral problems on Fock and Bergman spaces.}
Symbol-recovery and inverse-spectral problems for Toeplitz operators have been explored, in various directions, by Berger--Coburn \cite{Berger-Coburn-Symbol,Berger-Coburn,Berger-Coburn-Heat}, Berezin \cite{Berezin}, and Zhu \cite{Zhu}, and more recently by Bastianoni--Cordero--Nicola--Teofanov \cite{Bastianoni-Cordero-Nicola,Bastianoni-Teofanov}, Fulsche--Hagger \cite{Fulsche-Hagger}, and Kisil \cite{Kisil-CrossToeplitz}; see also \cite{abreu2021donoho} for related polyanalytic and large-sieve phenomena. In all of these the symbol typically belongs to a function class on phase space, and the inverse problem is to read off properties of the symbol from the spectrum. By contrast, the inverse problem we solve here is geometric: the symbol is the indicator $\mathbf 1_U$ of an unknown domain, and the spectral datum is a single eigenfunction.

\smallskip
\noindent\emph{Concentration--compactness in time-frequency analysis.}
The profile decomposition we use in \S\ref{sec:global-max} extends and complements the time-frequency adaptations of P.-L.~Lions's method developed by Nicola--Tilli \cite{Nicola-Tilli,Nicola-Tilli-2} and Nicola--Romero--Trapasso \cite{Nicola-Romero-Trapasso}, with the principal subtlety being the lack of translation compactness in $\mathcal F^2(\C)$. Closely related concentration phenomena and profile decompositions in adjacent time-frequency settings appear in \cite{Marceca-Romero-Speckbacher,Samuelsen,Fulsche-Hagger}.

\smallskip
\noindent\emph{Hermite structures in the Bargmann model.}
The role of Gaussian decay and Hermite-type rigidity is also visible in nearby phase-space problems \cite{Goncalves-Oliveira-Steinerberger,Thangavelu}; the inverse-spectral perspective adopted here is, however, of a different character, in that it produces the domain rather than constraining it.

\subsection{Organization}\label{ssec:intro-organization}

Section~\ref{sec:preliminaries} collects the analytic preliminaries on the Bargmann--Fock space, the main technical inputs from the external literature, and the bridge from time--frequency analysis to free-boundary problems via the divisibility lemma. Section~\ref{sec:perturb-proof} is devoted to the proof of Theorem~\ref{thm:perturb}, including the self-contained disk inverse step. Section~\ref{sec:classical-rigidity} contains the two proofs of Theorem~\ref{thm:eig-loc-U}, together with the multiply connected counterexamples of Theorem~\ref{thm:counterexample} in \S\ref{subsec:counterexample}. Section~\ref{sec:GGRT-sharp} uses Theorem~\ref{thm:perturb} to prove Theorem~\ref{thm:GGRT-sharp}. Section~\ref{sec:local-max} is devoted to local maximizers and Theorem~\ref{thm:local-maximizers}. Section~\ref{sec:global-max} establishes the profile decomposition in $\mathcal F^2(\C)$, proves existence by profile saturation, and combines it with Theorem~\ref{thm:local-maximizers} to give the new proof of (NT).

\subsection*{LLM Usage} The main ideas of this manuscript, with the single exception of the counterexample construction of \S\ref{subsec:counterexample} --- which was arrived at through interaction with GPT 5.6-Sol, by analogy with the two 1972 papers of Kuran on the mean-value property of harmonic functions \cite{Kuran} and of Richardson on Hele--Shaw injection \cite{Richardson} ---,
among which the connection between eigenfunctions and geometric constraints through free boundary inverse problems, and the variational nature of the local extremal characterization, are all due exclusively to the author. The wording of many technical arguments, however, was largely aided by several different LLMs, such as GPT 5.4, 5.5 and most lately GPT 5.6-Sol, as well as Claude Opus 4.6, 4
.7, 4.8 and, more recently, Claude Fable 5. The final draft has been thoroughly rewritten and reorganized by the author, and he takes full responsibility for the contents of this manuscript. 
\section*{Acknowledgements}

I am grateful to André Guerra for several discussions that motivated me to investigate
these issues further, and to Alessio Figalli, Luis Daniel Abreu, Mitchell
Taylor, Rupert Frank, Simon Larson, Paolo Tilli and Jaime G\'omez for further helpful comments and discussions. 

This
work was supported by FAPERJ---the Carlos Chagas Filho Foundation for Research
Support of the State of Rio de Janeiro---under Process SEI-260003/020475/2025.
The author also acknowledges support by the Portuguese government through FCT---Fundação
para a Ciência e a Tecnologia, I.P.---project 2023.17881.ICDT, with DOI
identifier 10.54499/2023.17881.ICDT (project SHADE).


\section{Preliminaries}\label{sec:preliminaries}

\subsection{Notation}\label{subsec:notation}

We identify $\C$ with $\R^2$ by writing $z=x+iy$, and use $\mathrm{d}A(z)$
for planar Lebesgue measure. For $a\in\C$ and $R>0$, we write
$D_R(a):=\{z\in\C:|z-a|<R\}$ and $D_R:=D_R(0)$. If $E$ is measurable,
$|E|$ denotes its Lebesgue measure and $\mathbf{1}_E$ its indicator function.
The complex derivatives are
\[
\partial_z:=\frac12(\partial_x-i\partial_y),\qquad
\partial_{\bar z}:=\frac12(\partial_x+i\partial_y),\qquad
\Delta=\partial_x^2+\partial_y^2=4\partial_z\partial_{\bar z}.
\]

For an open set $U\subset\R^n$, $\mathcal D(U):=C_c^\infty(U)$ is the space
of test functions and $\mathcal D'(U)$ its dual, the space of distributions.
We write $\mathcal S(\R^n)$ for the Schwartz space and $\mathcal S'(\R^n)$
for the tempered distributions. The action of a distribution $T$ on a test
function $\psi$ is denoted by $\langle T,\psi\rangle$; in particular,
$\langle\delta_a,\psi\rangle=\psi(a)$. A compactly supported distribution
$T$ has order at most $N$ if, for some $C>0$,
\[
|\langle T,\psi\rangle|
\le C\sum_{|\alpha|\le N}\|\partial^\alpha\psi\|_{L^\infty(\R^n)}
\qquad (\psi\in\mathcal D(\R^n)),
\]
where $\alpha$ ranges over multi-indices. For $f\in\mathcal S(\R^n)$ we use
the Fourier-transform convention
\[
\widehat f(\xi)=\int_{\R^n}f(x)e^{-2\pi i x\cdot\xi}\,\mathrm{d}x,
\]
extended to distributions by duality. Thus, in dimension two,
$\widehat{\Delta T}(\xi)=-4\pi^2|\xi|^2\widehat T(\xi)$.

For Fourier series on the circle we use the normalized angular variable
$\theta\in\T:=\R/\Z$. For a $1$-periodic integrable function $g$, the Fourier
coefficients are normalized by
\[
g_k:=\int_0^1g(\theta)e^{-2\pi i k\theta}\,\mathrm{d}\theta,
\qquad k\in\Z,
\]
and, whenever the Fourier series converges in the relevant function space, we
write
\[
g(\theta)=\sum_{k\in\Z}g_k e^{2\pi i k\theta}.
\]
We also write $\widehat g(k):=g_k$ when referring to the Fourier-series
coefficient of a periodic function; the same hat denotes the Fourier transform
when the argument is a frequency variable in $\R^n$.
With this convention,
\[
\int_0^1e^{2\pi i(k-j)\theta}\,\mathrm{d}\theta=\delta_{kj},
\]
and a real-valued function satisfies $g_{-k}=\overline{g_k}$.

\subsection{Main technical inputs}\label{subsec:main-technical-inputs}

For the convenience of the reader, and to make the boundaries between
quoted material, adapted material, and original arguments transparent, we
collect here the principal external theorems used in this paper, together
with an indication of the form in which each is used.

\begin{itemize}
\item \emph{Cherednichenko's perturbative inverse logarithmic-potential method.}
In Section~\ref{sec:perturb-proof} we adapt the analytic-density boundary
formulation of \cite[Theorem~2]{Cherednichenko1973} and the conformal
perturbation method of \cite{Cherednichenko1975}; see also
\cite{Cherednichenko}. The method is \emph{adapted with proof}: we use the
primitive of the Gaussian-weighted density, check the real-residue
compatibility condition, and invert the linearized boundary operator at
the disk on a fixed Wiener space. The analytic parameter dependence is
proved in this space by the implicit-function theorem. Isakov's work
\cite{Isakov} provides the broader inverse-source context; his existence
theorem is not invoked in this proof.

\item \emph{Bathtub principle.} The Lieb--Loss bathtub principle is used
in Sections~6 and~7 to identify maximizers of $\int_\Omega u_F\,\mathrm{d}A$ under
a measure constraint with superlevel sets of $u_F$. We quote it
\emph{verbatim} from \cite[Theorem~1.14]{Lieb-Loss}; see also Lemma~\ref{lem:basic-properties-global}(i)
below.

\item \emph{Concentration--compactness profile decomposition.} In
Section~7 we use a profile-decomposition statement in $\mathcal F^2(\C)$
in the spirit of P.-L.~Lions's concentration--compactness method and of
its time--frequency adaptations by Nicola and Tilli \cite{Nicola-Tilli,Nicola-Tilli-2}
and by Nicola, Romero, and Trapasso \cite{Nicola-Romero-Trapasso}. The
form we use is \emph{adapted with proof}: the abstract concentration
mechanism is classical, while the precise profile decomposition for
maximizing sequences of $I_F(s)$ in the Fock space, including the
treatment of the lack of translation compactness, is reproved here.

\item \emph{Ehrenpreis--Malgrange divisibility.} The passage from
algebraic divisibility of an entire function on $\C^2$ to divisibility
within a Paley--Wiener class is the classical
Ehrenpreis--Malgrange division theorem
\cite[Ch.~VII]{Hormander-PDE1}. Lemma~\ref{lemma:fourier-to-fboundary}
separates the elementary homogeneous factorization from this
Paley--Wiener division input, and records exactly the compact-support and
regularity conclusions used later.

\item \emph{Richardson's conservation law and weighted Hele--Shaw flow.} The counterexamples of \S\ref{subsec:counterexample} are built from the injection law for Hele--Shaw flow with variable permeability, whose underlying moment conservation goes back to Richardson \cite{Richardson}; see \cite{Loutsenko,RossWitt} for the weighted framework. The conservation law is \emph{reproved} here in the Gaussian-weighted form actually needed, as Proposition~\ref{prop:ce-Richardson}. The short-time existence of the flow itself, recorded as Proposition~\ref{prop:ce-existence}, is \emph{quoted}: in conformal parametrization it is the weighted Polubarinova--Galin equation with real-analytic data, covered by the standard Cauchy--Kowalevsky arguments of \cite{Gustafsson-Vasiliev,Hedenmalm-Shimorin,RossWitt}. For comparison purposes only we also \emph{quote} the inverse mean-value theorem of Kuran \cite{Kuran}.

\item \emph{Paley--Wiener and Paley--Wiener--Schwartz theorems.} The
classical Paley--Wiener theorem (and its Schwartz extension to compactly
supported distributions) is used several times to translate vanishing of
a Fourier transform along complex lines into vanishing of a function on
the support side. We \emph{quote} these results from
\cite{Hormander,Folland} in the standard form.
\end{itemize}

The remainder of Section~2 is devoted to the analytic preliminaries on
the Bargmann--Fock space and the divisibility lemma proper.

\subsection{Properties of the Bargmann-Fock space} We briefly recall the connection between the Gaussian short-time Fourier transform
(STFT) and the Bargmann--Fock space, which will be used systematically in what
follows. Useful references for the material in this subsection are
\cite{Bargmann,Folland,Grochenig,Zhu,Ramos-Optimal-Localization}; see also
\cite{Berezin,Berger-Coburn,Daubechies}.

\medskip

Let $\varphi(x) = 2^{1/4} e^{-\pi x^2}$ be the normalized Gaussian. For $f \in L^2(\mathbb{R})$, its Gaussian STFT is
\begin{equation}\label{eq:STFT}
V_\varphi f(x,\omega)
= \int_{\mathbb{R}} f(t)\, \varphi(x-t)\, e^{-2\pi i t \omega}\,\mathrm{d}t.
\end{equation}
The Bargmann transform
$\mathcal{B} : L^2(\mathbb{R}) \to \mathcal{F}^2(\mathbb{C})$ is defined by
\begin{equation}\label{eq:Bargmann}
(\mathcal{B}f)(z)
= 2^{1/4} \int_{\mathbb{R}} f(t)\, e^{2\pi t z - \pi t^2 - \frac{\pi}{2} z^2}\,\mathrm{d}t,
\qquad z \in \mathbb{C}.
\end{equation}
It is well known that $\mathcal{B}$ is a unitary isomorphism onto the Bargmann--Fock
space $\mathcal{F}^2(\mathbb{C})$, endowed with the norm
\[
\|F\|_{\mathcal{F}^2}^2
= \int_{\mathbb{C}} |F(z)|^2 e^{-\pi |z|^2} \, \mathrm{d}A(z).
\]
A direct computation shows that, for $z = x + i\omega$,
\begin{equation}\label{eq:STFT-B}
V_\varphi f(x,-\omega)
= e^{\pi i x \omega}\, (\mathcal{B}f)(z)\, e^{-\frac{\pi}{2}|z|^2}.
\end{equation}
When using complex notation, we therefore write
$V_\varphi f(z):=V_\varphi f(x,-\omega)$ for $z=x+i\omega$.
In particular, setting $F = \mathcal{B}f$, one obtains
\begin{equation}\label{eq:energy}
\int_{\mathbb{R}^2} |V_\varphi f(x,\omega)|^2 \,\mathrm{d}x\, \mathrm{d}\omega
= \int_{\mathbb{C}} |F(z)|^2 e^{-\pi |z|^2} \, \mathrm{d}A(z)
= \|F\|_{\mathcal{F}^2}^2.
\end{equation}
Set
\begin{equation}\label{eq:uF}
u_F(z) := |F(z)|^2 e^{-\pi |z|^2}, \qquad z \in \mathbb{C}.
\end{equation}

\medskip

\noindent
\subsubsection*{Weyl operators and invariance.}
For $a \in \mathbb{C}$, define the operator
\begin{equation}\label{eq:Ta}
(T_a F)(z) := e^{\pi \overline{a} z - \frac{\pi}{2}|a|^2} F(z - a).
\end{equation}
Then $T_a$ is unitary on $\mathcal{F}^2(\mathbb{C})$. Indeed,
\begin{align*}
\|T_aF\|_{\mathcal F^2}^2
&=
\int_{\C} e^{2\pi \Re(\overline{a}z)-\pi |a|^2}\,|F(z-a)|^2 e^{-\pi |z|^2}\,\mathrm{d}A(z) \\
&=
\int_{\C} |F(z-a)|^2 e^{-\pi |z-a|^2}\,\mathrm{d}A(z)
=
\|F\|_{\mathcal F^2}^2,
\end{align*}
Moreover,
\begin{equation}\label{eq:u-translation-global}
u_{T_a F}(z) = u_F(z - a).
\end{equation}
In particular, for \(s\ge 0\), define
\[
I_F(s):=\sup_{|\Omega|=s}\int_\Omega u_F\,\mathrm{d}A.
\]
Thus $I_F(0)=0$. By translation invariance of Lebesgue measure on \(\C\),
these functionals satisfy
\begin{equation}\label{eq:I-translation-global}
I_{T_a F}(s) = I_F(s), \qquad a \in \mathbb{C},
\end{equation}
reflecting the invariance of the Gaussian STFT under time--frequency shifts;
see \cite{Daubechies,Grochenig,Goncalves-Oliveira-Steinerberger}.

\medskip

\noindent
\subsubsection*{Connection with localization.}
If $F = \mathcal{B}f$, then for every measurable $\Omega \subset \mathbb{C}$,
\[
\int_\Omega u_F \, \mathrm{d}A
= \int_{\widetilde{\Omega}} |V_\varphi f(x,\omega)|^2 \,\mathrm{d}x\, \mathrm{d}\omega,
\]
where $\widetilde{\Omega} = \{(x,\omega) : (x,-\omega) \in \Omega\}$. Relabeling
$\widetilde\Omega$ as $\Omega$, the corresponding Bargmann--Fock localization
operator is
\begin{equation}\label{eq:fock-localization-operator}
(\mathcal L_\Omega F)(w)
:=
\int_\Omega F(z)e^{\pi w\overline z}e^{-\pi|z|^2}\,\mathrm{d}A(z).
\end{equation}
For bounded measurable $\Omega$, this operator is self-adjoint and compact.
Thus, concentration properties of the STFT are naturally encoded by $u_F$,
which is the main object of interest in the Bargmann--Fock setting, in terms of the
quantities $I_F(s)$.

\medskip

We next collect a few elementary facts that will be used repeatedly.

\begin{lemma}\label{lem:Fock-basic}
Let $F,G \in \mathcal{F}^2(\mathbb{C})$. Then:
\begin{enumerate}
\item[(i)] The space $\mathcal{F}^2(\mathbb{C})$ is a reproducing kernel Hilbert space with kernel
\[
K_w(z) = e^{\pi z \overline{w}}, \qquad z,w \in \mathbb{C}.
\]

\item[(ii)] (Sharp pointwise bound) For every $z \in \mathbb{C}$,
\[
u_F(z) \le \|F\|_{\mathcal{F}^2}^2,
\]
and equality holds at some point $w \in \mathbb{C}$ if and only if $F$ is
proportional to the reproducing kernel $K_w$.

\item[(iii)] $u_F(z) \to 0$ as $|z| \to \infty$.
\end{enumerate}
\end{lemma}

\begin{proof}
(i) For \(w\in \C\), the function \(K_w(z)=e^{\pi z\overline w}\) is entire and
satisfies, after completing the square via
\(2\Re(z\overline w)-|z|^2=-|z-w|^2+|w|^2\),
\[
\|K_w\|_{\mathcal F^2}^2
=
\int_{\C} e^{2\pi \Re(z\overline w)-\pi |z|^2}\,\mathrm{d}A(z)
=
e^{\pi |w|^2}\int_{\C}e^{-\pi |z-w|^2}\,\mathrm{d}A(z)
=
e^{\pi |w|^2},
\]
since the Gaussian \(e^{-\pi |\cdot|^2}\) integrates to \(1\) over \(\C\) (with
the Lebesgue measure \(\mathrm{d}A\) on \(\C\simeq \mathbb R^2\)). In particular
\(K_w\in \mathcal F^2(\C)\). Now, expanding \(F\in \mathcal F^2(\C)\) along the
orthonormal basis \(\{e_k\}_{k\geq 0}\) of Lemma \ref{lem:monomials} and using
the Taylor expansion
\(K_w(z)=\sum_{k\geq 0}\frac{\pi^k\overline w^k}{k!}z^k=\sum_{k\geq 0}\overline{e_k(w)}\,e_k(z)\),
one obtains the reproducing identity
\[
\langle F,K_w\rangle_{\mathcal F^2}
=
\sum_{k\geq 0} a_k\,e_k(w)
=
F(w),
\qquad F(z)=\sum_{k\geq 0}a_k e_k(z).
\]
Point evaluation is therefore continuous, and \(F\) is entire as a uniform
limit on compact sets of polynomials. This is the standard description of the
Bargmann--Fock space, see \cite{Bargmann,Folland,Grochenig,Zhu}.

(ii) Applying Cauchy--Schwarz to the reproducing identity from (i),
\begin{equation}\label{eq:pointwise}
|F(z)|
=
|\langle F,K_z\rangle_{\mathcal F^2}|
\le
\|F\|_{\mathcal F^2}\,\|K_z\|_{\mathcal F^2}
=
\|F\|_{\mathcal F^2}\,e^{\frac{\pi}{2}|z|^2},
\end{equation}
which immediately gives \(u_F(z)=|F(z)|^2 e^{-\pi |z|^2}\le \|F\|_{\mathcal F^2}^2\).
Equality at a point \(z=w\) corresponds to equality in Cauchy--Schwarz, that is,
\(F=c\,K_w\) for some \(c\in \C\); conversely, for \(F=c K_w\) one computes
\(u_F(w)=|c|^2 e^{2\pi |w|^2}e^{-\pi |w|^2}=|c|^2 e^{\pi |w|^2}=|c|^2\|K_w\|_{\mathcal F^2}^2
=\|F\|_{\mathcal F^2}^2\), proving the rigidity statement. The sharpness of this
pointwise bound goes back to \cite{Lieb}.

(iii) Fix \(\varepsilon>0\), and choose a polynomial \(P\) such that
\[
\|F-P\|_{\mathcal F^2}<\varepsilon
\]
using the density of polynomials in \(\mathcal F^2(\C)\). By the pointwise bound from (ii),
\[
|F(z)-P(z)|e^{-\frac{\pi}{2}|z|^2}\le \varepsilon
\qquad\text{for every }z\in \C.
\]
Since \(P\) is a polynomial,
\[
P(z)e^{-\frac{\pi}{2}|z|^2}\to 0
\qquad\text{as }|z|\to\infty.
\]
Hence
\[
\limsup_{|z|\to\infty}|F(z)|e^{-\frac{\pi}{2}|z|^2}\le \varepsilon.
\]
As \(\varepsilon>0\) is arbitrary, we conclude that
\[
|F(z)|e^{-\frac{\pi}{2}|z|^2}\to 0,
\]
that is, \(u_F(z)\to 0\) as \(|z|\to\infty\).
\end{proof}

\medskip

We now note the following crucial fact about the Bargmann transform: it
transforms the basis of normalized Hermite functions \(\{h_k\}_{k\geq 0}\) on
\(L^2(\R)\) (see \cite{Thangavelu}) onto normalized monomials, yielding a
natural orthonormal basis on $\mathcal{F}^2(\C).$ More precisely, one has
\(\mathcal B h_k=e_k\) for every \(k\ge 0\), see \cite{Bargmann,Folland,Grochenig}.

\begin{lemma}\label{lem:monomials}
The functions
\[
e_k(z) := b_k z^k, \qquad k \ge 0,
\]
form an orthonormal basis of $\mathcal{F}^2(\mathbb{C})$, where
\[
b_k = \left( \frac{\pi^k}{k!} \right)^{1/2}.
\]
In particular,
\[
\langle e_k, e_m \rangle_{\mathcal{F}^2} = \delta_{k,m}.
\]
\end{lemma}

\begin{proof}
Writing \(z=re^{i\theta}\) and using \(\mathrm{d}A(z)=r\,\mathrm{d}r\,\mathrm{d}\theta\) on \(\C\),
a direct computation in polar coordinates gives
\[
\int_{\mathbb{C}} |z|^{2k} e^{-\pi |z|^2}\,\mathrm{d}A(z)
= \int_0^{2\pi}\!\!\int_0^\infty r^{2k+1} e^{-\pi r^2}\,\mathrm{d}r\,\mathrm{d}\theta
= 2\pi \int_0^\infty r^{2k+1} e^{-\pi r^2}\,\mathrm{d}r
= \frac{k!}{\pi^k},
\]
where the last identity follows from the substitution \(s=\pi r^2\), giving
\(2\pi\int_0^\infty r^{2k+1}e^{-\pi r^2}\,\mathrm{d}r=\pi^{-k}\int_0^\infty s^k e^{-s}\,\mathrm{d}s
=\pi^{-k}\Gamma(k+1)\). Thus
\(\|z^k\|_{\mathcal{F}^2}^2 = \frac{k!}{\pi^k}\), which yields the normalization.
For orthogonality, again writing \(z=re^{i\theta}\),
\[
\langle z^k,z^m\rangle_{\mathcal F^2}
=
\int_0^\infty r^{k+m+1}e^{-\pi r^2}\,\mathrm{d}r\int_0^{2\pi}e^{i(k-m)\theta}\,\mathrm{d}\theta
=
0
\qquad\text{for }k\neq m.
\]
For completeness, since \(\mathcal B\colon L^2(\R)\to \mathcal F^2(\C)\) is a
unitary isomorphism and \(\mathcal Bh_k=e_k\) for every \(k\ge 0\), where
\(\{h_k\}_{k\ge 0}\) is the orthonormal basis of Hermite functions on
\(L^2(\R)\) (see \cite{Thangavelu}), the family \(\{e_k\}_{k\ge 0}\) is the
image of an orthonormal basis under a unitary, hence itself an orthonormal
basis of \(\mathcal F^2(\C)\). Equivalently, polynomials are dense in
\(\mathcal F^2(\C)\); see \cite{Bargmann,Zhu}.
\end{proof}

\medskip

We proceed with a few elementary facts.

\begin{lemma}\label{lem:basic-properties-global}
Let \(F,G\in \mathcal F^2(\C)\) and \(s\ge0\). Then:
\begin{enumerate}
    \item[(i)] \(I_F(s)\) is attained by a superlevel set of \(u_F\), and
    \[
    0\le I_F(s)\le \|F\|_{\mathcal F^2}^2.
    \]
    \item[(ii)] For every \(c\in \C\),
    \[
    I_{cF}(s)=|c|^2 I_F(s).
    \]
    \item[(iii)] One has the Lipschitz bound
    \[
    |I_F(s)-I_G(s)|
    \le
    \|u_F-u_G\|_{L^1(\C)}
    \le
    \big(\|F\|_{\mathcal F^2}+\|G\|_{\mathcal F^2}\big)\,
    \|F-G\|_{\mathcal F^2}.
    \]
\end{enumerate}
\end{lemma}

\begin{proof}
The identity
\[
\int_\C u_F\,\mathrm{d}A=\int_\C |F(z)|^2e^{-\pi |z|^2}\,\mathrm{d}A(z)=\|F\|_{\mathcal F^2}^2
\]
is immediate from the definition of the Fock norm. By
Lemma~\ref{lem:Fock-basic}(ii)--(iii), \(u_F\) is bounded, continuous, and
satisfies \(u_F(z)\to 0\) as \(|z|\to\infty\); thus
\(u_F\in C_0(\C)\cap L^1(\C)\).

The existence of a maximizing set for \(I_F(s)\), of the form
\(\{u_F>t\}\) (possibly together with part of \(\{u_F=t\}\)) with \(t\geq 0\)
chosen so that the resulting set has measure \(s\), is the classical bathtub
principle \cite[Theorem~1.14]{Lieb-Loss}, which applies since
\(u_F\in C_0(\C)\cap L^1(\C)\). The bound
\[
0\le I_F(s)\le \int_\C u_F\,\mathrm{d}A=\|F\|_{\mathcal F^2}^2
\]
is immediate. The scaling identity in (ii) is obvious.

For (iii), for every measurable \(\Omega\subset \C\) with \(|\Omega|=s\),
\[
\left|\int_\Omega u_F\,\mathrm{d}A-\int_\Omega u_G\,\mathrm{d}A\right|
\le \int_\C |u_F-u_G|\,\mathrm{d}A.
\]
Taking suprema over \(|\Omega|=s\) gives
\[
|I_F(s)-I_G(s)|\le \|u_F-u_G\|_{L^1(\C)}.
\]
Finally,
\begin{align*}
\|u_F-u_G\|_{L^1(\C)}
&=
\int_\C \big||F|^2-|G|^2\big|\,e^{-\pi |z|^2}\,\mathrm{d}A \\
&\le
\int_\C |F-G|(|F|+|G|)e^{-\pi |z|^2}\,\mathrm{d}A \\
&\le
\|F-G\|_{\mathcal F^2}\big(\|F\|_{\mathcal F^2}+\|G\|_{\mathcal F^2}\big)
\end{align*}
by Cauchy--Schwarz.
\end{proof}

\subsection{A Paley--Wiener-type lemma}\label{sec:fourier-to-fboundary}

We first record the precise Paley--Wiener--Schwartz form used below and
later in Section~\ref{sec:classical-rigidity}.

\begin{theorem}[Paley--Wiener--Schwartz, planar form]\label{thm:PW-planar}
Let $\omega\in\mathcal D'(\R^2)$ be a compactly supported distribution of
order $\le N$ with $\operatorname{supp}\omega\subset \overline{B(0,R)}$.
Then its Fourier transform $\widehat\omega$ extends to an entire function
on $\C^2$ satisfying the Paley--Wiener bound
\begin{equation}\label{eq:PW-bound-planar}
|\widehat\omega(\zeta)|\le C\,(1+|\zeta|)^N
 e^{2\pi R|\Im\zeta|},
\qquad \zeta\in\C^2,
\end{equation}
for some $C>0$. Conversely, any entire function on $\C^2$ obeying such a
bound is the Fourier--Laplace transform of a unique compactly supported
distribution of order $\le N$ supported in $\overline{B(0,R)}$. If, in
addition, its restriction to $\R^2$ belongs to $L^2(\R^2)$, then the
corresponding distribution belongs to $L^2(\R^2)$.
\end{theorem}

This is the classical Paley--Wiener--Schwartz theorem in the form stated in
\cite[Theorem~7.3.1]{Hormander-PDE1}; see also \cite[\S 9.3]{Folland}. The
$L^2$ refinement follows from Plancherel and uniqueness of Fourier inversion.
For compactly supported finite Radon measures the estimate
\eqref{eq:PW-bound-planar} holds with $N=0$ on the real locus, with
$|\widehat\omega|\le \|\omega\|_{\mathrm{TV}}$.

We now state a Fourier-analytic divisibility statement, close in spirit to
Berenstein's work connecting inverse spectral questions with the Pompeiu
problem \cite{Berenstein-Pompeiu}.

In order to state the lemma cleanly, we make precise the class of
distributions with which we work. We say that a compactly supported
distribution $\omega$ is a \emph{singular-continuous free measure} if it
defines a (finite, signed) Radon measure on $\R^2$ whose Lebesgue
decomposition with respect to planar Lebesgue measure $\mathrm{d}x\,
\mathrm{d}y$ takes the form
\[
\omega = \omega_{\mathrm{ac}} + \omega_{\mathrm{pp}}
       = f\,\mathrm{d}x\,\mathrm{d}y + \sum_{j=1}^N c_j\,\delta_{p_j},
\]
with \(f\in L^\infty(\R^2)\) compactly supported, \(c_j\in \R\setminus\{0\}\), and
$p_1,\ldots,p_N\in \R^2$ pairwise distinct. In other words, the continuous
singular part $\omega_{\mathrm{cs}}$ in the Lebesgue--Radon--Nikodym
decomposition $\omega=\omega_{\mathrm{ac}}+\omega_{\mathrm{cs}}+
\omega_{\mathrm{pp}}$ is required to vanish identically, and the pure
point part is finite. The terminology \emph{free measure} is suggested by
the role of $\omega$ as the right-hand side of a free boundary problem
arising in Section~\ref{sec:classical-rigidity}.

\begin{lemma}\label{lemma:fourier-to-fboundary}
Let $\omega$ be a compactly supported singular-continuous free measure such
that
\[
\widehat{\omega}(w,\pm iw)=0,\qquad \forall w\in\C.
\]
Then there exist a finite set $P\subset \R^2$ and a compactly supported
function $u\in L^2(\R^2)$ such that
\[
\Delta u=\omega
\]
in the sense of distributions, and
\[
u\in C(\R^2\setminus P).
\]
\end{lemma}

\begin{proof}
By the Paley--Wiener--Schwartz theorem
\cite[Theorem~7.3.1]{Hormander-PDE1}, the Fourier
transform $\widehat\omega(\xi,\eta)$ extends to an entire function on $\C^2$
of exponential type, satisfying the bound
\[
|\widehat\omega(\zeta)|\le C(1+|\zeta|)^{N}e^{2\pi R|\Im\zeta|}
\]
for some $C,R>0$ and some integer $N\ge 0$ depending on the order of
$\omega$, where $\zeta=(\xi,\eta)$ and $R$ is the radius of a disk containing
$\operatorname{supp}\omega$. In particular, $\widehat\omega$ admits an
everywhere convergent power series expansion
\[
\widehat{\omega}(\xi,\eta)=\sum_{k,l\ge 0} a_{k,l}\,\xi^k\eta^l,
\qquad (\xi,\eta)\in\C^2.
\]
Using the vanishing hypothesis on the complex lines $\eta=\pm i\xi$, we may
group terms by total homogeneity: setting $\eta=\pm iw$, $\xi=w$, and
collecting powers of $w$,
\[
0=\widehat{\omega}(w,\pm iw)
   =\sum_{n\ge 0}\Bigl(\sum_{k=0}^n a_{k,n-k}(\pm i)^{\,n-k}\Bigr)w^n,
\qquad \forall w\in\C.
\]
Since this holds for all $w\in\C$, comparing coefficients yields
\[
\sum_{k=0}^n a_{k,n-k}(\pm i)^{\,n-k}=0,\qquad \forall n\ge 0.
\]
For each $n\ge 0$, define the auxiliary one-variable polynomial
\[
p_n(t):=\sum_{k=0}^n a_{k,n-k}\, t^{\,n-k}.
\]
The two scalar equations above precisely say that $p_n(\pm i)=0$, so
$(t-i)(t+i)=t^2+1$ divides $p_n(t)$. Hence we may write
\[
p_n(t)=(t^2+1)\,q_n(t),
\]
with $q_n$ a polynomial of degree at most $n-2$ (and $q_n\equiv 0$ when
$n<2$).

We now reassemble the homogeneous pieces. Let
\[
H_n(\xi,\eta):=\sum_{k=0}^n a_{k,n-k}\,\xi^k\eta^{\,n-k},
\qquad n\ge 0,
\]
so that $\widehat\omega=\sum_{n\ge 0}H_n$ is the decomposition of
$\widehat\omega$ into homogeneous components (of degree $n$). The argument
above shows that each $H_n$ vanishes on the two complex lines
$\eta=\pm i\xi$, hence is divisible, as a homogeneous polynomial in two
variables, by $\xi^2+\eta^2$. Concretely,
\[
H_n(\xi,\eta)=(\xi^2+\eta^2)\,G_{n-2}(\xi,\eta)
\]
for a (uniquely determined) homogeneous polynomial $G_{n-2}$ of degree
$n-2$, with $G_{n-2}\equiv 0$ for $n<2$. Summing the homogeneous pieces, the
formal series
\[
F(\xi,\eta):=\sum_{n\ge 2} G_{n-2}(\xi,\eta)
\]
converges absolutely on every polydisc on which $\widehat\omega$ does, and
defines an entire function on $\C^2$ satisfying
\[
\widehat{\omega}(\xi,\eta)=(\xi^2+\eta^2)\,F(\xi,\eta).
\]
Since $\widehat\omega$ has a zero of order at least two at the origin
(coming from the vanishing of $H_0$ and $H_1$, which is forced by
$p_0(\pm i)=p_1(\pm i)=0$), this factorization is consistent at $0$ and $F$
is entire there as well.

We next quantify the size of $F$ on $\R^2$. Since $\omega$ is a finite
(signed) Radon measure with compact support, its Fourier transform is in
fact bounded on $\R^2$:
\[
|\widehat{\omega}(\xi,\eta)|\le \|\omega\|_{\mathrm{TV}},
\qquad (\xi,\eta)\in\R^2,
\]
where $\|\omega\|_{\mathrm{TV}}$ is the total variation; equivalently,
the Paley--Wiener bound holds with $N=0$. On $\R^2$, $\xi^2+\eta^2=
|\xi|^2+|\eta|^2$, so for $|\xi|^2+|\eta|^2\ge 1$ we have $\xi^2+\eta^2
\ge \tfrac{1}{2}(1+|\xi|^2+|\eta|^2)$, whence
\[
|F(\xi,\eta)|=\frac{|\widehat\omega(\xi,\eta)|}{\xi^2+\eta^2}
\le \frac{2\|\omega\|_{\mathrm{TV}}}{1+|\xi|^2+|\eta|^2},
\qquad |\xi|^2+|\eta|^2\ge 1.
\]
Near the origin, $F$ is entire (as established above) and hence bounded;
absorbing this region into the constants we obtain
\[
|F(\xi,\eta)|\le \frac{C''}{1+|\xi|^2+|\eta|^2},
\qquad (\xi,\eta)\in\R^2,
\]
for some $C''>0$. Therefore $F|_{\R^2}\in L^2(\R^2)\cap L^1(\R^2)$.

We now invoke the following Paley--Wiener division theorem for
constant-coefficient operators \cite[Ch.~VII]{Hormander-PDE1}. If \(P\) is a
polynomial and \(\Phi\) is an entire function on \(\C^n\) satisfying a
Paley--Wiener--Schwartz estimate
\[
|\Phi(\zeta)|\le C(1+|\zeta|)^N e^{2\pi H_K(\Im\zeta)}
\]
for a compact convex set \(K\subset\R^n\), and if
\(\Phi=P\Psi\) with \(\Psi\in\mathcal O(\C^n)\), then \(\Psi\) also satisfies
a Paley--Wiener--Schwartz estimate, with possibly larger polynomial order and
with support function \(H_{K'}\) for a compact set \(K'\) depending only on
\(K\) and \(P\). In particular, the quotient is the Fourier--Laplace transform
of a compactly supported distribution. Since the homogeneous argument above
proves \(\widehat\omega=(\xi^2+\eta^2)F\), this theorem gives constants
\(C',R',N'\) for which
\[
|F(\zeta)|
\le C'(1+|\zeta|)^{N'}e^{2\pi R'|\Im\zeta|},
\qquad \zeta\in\C^2.
\]
The converse Paley--Wiener--Schwartz theorem and the \(L^2\) estimate on the
real locus therefore provide a compactly supported \(u\in L^2(\R^2)\) with
\(\widehat u=F\).

Computing distributional Laplacians on the Fourier side, with the convention
\[
\widehat g(\xi,\eta)
=\int_{\R^2}g(x,y)e^{-2\pi i(x\xi+y\eta)}\,\mathrm{d}x\,\mathrm{d}y,
\]
we have
\[
\begin{aligned}
\widehat{\Delta u}(\xi,\eta)
&= -4\pi^2(\xi^2+\eta^2)\,\widehat{u}(\xi,\eta) \\
&= -4\pi^2(\xi^2+\eta^2)\,F(\xi,\eta)
 = -4\pi^2\,\widehat\omega(\xi,\eta);
\end{aligned}
\]
with any other standard normalization, the same identity holds up to a
nonzero real multiplicative factor. In all cases, $\Delta u$ is a constant
nonzero real multiple of $\omega$. Replacing $u$ by $\lambda u$ for the
appropriate (nonzero, real) $\lambda$ -- this is the harmless rescaling
referred to above -- we may assume that
\[
\Delta u = \omega
\]
in the sense of distributions. The rescaling preserves all the previously
established properties of $u$: compact support, membership in $L^2(\R^2)$,
and the Fourier identity $\widehat u = \lambda F$.

To pin down $u$ pointwise, observe that since $\omega$ is a compactly
supported Radon measure on $\R^2$, the convolution
\[
v(x,y) := \frac{1}{2\pi}(\log|\cdot|*\omega)(x,y)
\]
is a well-defined locally integrable function on $\R^2$ which solves
$\Delta v = \omega$ in the sense of distributions; here $G(z)=\log|z|$ is
(up to the standard normalization $\tfrac{1}{2\pi}$) the fundamental solution
of the Laplacian in $\R^2$, in the sense that
$\Delta\bigl(\tfrac{1}{2\pi}\log|\cdot|\bigr)=\delta_0$, see, e.g.,
\cite{Hormander} for the underlying distribution theory and
\cite{Aharonov-Schiffer-Zalcman} for related representations in the planar
quadrature-domain context.
The vanishing hypothesis at the origin gives
\(\widehat\omega(0,0)=\omega(\R^2)=0\). Consequently the logarithmic
potential satisfies \(v(x)=O(|x|^{-1})\) as \(|x|\to\infty\). The difference
\(u-v\) is a harmonic distribution on \(\R^2\), hence a smooth harmonic
function by Weyl's lemma \cite{Hormander-PDE1}. Outside a sufficiently large disk, \(u=0\) by
compact support, so \(u-v=-v\to0\) at infinity. Liouville's theorem for
harmonic functions therefore gives \(u-v\equiv0\). Thus
\[
u(x,y)=\frac{1}{2\pi}\int_{\R^2}\log|(x,y)-(x',y')|\,\mathrm{d}\omega(x',y'),
\]
the logarithmic potential of $\omega$.

We finally analyze the regularity of $u$. Decomposing
$\omega = f\,\mathrm{d}x\,\mathrm{d}y+\sum_{j=1}^N c_j\,\delta_{p_j}$ as in
the definition of singular-continuous free measure, we obtain
correspondingly
\[
u(x,y)=\frac{1}{2\pi}\int_{\R^2}\log|(x,y)-(x',y')|\,f(x',y')\,
\mathrm{d}x'\mathrm{d}y'
+\frac{1}{2\pi}\sum_{j=1}^N c_j\,\log|(x,y)-p_j|.
\]
The first summand is the logarithmic potential of the bounded, compactly
supported density \(f\in L^\infty(\R^2)\) and is therefore not merely
continuous but in fact $C^{1,\alpha}(\R^2)$ for every $\alpha\in(0,1)$
(by the standard Calder\'on--Zygmund / Schauder estimates applied to
$\Delta=\partial_x^2+\partial_y^2$ with right-hand side in $L^\infty$
of compact support, see, e.g., the Newtonian-potential discussion in
\cite{Hormander}).
The continuity used here is in fact elementary: by dominated convergence,
$z\mapsto \int \log|z-w|\,f(w)\,\mathrm{d}A(w)$ is continuous since
$\log|\cdot-w|$ is locally integrable in $w$ uniformly on compacta in $z$,
and $f$ is bounded with compact support. It is here that the absence of a continuous
singular part in $\omega$ is essential: a generic continuous singular
measure (e.g.\ the Hausdorff measure on a Cantor-type set of positive
$1$-capacity) gives rise to a logarithmic potential which need not be
continuous, only quasi-continuous, and the conclusion of the lemma would
fail in that generality. Each term $\log|(x,y)-p_j|$ is continuous on
$\R^2\setminus\{p_j\}$, with a logarithmic singularity at $p_j$. Setting
$P:=\{p_1,\ldots,p_N\}$, we conclude that $u\in C(\R^2\setminus P)$, as
claimed. The fact that $u$ is compactly supported has already been
established in the construction above and is consistent with the
representation as a logarithmic potential, which differs from a compactly
supported function only by a globally harmonic correction that, as
explained, must vanish identically.
\end{proof}


\section{Proof of Theorem \ref{thm:perturb}}\label{sec:perturb-proof}

\subsection{The disk inverse step}\label{subsec:disk-inverse-step}

Can we keep the eigenvalue fixed by deforming the disk when the Gaussian
is replaced by a nearby Hermite combination? In the Bargmann--Fock
representation the prescribed function is a polynomial
\[
P_b(z):=1+\sum_{k=1}^N b_kz^k,
\qquad b_k=a_k\left(\frac{\pi^k}{k!}\right)^{1/2},
\]
since the Bargmann transform sends $h_k$ to the normalized monomial $e_k$.
The localization kernel in \eqref{eq:fock-localization-operator} makes the
condition on the unknown domain $U$ an identity of entire functions:
\begin{equation}\label{eq:eigenfunction-general-new}
\int_U P_b(z)e^{-\pi|z|^2}e^{\pi\overline z w}\,\mathrm{d}A(z)
=\lambda P_b(w),\qquad w\in\C.
\end{equation}
We first note that for $b=0$ the angular average of $e^{\pi\bar z w}$
on each centered circle is $1$. Thus on the disk $D_{r_\lambda}$ the
left-hand side is $1-e^{-\pi r_\lambda^2}=\lambda$. This fixes the radius
$r_\lambda>0$ of
the Daubechies disk at the prescribed $\lambda\in(0,1)$.

We now seek a real density, since the boundary construction will use the
reality of its total mass. The density in \eqref{eq:eigenfunction-general-new}
is complex in general, so we differentiate the kernel to insert the factor
$\overline{P_b(z)}$. Write
\begin{equation}\label{eq:Phash-def}
P_b^\#(z):=\overline{P_b(\overline z)}
=1+\sum_{k=1}^N\overline{b_k}z^k,
\end{equation}
and define the differential operator
\begin{equation}\label{eq:Phash-operator}
P_b^\#(\partial_w/\pi):=
1+\sum_{k=1}^N\overline{b_k}(\partial_w/\pi)^k,
\end{equation}
where $1$ is the identity and $\partial_w$ is the holomorphic derivative.
The kernel satisfies
\begin{equation}\label{eq:Phash-on-planewaves}
P_b^\#(\partial_w/\pi)e^{\pi\bar z w}
= P_b^\#(\bar z)e^{\pi\bar z w}
= \overline{P_b(z)}e^{\pi\bar z w}.
\end{equation}
We note that this differentiation commutes with the integral in
\eqref{eq:eigenfunction-general-new}: for bounded $U$, the $k$th scaled
derivative of the kernel is bounded by $M^k e^{\pi M|w|}$, where
$M=\sup_{z\in U}|z|$, uniformly on compact sets in $w$. With
\begin{equation}\label{eq:Q-def}
Q_b(w):=P_b^\#(\partial_w/\pi)P_b(w),
\end{equation}
the resulting equation has density $|P_b(z)|^2 e^{-\pi|z|^2}$ and
right-hand side $\lambda Q_b(w)$.

We now explain why it is enough to solve this transformed equation. If it
holds, the difference between the two sides of
\eqref{eq:eigenfunction-general-new} satisfies the homogeneous equation
\begin{equation}\label{eq:almost-reduced-new}
P_b^\#(\partial_w/\pi)\bigl(\mathcal L_U P_b-\lambda P_b\bigr)=0.
\end{equation}
We remark that the residual grows at most like a polynomial times
$e^{\pi M|w|}$, since the integral is over $|z|\le M$ and $P_b$ is a
polynomial. If $P_b^\#$ has no zeros in the closed disk of radius $M$,
the characteristic exponents, which
are $\pi$ times its roots, all have modulus greater than $\pi M$. The
growth bound then forces the residual to vanish, as shown in
Lemma~\ref{lem:step3-rigidity} below. For domains close to
$D_{r_\lambda}$ the radius $M$ stays bounded, while $P_b^\#$ is close to $1$
on any fixed disk as $b\to0$. It remains to construct a domain for the
real-density equation.

To control the boundary as $b$ varies, we describe it by a radial graph.
In the normalized angular variable $\theta\in\T$ of
Section~\ref{subsec:notation}, write
\[
U_\varphi:=\{re^{2\pi i\theta}:\theta\in\T,\ 0\le r<r_\lambda+\varphi(\theta)\},
\qquad \|\varphi\|_{C^0}<r_\lambda/2.
\]
For analytic control of this graph, we use the real subspace
$\mathcal A_\tau^\R$ of the complex Wiener algebra $\mathcal A_\tau$ of
width $\tau>0$. Its real elements have Fourier series
\begin{equation}\label{eq:Atau-def}
\varphi(\theta)=\sum_{k\in\Z}\varphi_k e^{2\pi i k\theta},
\qquad \varphi_{-k}=\overline{\varphi_k},
\end{equation}
and the norm in either space is
\begin{equation}\label{eq:Atau-norm}
\|\varphi\|_{\mathcal A_\tau}:=\sum_{k\in\Z}|\varphi_k|e^{\tau |k|}<\infty.
\end{equation}
We note that the weight is submultiplicative, so multiplication is bounded
in this norm. The weight also gives a holomorphic extension to the strip
$|\Im\theta|<\tau/(2\pi)$ and continuous embeddings into $C^m(\T)$ for
all $m\ge0$, hence into $C^{1,\alpha}(\T)$ for $0<\alpha<1$; see, for
background, \cite{Folland,Grochenig,Krantz-Parks-Primer}. In particular,
a small real element gives a real-analytic Jordan boundary, with size
controlled in $C^2$.

\begin{proposition}\label{prop:disk-inverse-tailored}
Let $N\ge1$, $0<\lambda<1$, and let $r_\lambda$ be as above. There are
$\tau>0$, $\varepsilon>0$, and a real-analytic map
\[
b=(b_1,\ldots,b_N)\longmapsto \varphi_b\in\mathcal A_\tau^\R,
\qquad \varphi_0=0,
\]
defined for $\|b\|:=\max_{1\le j\le N}|b_j|<\varepsilon$, such that
$U_b:=U_{\varphi_b}$ satisfies
\begin{equation}\label{eq:disk-inverse-identity}
\int_{U_b}|P_b(z)|^2 e^{-\pi|z|^2}e^{\pi\overline z w}\,\mathrm{d}A(z)
=\lambda Q_b(w),\qquad w\in\C.
\end{equation}
Moreover, $\|\varphi_b\|_{C^2(\mathbb S^1)}\le C_{\lambda,N}\|b\|$.
\end{proposition}

\begin{proof}
We adapt Cherednichenko's analytic-density boundary formulation
\cite[Theorem~2]{Cherednichenko1973} and conformal perturbation method
\cite{Cherednichenko1975}. We replace the area moments by a holomorphic
extension condition, solve it on a fixed circle, and then recover the
radial graph.

To write the real-density identity as moment equations, fix $b$ and
write $Q_b(w)=\sum_{j=0}^N q_j w^j$, where $Q_b$ is defined in
\eqref{eq:Q-def} and $q_j=0$ for $j>N$. We note that the constant
coefficient, which prescribes the total weighted mass, is real:
\[
q_0=1+\sum_{k=1}^N\frac{k!}{\pi^k}|b_k|^2.
\]
We can now match coefficients in the real-density identity
\eqref{eq:disk-inverse-identity}. For bounded $U$, the exponential series
for the kernel converges uniformly in $z\in\overline U$ and locally
uniformly in $w$, so the identity is equivalent to
\begin{equation}\label{eq:disk-moment-system}
\frac{\pi^j}{j!}\int_U |P_b(z)|^2\overline z^{\,j} e^{-\pi|z|^2}\,\mathrm{d}A(z)
=\lambda q_j,\qquad j\ge0.
\end{equation}
We now conjugate the moment equations so that the factors $z^j$ are
holomorphic. A single $\partial_{\bar z}$-primitive of the density will
then convert all the moments to contour integrals. Following
\cite[\S1, equations~(3)--(4)]{Cherednichenko1973}, we take the entire function
\[
A_b(z,s):=P_b(z)\int_0^s P_b^\#(v)e^{-\pi zv}\,\mathrm{d}v,
\qquad \partial_{\bar z} A_b(z,\bar z)=|P_b(z)|^2e^{-\pi|z|^2}.
\]
For an analytic Jordan boundary enclosing $0$, we can encode the prescribed
moments by poles at $0$. We subtract their contribution to form the boundary
trace
\begin{equation}\label{eq:disk-boundary-trace}
g_b(z):=A_b(z,\bar z)-\lambda\sum_{k=0}^N
\frac{k!\,\overline{q_k}}{\pi^{k+1}z^{k+1}},\qquad z\ne 0.
\end{equation}

\begin{lemma}\label{lem:disk-boundary-criterion}
Let $U$ be a bounded domain containing $0$ with an analytic Jordan boundary.
The moment equations \eqref{eq:disk-moment-system} are equivalent to
\[
\int_{\partial U} z^j g_b(z)\,\mathrm{d}z=0,\qquad j\ge0,
\]
with positive orientation. These integrals all vanish if and only if the
trace of $g_b$ on $\partial U$ is the boundary value of a function
holomorphic in $U$ and continuous on $\overline U$.
\end{lemma}

\begin{proof}
We first test the $j$th moment by integrating the trace against
$z^j\,\mathrm{d}z/(2i)$. Cauchy--Green gives the area term, while the
residue at $0$ selects the $k=j$ term in the pole sum (no term if $j>N$).
Thus
\[
\frac{1}{2i}\int_{\partial U} z^j g_b(z)\,\mathrm{d}z
=\int_U z^j|P_b(z)|^2e^{-\pi|z|^2}\,\mathrm{d}A(z)
-\lambda\overline{q_j}\frac{j!}{\pi^j}.
\]
Conjugating shows that its vanishing is exactly the $j$th equation in
\eqref{eq:disk-moment-system}.

We note that a holomorphic extension makes all the contour integrals
vanish by Cauchy's theorem. Conversely, expanding the exterior Cauchy
integral of $g_b$ at infinity gives the contour integrals as its
coefficients, up to a common
factor. If they vanish, the Cauchy integral is zero on the connected
exterior. The Cauchy jump formula then identifies the interior boundary
values with $g_b$, giving the required extension.
\end{proof}

We now transfer the extension condition to a fixed circle. Although the
radial graph will describe the final shape, here we seek a conformal map
$\Phi$ from the unit disk $\mathbb D$, close to $\zeta\mapsto r_\lambda\zeta$,
for which the trace of $g_b$ extends holomorphically into $U=\Phi(\mathbb D)$
and continuously to its closure. By Lemma~\ref{lem:disk-boundary-criterion},
this extension gives all the prescribed moments on $U$ at once, and hence
the real-density identity \eqref{eq:disk-inverse-identity}.

To this end, we parametrize the nearby domains by maps of the form
\[
\Phi(\zeta)=r_\lambda\zeta+c(\zeta),\qquad
c(\zeta)=\sum_{n\ge1}c_n\zeta^n,\qquad c_1\in\R.
\]
Using $\zeta=e^{2\pi i\theta}$, let $\mathcal X_\tau$ be the closed real
subspace of $\mathcal A_\tau$ consisting of series $c$ of this form, with
the inherited norm. Small corrections give $\Phi(0)=0$, $\Phi'(0)>0$, and
univalence on a neighborhood of $\overline{\mathbb D}$. Indeed, on a
fixed disk of radius strictly between $1$ and $e^\tau$, smallness in the
Wiener norm gives $\sup|\Phi'-r_\lambda|<r_\lambda$. Integration along line
segments then proves injectivity.

To impose the extension condition, let $\mathcal E(\zeta)=g_b(\Phi(\zeta))$
on $|\zeta|=1$ and denote its Laurent coefficients by $\mathcal E_k$.
For the prescribed $b$, we seek $\Phi$ satisfying
\[
\mathcal E_{-n}=\int_0^1 g_b\bigl(\Phi(e^{2\pi i\theta})\bigr)
 e^{2\pi i n\theta}\,\mathrm{d}\theta=0,\qquad n\ge1.
\]
Here $2\pi\theta$ is the angle of $\zeta$ on the source unit circle, not
the polar angle of its image $\Phi(\zeta)$.

These equations are equivalent to the required extension. If the trace
extends holomorphically into $U$ and continuously to its closure, its
pullback by $\Phi$ has no negative modes. For the converse, we note that
$\theta\mapsto\mathcal E(e^{2\pi i\theta})$ is real-analytic: $\Phi$
extends holomorphically through the source circle and is nonzero there,
and $g_b$ is real-analytic away from $0$. If the negative modes vanish,
the remaining series gives a holomorphic function on the disk, continuous
up to the circle. Composing with $\Phi^{-1}$ gives the extension into $U$.

We claim then that in the equation $\mathcal E_{-1}=0$ it suffices to
impose the real part, provided $\mathcal E_{-n}=0$ for $n\ge2$. To see
this, we change variables in the zeroth contour integral of the lemma:
\[
\begin{aligned}
\frac{1}{2\pi i}\int_{|\zeta|=1}
\mathcal E(\zeta)\Phi'(\zeta)\,\mathrm{d}\zeta
&=\frac{1}{\pi}\int_U |P_b(z)|^2e^{-\pi|z|^2}\,\mathrm{d}A(z)
-\frac{\lambda q_0}{\pi}\in\R.
\end{aligned}
\]
When the modes of order $-2$ and below vanish, the left side is
$\mathcal E_{-1}\Phi'(0)$. Since $\Phi'(0)>0$ is real, $\mathcal E_{-1}$
is already real. This is the real-residue compatibility condition;
compare \cite[\S5]{Cherednichenko1975}.

We can therefore solve for $c$ by setting $\mathcal R(b,c)=0$, where the
boundary operator $\mathcal R$ moves the required negative modes to the
corresponding positive modes. For $(b,c)$ near $(0,0)$ in
$\C^N\times\mathcal X_\tau$, set
\[
\mathcal R(b,c):=(\Re\mathcal E_{-1})\zeta
+\sum_{n\ge2}\mathcal E_{-n}\zeta^n\in\mathcal X_\tau.
\]

Before applying the implicit-function theorem, we verify that
$\mathcal R$ is real-analytic near $(0,0)$ in the stated spaces. Indeed,
both $\Phi$ and its conjugate on the circle belong to $\mathcal A_\tau$:
the latter has series
$r_\lambda\zeta^{-1}+\sum_{n\ge1}\bar c_n\zeta^{-n}$. Thus the Taylor
series of $A_b$ converges absolutely after substitution in the algebra.
The pole terms are analytic in $c$ by expanding the reciprocal of
$\Phi/\zeta$ about $r_\lambda$ in the same norm. The dependence on $b$ and
$\bar b$ in $A_b$ and $Q_b$ is polynomial, and extracting the displayed
coefficients is bounded and real-linear.

We next show that at the unperturbed disk $\mathcal R(0,0)=0$ and that
the Fr\'echet derivative $D_c\mathcal R(0,0)\colon\mathcal X_\tau\to\mathcal X_\tau$
is a bounded isomorphism. To see both facts, hold $b=0$ fixed. The primitive is
$A_0(z,s)=(1-e^{-\pi zs})/(\pi z)$, extended by continuity at $z=0$, so on
subtracting the pole the trace is
\[
\mathcal E(\zeta)=\frac{1-e^{-\pi|\Phi(\zeta)|^2}-\lambda}{\pi\Phi(\zeta)},
\qquad |\zeta|=1.
\]
We note that the numerator vanishes at $\Phi=r_\lambda\zeta$. To compute the
linear response to $c$, set $\Phi=r_\lambda\zeta+t c(\zeta)$ for real
$t$ and differentiate this trace at $t=0$. Only the derivative of the
numerator contributes, giving
\[
(1-\lambda)\left(\sum_{n\ge1}c_n\zeta^{n-2}
+\sum_{n\ge1}\bar c_n\zeta^{-n}\right),
\]
since $\overline\Phi/\Phi=\zeta^{-2}$ at the disk. Only the conjugate
variation contributes to modes $-2,-3,\dots$; at mode $-1$ both terms
contribute $c_1$. Thus $D_c\mathcal R(0,0)$ acts on the coefficients by
\[
(c_1,c_2,c_3,\ldots)\longmapsto(1-\lambda)(2c_1,\bar c_2,\bar c_3,\ldots).
\]
We remark that $1-\lambda=e^{-\pi r_\lambda^2}>0$ is the Gaussian density
on the circle; apart from this factor the linearization is the one for
constant density. The inverse divides by $1-\lambda$, halves the first
coefficient, and conjugates the rest, so it is bounded in the same Wiener
norm.

We can now apply the real-analytic Banach-space implicit-function theorem.
It gives a solution $\Phi_b$ depending real-analytically on $b$, with
$\|\Phi_b-r_\lambda\zeta\|_{\mathcal A_\tau}\le C\|b\|$. Since the
boundary operator vanishes, the image $U_b=\Phi_b(\mathbb D)$ satisfies
the real-density identity \eqref{eq:disk-inverse-identity}.

We now need to express $U_b$ as a radial graph without losing its
real-analytic dependence on $b$. This is the content of the next lemma.

\begin{lemma}\label{lem:conformal-radial-graph}
Let $R>0$ and $0<\sigma<\tau$. If $\Phi-R\zeta\in\mathcal X_\tau$ is
sufficiently small, then
\[
\Phi(\mathbb D)=\{re^{2\pi i\theta}:\theta\in\T,\ 0\le r<R+\varphi(\theta)\}
\]
for a radial perturbation $\varphi\in\mathcal A_\sigma^\R$ that depends
real-analytically on $\Phi$. Moreover,
$\|\varphi\|_{\mathcal A_\sigma}\le C\|\Phi-R\zeta\|_{\mathcal A_\tau}$,
where $C$ depends on $R,\sigma,\tau$.
\end{lemma}

\begin{proof}
We first note that smallness in $\mathcal A_\tau$ gives $\sup|\Phi'-R|<R$
on a disk slightly larger than $\mathbb D$, so integration along line
segments makes $\Phi$ univalent there. To read off the radius in polar angle, write
\[
L(\theta):=\log\frac{\Phi(e^{2\pi i\theta})}{R e^{2\pi i\theta}},
\]
using the branch with $\log 1=0$. The logarithm lies in $\mathcal A_\tau$ and
depends real-analytically on $\Phi$ by the convergent power series for
the logarithm in the algebra. Its real part gives the log-radius relative
to $R$, and its imaginary part gives the change in angle.

We next verify that the change of parameter is analytic if we allow a
smaller strip. For $g\in\mathcal A_\tau$ and $\eta\in\mathcal A_\sigma^\R$,
the algebra bound on the exponential gives
\[
\|g(\theta+\eta(\theta))\|_{\mathcal A_\sigma}
\le\sum_{k\in\Z}|g_k|e^{(\sigma+2\pi\|\eta\|_{\mathcal A_\sigma})|k|}
\le\|g\|_{\mathcal A_\tau}
\quad\text{if }2\pi\|\eta\|_{\mathcal A_\sigma}<\tau-\sigma.
\]
The normally convergent exponential series on smaller balls gives
real-analytic dependence on $(g,\eta)$ in these norms.

With this estimate, we can solve for the conformal parameter at a prescribed
polar angle. To reach $2\pi\theta$, we use $\theta+\eta(\theta)$, where
\[
\eta(\theta)+\frac{1}{2\pi}\Im L(\theta+\eta(\theta))=0.
\]
Real and imaginary parts are taken on the real circle. The derivative
with respect to $\eta$ at $(L,\eta)=(0,0)$ is the identity. The
implicit-function theorem gives a small $\eta\in\mathcal A_\sigma^\R$
depending real-analytically on $L$. By $C^1$ smallness, the change of
parameter is a diffeomorphism. The boundary radius at angle $2\pi\theta$
is therefore $R+\varphi(\theta)$, where
\[
\varphi(\theta)=R\left(\exp\bigl(\Re L(\theta+\eta(\theta))\bigr)-1\right).
\]
We thus obtain the stated radial domain, since the radius is positive.
The displayed formula is real-analytic in $\Phi$ with values in
$\mathcal A_\sigma^\R$ and vanishes at $\Phi=R\zeta$, giving the norm bound.
\end{proof}

Returning to $U_b$, we take $R=r_\lambda$ in
Lemma~\ref{lem:conformal-radial-graph}. This gives $U_b=U_{\varphi_b}$
with $\varphi_b$ real-analytic in $b$ and
$\|\varphi_b\|_{\mathcal A_\sigma}\le C\|b\|$. The embedding into
$C^2$ gives the claimed bound. Taking $\sigma$ as the width in the
proposition completes the proof.
\end{proof}

Now that we have the real-density identity, we return to the original
eigenfunction equation. To recover it, we compare the growth allowed by
the bounded integration domain with that of the homogeneous solutions.

\begin{lemma}\label{lem:step3-rigidity}
Let $N\ge1$ and let $P_b(z)=1+\sum_{k=1}^N b_kz^k$. No smallness
assumption on $b$ is needed. Suppose an entire function $H$ satisfies
\begin{equation}\label{eq:lemma-step3-eq}
P_b^\#(\partial_w/\pi)H(w)=0,\qquad w\in\C,
\end{equation}
and, for some $C>0$ and $M>0$,
\begin{equation}\label{eq:lemma-step3-growth}
|H(w)|\le C(1+|w|)^N e^{\pi M|w|},\qquad w\in\C.
\end{equation}
If every characteristic root $\zeta_j$ (so $P_b^\#(\zeta_j/\pi)=0$)
satisfies
\begin{equation}\label{eq:lemma-step3-zeros}
|\zeta_j|>\pi M,
\end{equation}
then $H\equiv0$.
\end{lemma}

\begin{proof}
We first note that when $P_b\equiv1$ the equation is $H=0$. Otherwise,
write the characteristic polynomial as a nonzero constant times
$\prod_{j=1}^\ell(\zeta-\zeta_j)^{m_j}$, where the roots are distinct
and $m_j$ are their multiplicities. The standard solution formula for
constant-coefficient ordinary differential equations
(cf.~\cite[\S12.7]{Hormander-PDE2}) gives
\begin{equation}\label{eq:lemma-step3-decomposition}
H(w)=\sum_{j=1}^\ell R_j(w)e^{\zeta_j w},
\qquad \deg R_j\le m_j-1.
\end{equation}
We now rule out cancellation. If any $R_j\not\equiv0$, choose an exponent
$\zeta_{j_0}$ of maximal modulus among those with $R_j\not\equiv0$. Along
$w_t=t\overline{\zeta_{j_0}}$, $t\ge0$, its real growth rate is strictly
larger than that of every other nonzero term:
\begin{equation}\label{eq:lemma-step3-strict}
\Re(\zeta_j\overline{\zeta_{j_0}})<|\zeta_{j_0}|^2\qquad (j\ne j_0,\ R_j\not\equiv0).
\end{equation}
The inequality follows from $|\zeta_j|\le|\zeta_{j_0}|$, since equality
would require both equal modulus and the same argument. The other terms,
even with their polynomial factors, are therefore $o(e^{t|\zeta_{j_0}|^2})$,
while $|R_{j_0}(w_t)|$ stays bounded away from zero for large $t$. Thus
\begin{equation}\label{eq:lemma-step3-lower}
|H(w_t)|\ge c_0 e^{t|\zeta_{j_0}|^2}
\qquad\text{for all sufficiently large }t.
\end{equation}
The assumed growth bound gives instead
\begin{equation}\label{eq:lemma-step3-upper}
|H(w_t)|\le C(1+t|\zeta_{j_0}|)^N e^{\pi M t|\zeta_{j_0}|}.
\end{equation}
Taking logarithms, dividing by $t$, and letting $t\to\infty$ yields
$|\zeta_{j_0}|\le\pi M$, contrary to the root condition. Hence every
$R_j$ vanishes.
\end{proof}

\begin{proof}[Proof of Theorem \ref{thm:perturb}]
We first use the conversion from Hermite coefficients to monomials: it
bounds $\|b\|$ by a constant depending only on $N$ times $\max_k|a_k|$. For
sufficiently small data, Proposition~\ref{prop:disk-inverse-tailored}
therefore gives an analytic radial graph $U=U_b$ with the real-density
identity and $M=\sup_{z\in U}|z|<3r_\lambda/2$. The localization integral
is entire and satisfies
\begin{equation}\label{eq:G-growth-bound}
|(\mathcal L_U P_b)(w)|\le
\left(\sup_{z\in U}|P_b(z)|e^{-\pi|z|^2}\right)|U|e^{\pi M|w|}.
\end{equation}
Since $P_b$ has degree at most $N$, the residual
$H=\mathcal L_U P_b-\lambda P_b$ satisfies the growth bound of
Lemma~\ref{lem:step3-rigidity}. By the kernel differentiation formula
\eqref{eq:Phash-on-planewaves}, the real-density identity gives the
homogeneous equation $P_b^\#(\partial_w/\pi)H=0$.

Finally, shrink $\varepsilon_0(\lambda,N)$ so that
\[
\|b\|\sum_{k=1}^N(2r_\lambda)^k<1.
\]
Then $|P_b^\#(z)-1|<1$ on $|z|\le 2r_\lambda$, so every characteristic
root has modulus greater than $2\pi r_\lambda>\pi M$. The lemma now gives
$H=0$: the constructed domain solves the original localization equation,
not just its differentiated version. Applying the inverse Bargmann
transform completes the proof.
\end{proof}


\section{Proof of the Abreu-D\"orfler theorem}\label{sec:classical-rigidity}

We now move on to the proof of of Theorem \ref{thm:eig-loc-U}, providing one proof in the case where the boundary of the set $U$ is assumed to be somewhat regular, and one proof in the general case. In \S\ref{subsec:counterexample} we then prove Theorem~\ref{thm:counterexample}, which shows that the simple connectivity hypothesis cannot be dropped from either statement.

\subsection{First proof} The main tools used in the first proof we provide orbit around the Paley--Wiener--Schwartz theorem and divisibility results for distributions. We recall that such results have been previously collected 
collected in Subsection~\ref{sec:fourier-to-fboundary}, and point to the relevant
notation only when it enters the argument.

\begin{proof}[First proof under the additional assumption that $\partial U$ is $C^1$]
The Bargmann transform of the Hermite function $h_k$ is a nonzero scalar
multiple of the monomial $z^k$. Thus, arguing as in
\cite{Abreu-Doerfler}, the assumption that $h_k$ is an eigenfunction of
the localization operator $\mathcal{L}_U$ is equivalent to the identity
\[
\int_U z^k\, e^{\pi w\overline{z}}\, e^{-\pi |z|^2}\,\mathrm{d}A(z)
=\lambda\, w^k,
\qquad w\in \C.
\]
Both sides are entire functions of $w$. Since $U$ is bounded and the
weight $e^{-\pi|z|^2}$ is integrable, the dominated convergence
theorem applied to difference quotients shows that for each
multi-index $\alpha$ in $w$, the partial derivative
$\partial_w^\alpha$ commutes with the integral on the left-hand side.
Iterating $k$ times, we obtain
\[
\frac{1}{(\pi)^k}\,\partial_w^k
\Big[\int_U z^k\, e^{\pi w\overline{z}}\, e^{-\pi |z|^2}\,\mathrm{d}A(z)\Big]
=\int_U z^k\overline{z}^k\, e^{\pi w\overline{z}}\, e^{-\pi |z|^2}\,\mathrm{d}A(z)
=\int_U |z|^{2k}\, e^{\pi w\overline{z}}\, e^{-\pi |z|^2}\,\mathrm{d}A(z).
\]
On the right-hand side, $\partial_w^k(\lambda w^k)=\lambda\,k!$.
Setting $c(\lambda,k):=\lambda\,k!/\pi^k\in\R$ and using
$\mathbf{1}_U$ to extend the integration to $\C$, we obtain
\begin{equation}\label{eq:fourier-delta}
\int_{\C} \mathbf{1}_U(z)\,|z|^{2k}\,e^{-\pi |z|^2}\, e^{\pi \overline{z}\, w}\,\mathrm{d}A(z)
=c(\lambda,k),
\end{equation}
for some real constant $c(\lambda,k)$.

Consider now the compactly supported distribution
\begin{equation}\label{eq:eta-def-section4}
\eta:=|z|^{2k}e^{-\pi |z|^2}\mathbf{1}_U(z)-c(\lambda,k)\,\delta_0,
\end{equation}
which is a finite signed Radon measure with compact support in any closed
ball containing $U\cup\{0\}$. Its absolutely continuous part has bounded
density $|z|^{2k}e^{-\pi|z|^2}\mathbf{1}_U\in L^\infty_c(\R^2)$, and its
pure-point part is the single Dirac mass $-c(\lambda,k)\delta_0$. Hence
$\eta$ is a singular-continuous free measure in the sense of
Subsection~\ref{sec:fourier-to-fboundary}, and we may apply
Theorem~\ref{thm:PW-planar} and Lemma~\ref{lemma:fourier-to-fboundary} to it.

Writing $z=x+iy$ and $w\in\C$, we rewrite the exponential factor in
\eqref{eq:fourier-delta} as
\[
e^{\pi \overline{z}\, w}=e^{\pi(x-iy)w}=e^{\pi xw}\, e^{-i\pi yw}.
\]
The function class entering the Paley--Wiener step is now unambiguous:
$\eta$ is a compactly supported distribution of order zero, in fact a finite
signed Radon measure. Theorem~\ref{thm:PW-planar}
therefore applies and produces an entire function $\widehat\eta$ on
$\C^2$ obeying the Paley--Wiener bound \eqref{eq:PW-bound-planar} with
$N=0$ and $R$ equal to the radius of any closed disk containing
$U\cup\{0\}$; in particular $\widehat\eta|_{\R^2}$ is bounded by
$\|\eta\|_{\mathrm{TV}}$ and extends to $\C^2$ as an entire function of
exponential type. With the convention
$\widehat g(\xi,\nu):=\int_{\R^2}g(x,y)\,e^{-2\pi i(x\xi+y\nu)}\,\mathrm{d}x\,\mathrm{d}y$
used in Subsection~\ref{sec:fourier-to-fboundary}, a comparison of exponents
shows that for $w\in\C$ the kernel $e^{\pi xw}\,e^{-i\pi yw}$ coincides
with $e^{-2\pi i(x\xi+y\nu)}$ on the complex line
$(\xi,\nu)=\bigl(\tfrac{iw}{2},\tfrac{w}{2}\bigr)$, $w\in\C$, whose
defining equation is $\nu=-i\xi$. Reparametrizing this line in the form
$(\xi,\nu)=(w,-iw)$, identity \eqref{eq:fourier-delta}---which says that
the integral on the left-hand side, viewed as an entire function of
$w\in\C$, is identically equal to the constant $c(\lambda,k)$, so that
the Fourier--Laplace transform of the modified distribution $\eta$
vanishes there---becomes
\begin{equation}\label{eq:eta-vanishes-on-line}
\widehat{\eta}(w,-iw)=0,\qquad \forall w\in\C.
\end{equation}
Symmetrically, since $c(\lambda,k)\in\R$ and the absolutely continuous
density of $\eta$ is real-valued, complex conjugation of
\eqref{eq:fourier-delta} after the substitution $w\mapsto\overline w$,
combined with the same Paley--Wiener identification applied to the
conjugate complex line $\nu=i\xi$, yields
\begin{equation}\label{eq:eta-vanishes-on-conjugate-line}
\widehat{\eta}(w,iw)=0,\qquad \forall w\in\C.
\end{equation}
Identities \eqref{eq:eta-vanishes-on-line} and
\eqref{eq:eta-vanishes-on-conjugate-line} together assert that
$\widehat\eta$ vanishes on the complex characteristic variety
$\Sigma=\{\xi^2+\nu^2=0\}$ of the Laplacian, that is, on the union of the two lines $\nu=\pm i\xi$.

Applying the divisibility lemma (Lemma~\ref{lemma:fourier-to-fboundary}) to the distribution $\eta$, we
obtain a compactly supported function $u\in L^2(\R^2)$, continuous
off a finite set, such that
\begin{equation}\label{eq:potato}
\Delta u=f(|z|)\,\mathbf{1}_U-c(\lambda,k)\,\delta_0
\end{equation}
in the sense of distributions on $\R^2$, where
\[
f(|z|)=|z|^{2k}e^{-\pi |z|^2}.
\]
Writing $u$ explicitly as a logarithmic potential of the right-hand side
of \eqref{eq:potato}, we conclude that $U$ satisfies
\begin{equation}\label{eq:potential}
\int_U \log|z-z'|\, f(|z'|)\,\mathrm{d}A(z')
=c(\lambda,k)\,\log|z|,
\qquad z\in \C\setminus (U\cup\{0\}).
\end{equation}

\medskip

\noindent{\bf Location of the logarithmic pole.}\quad
The constant \(c(\lambda,k)=\lambda k!/\pi^k\) is strictly positive. Hence
the pole at the origin in the right-hand side of \eqref{eq:potential} cannot
lie in the open exterior of \(U\): if \(0\notin\overline U\), then the
left-hand side of \eqref{eq:potential} is harmonic and bounded in a
neighborhood of \(0\), whereas \(c(\lambda,k)\log|z|\) is singular there.
If \(0\in\partial U\), the logarithmic potential on the left remains finite
along exterior points tending to \(0\), again contradicting the divergence of
\(c(\lambda,k)\log|z|\). Thus
\[
0\in U.
\]
This observation is all that is needed below: every boundary patch used in
the final contradiction is chosen away from the pole at \(0\), so no
puncturing of \(U\) is required.

\medskip

Define the logarithmic potential
\[
\Psi_U(z):=\int_U \log|z-z'|\, f(|z'|)\,\mathrm{d}A(z').
\]
We claim that $\Psi_U$ is \emph{locally radial}: for each $z_0\in U$,
there exists a ball $B(z_0)\subset U$ such that $\Psi_U$ agrees on
$B(z_0)$ with a radial function of $|z|$.

To prove this, let $B\supset U$ be a ball centered at $0$, and define
\[
\Phi(z):=\int_B \log|z-z'|\, f(|z'|)\,\mathrm{d}A(z').
\]
Since both $B$ and the weight $f(|z'|)$ are radial, $\Phi$ is radial; in
particular, $\nabla \Phi(z)$ is parallel to $z$ for every $z\in \R^2$.
Furthermore, on the interior of $B$,
\[
\Delta \Phi(z)=2\pi\, f(|z|).
\]

Now $\Psi_U$ satisfies \eqref{eq:potential} on the unbounded
component of $\C\setminus(U\cup\{0\})$. Differentiating that identity
in $z$ for $z\not\in \overline{U}\cup\{0\}$ gives
\[
\nabla \Psi_U(z)=c(\lambda,k)\,\frac{z}{|z|^2},
\qquad z\in \C\setminus(\overline{U}\cup\{0\}).
\]
Since \(f(|z|)\mathbf 1_U\in L^\infty_c(\R^2)\) and
\(z\mapsto z/|z|^2\) is locally integrable, the gradient
\(\nabla\Psi_U\) is a continuous convolution and therefore extends
continuously across \(\partial U\). Hence on \(\partial U\),
\[
\nabla \Psi_U(z)=c(\lambda,k)\,\frac{z}{|z|^2},
\qquad z\in \partial U,
\]
which is, in particular, parallel to $z$. Therefore the difference
\[
v:=\Psi_U-\Phi
\]
satisfies $\Delta v=2\pi f(|z|)\mathbf 1_U-2\pi f(|z|)\mathbf 1_B
=-2\pi f(|z|)\mathbf 1_{B\setminus U}$ in the sense of distributions
on $B$, so $v$ is harmonic in $U$ and in $B\setminus\overline{U}$, and
on $\partial U$ satisfies
\[
\nabla v(x)\ \text{is parallel to}\ x.
\]
The next lemma converts this gradient condition into local radiality of $v$.

\begin{lemma}\label{lem:harmonic-locally-radial}
Let \(V\subset \R^2\) be a bounded domain, and suppose there exists a harmonic
function \(v\in C^1(\overline{V})\) such that \(\nabla v(x)\) is parallel to
\(x\) for every \(x\in \partial V\). Then \(\nabla v(x)\) is parallel to \(x\)
for every $x\in V$. In particular, on every circle
$S=\{|x|=r\}\subset \R^2$, the function $v$ is constant on every
connected component of $S\cap V$.
\end{lemma}

\begin{proof}
Consider the angular derivative
\[
g(x):=x_1\,\partial_2 v(x)-x_2\,\partial_1 v(x)
=\partial_\theta v(x),
\]
which equals the tangential derivative of $v$ along the circle
through $x$ centered at $0$. A direct computation gives
\[
\Delta g
=x_1\,\Delta(\partial_2 v)-x_2\,\Delta(\partial_1 v)
+2\big(\partial_1\partial_2 v-\partial_2\partial_1 v\big)
=0,
\]
since $v$ is harmonic and partial derivatives of harmonic functions
are themselves harmonic. By the boundary hypothesis, $\nabla v(x)$ is
parallel to $x$ on $\partial V$, so $g\equiv 0$ on $\partial V$.
Together with $g\in C(\overline{V})$, the maximum principle for
harmonic functions on the bounded domain \(V\) yields \(g\equiv 0\) in \(V\).
Hence $\nabla v(x)$ is
parallel to $x$ at every interior point as well, and on any circle
$S=\{|x|=r\}$ the tangential derivative $\partial_\theta v$ vanishes
on $S\cap V$. Therefore $v$ is constant on each connected component
of $S\cap V$.
\end{proof}

Applying Lemma~\ref{lem:harmonic-locally-radial} directly with \(V=U\), we
obtain \(\partial_\theta v=0\) throughout \(U\). Since \(\Phi\) is globally
radial, \(\partial_\theta\Psi_U=0\) throughout \(U\) as well. Equivalently,
on every circle centered at the origin, \(\Psi_U\) is constant on each
connected component of its intersection with \(U\).

We now conclude the proof. Since \(U\) is simply connected and bounded,
it suffices to rule out the possibility that \(U\) fails to be a disk
centered at the origin. Suppose that
\(U\) is not such a disk. Then \(\partial U\) is not a single circle
centered at $0$. Since the \(C^1\) boundary of a bounded simply connected
domain is a single \(C^1\) Jordan curve, if the function \(x\mapsto |x|\)
were locally constant on all of \(\partial U\), it would be constant on
\(\partial U\), and \(U\) would be the disk bounded by that centered circle.
Hence there is a point \(x_0\in\partial U\), \(x_0\ne0\), at which the
radial projection
\[
\rho(x):=|x|
\]
has nonzero tangential derivative along \(\partial U\). Equivalently, for
some \(C^1\) parametrization \(\gamma\) of \(\partial U\) with
\(\gamma(0)=x_0\),
\begin{equation}\label{eq:transverse-radius-point}
\frac{d}{dt}|\gamma(t)|\Big|_{t=0}\ne0 .
\end{equation}
Choose radii \(0<r_1<|x_0|<r_2\) so close to \(|x_0|\) that the annulus
\[
A:=\{z\in \R^2:\ r_1<|z|<r_2\}
\]
contains no origin. Below \(r>0\) is chosen small enough that
\(\overline{B_r(x_0)}\subset A\).

We next choose a local radial chart. By \eqref{eq:transverse-radius-point},
the map \(\rho|_{\partial U}\) is a \(C^1\) submersion near \(x_0\). Therefore,
after shrinking \(r>0\), the boundary arc \(\partial U\cap B_r(x_0)\) can be
written in polar coordinates as
\[
\theta=\alpha(\rho),\qquad \rho\in I,
\]
for a \(C^1\) function \(\alpha\) on a nontrivial interval
\(I\Subset(r_1,r_2)\). Since \(\partial U\) is \(C^1\), the ball can be chosen
so small that
\[
\widetilde{U}_0:=U\cap B_r(x_0)
\]
is connected, \(0\notin \overline{\widetilde U_0}\), and each circle
\(\{|x|=\rho\}\), \(\rho\in I\), meets \(\widetilde U_0\) in a single open
arc. Because \(\partial_\theta\Psi_U=0\) in \(U\), the function \(\Psi_U\) is
constant on each such arc. Thus there is a scalar function
\(\psi_U\colon I\to\R\) such that
\[
\Psi_U(x)=\psi_U(|x|),\qquad x\in \widetilde{U}_0,\ |x|\in I.
\]
Moreover \(\psi_U\in C^1(I)\), because \(\Psi_U\in C^1(\overline U)\) in the
present \(C^1\) boundary setting and \(\rho\) is a submersion in the chosen
chart.

On the other hand, by \eqref{eq:potential} and the continuity of the
gradient of a volume logarithmic potential with bounded density across
a \(C^1\) boundary, the gradient
\(\nabla\Psi_U\) extends continuously across \(\partial U\), and for
$z\in\partial U$ we have
\[
\nabla \Psi_U(z)=c\,\frac{z}{|z|^2}
\]
for the same constant $c=c(\lambda,k)\in \R$ as in
\eqref{eq:potential}. Combining this with $\Psi_U(x)=\psi_U(|x|)$ on
$\widetilde U_0$ and passing to $z\in \partial \widetilde U_0\cap
\partial U$ from the inside,
\[
\psi_U'(|z|)\,\frac{z}{|z|}=c\,\frac{z}{|z|^2},
\qquad z\in \partial U\cap \overline{B_r(x_0)},
\]
hence, using the polar chart and the fact that every \(\rho\in I\) occurs
on the boundary arc,
\[
\psi_U'(\rho)=\frac{c}{\rho}
\qquad\text{for every }\rho\in I.
\]
Therefore
\[
\psi_U'(r)=\frac{c}{r}
\qquad\text{for every }r\in I,
\]
and thus
\[
\psi_U(r)=c\log r+C_0,
\qquad r\in I,
\]
for some real constant $C_0$.

Let
\[
W_I:=\{x\in \widetilde U_0:\ |x|\in I\}.
\]
Then \(W_I\) is a nonempty open subset of \(U\), and on \(W_I\),
\[
\Psi_U(x)=c\log|x|+C_0.
\]
Hence
\[
\Delta \Psi_U\equiv 0 \qquad \text{in } W_I,
\]
since the origin is excluded from $\widetilde U_0$. On the other hand, on
the same open set \(W_I\subset U\setminus\{0\}\), the logarithmic potential
\(\Psi_U\) has no pole contribution and satisfies
\[
\Delta \Psi_U=2\pi\, f(|z|)=2\pi\, |z|^{2k}e^{-\pi |z|^2}.
\]
The right-hand side vanishes only at the origin, which lies outside
$W_I$, while the left-hand side just shown to be $0$ on $W_I$ — a
contradiction. This rules out the case in which $U$ is not a disk
centered at $0$, completing the proof.
\end{proof}

\begin{remark}
The structure of the argument above is in the spirit of
\cite{Aharonov-Schiffer-Zalcman}, where logarithmic-potential
rigidity statements were derived from Schiffer-style overdetermined
boundary conditions, and of the quadrature-domain literature
\cite{Gustafsson,Gustafsson-Shahgholian,Shahgholian,Karp}, in which
the regularity of free boundaries and the matching of an interior
potential with the Newtonian potential of a radially symmetric
density play a central role. Concretely, identity
\eqref{eq:potential} can be read as saying that $U$ is a
\emph{generalized quadrature domain} for the radial weight $f(|z|)$,
with quadrature data concentrated at the origin. From this perspective,
the conclusion that $U$ is a disk is a (sharp) instance of the
classical principle that radial quadrature data forces a radial
domain. The \(C^1\) hypothesis is used to obtain the boundary arc and
the local radial chart in the final step.
\end{remark}

\subsection{Second proof}\label{subsec:second-proof}

The first proof established Theorem~\ref{thm:eig-loc-U} under the auxiliary
hypothesis that $\partial U$ is of class $C^1$. The argument we now give does
not use differentiability, perimeter, or density assumptions on the boundary.
It proves the rigidity statement in the regular-open
representative class, which is the formulation naturally detected by the
operator \(\mathcal L_U\).

\begin{proof}
We may assume that \(U\) is regular open and satisfies
\[
U=\operatorname{int}\overline U.
\]

We keep the notation
\[
f(|z|):=|z|^{2k}e^{-\pi|z|^2},
\qquad
c:=c(\lambda,k),
\]
from the first proof. The argument below is in the spirit of the holomorphic
potential-theoretic and quadrature-domain methods of
\cite{Aharonov-Schiffer-Zalcman,Gustafsson,Karp},
in that we exploit the existence of a holomorphic function on \(U\) with prescribed real
boundary values to recover full radial symmetry. Starting again from \eqref{eq:fourier-delta} and
applying Lemma~\ref{lemma:fourier-to-fboundary}, we obtain a compactly supported distributional
solution of
\begin{equation}\label{eq:potato-second-proof}
\Delta u=f(|z|)\mathbf{1}_U-c\,\delta_0
\qquad\text{in }\mathcal D'(\R^2).
\end{equation}
Using the planar fundamental solution \((2\pi)^{-1}\log|z|\), we may write
\begin{equation}\label{eq:explicit-u-second-proof}
u(z):=\frac{1}{2\pi}\int_U \log|z-z'|\,f(|z'|)\,\mathrm{d}A(z')-\frac{c}{2\pi}\log|z|.
\end{equation}
The derivation of \eqref{eq:potential} used only the hypotheses common to
both proofs, before any boundary regularity was invoked. Hence this function
satisfies
\[
u\equiv 0
\qquad\text{on }\R^2\setminus (U\cup\{0\}).
\]

We first show that \(0\in U\). Since the singular term in
\eqref{eq:potato-second-proof} is \(-c\delta_0\), one must have \(0\in\overline U\). If
\(0\in\partial U\), then there exist exterior points \(z_j\to 0\) with
\[
u(z_j)=0.
\]
On the other hand, the first term in \eqref{eq:explicit-u-second-proof} remains finite as
\(z\to 0\), because \(f\in L^\infty(U)\) and \(\log|z-z'|\) is locally integrable, whereas the
second term equals \(-\frac{c}{2\pi}\log|z|\to +\infty\). This contradicts the exterior identity
\(u(z_j)=0\). Hence \(0\notin\partial U\), and therefore
\[
0\in U.
\]

We next show directly from \eqref{eq:explicit-u-second-proof} that
\[
u\in C^1(\R^2\setminus\{0\}).
\]
Indeed, if we write
\[
g(z):=f(|z|)\mathbf{1}_U(z)\in L^\infty_c(\R^2),
\]
then
\[
u(z)=\frac{1}{2\pi}\int_{\R^2}\log|z-z'|\,g(z')\,\mathrm{d}A(z')-\frac{c}{2\pi}\log|z|.
\]
Since \(g\in L^\infty_c(\R^2)\), the only contribution to the integral defining the convolution
\((K\ast g)(z)\) below comes from a fixed bounded set, so we may replace \(K\) by its restriction
\(K\mathbf{1}_{B_R}\) for any sufficiently large ball \(B_R\). The kernel
\(K(\zeta)=\zeta/|\zeta|^2\) satisfies \(|K(\zeta)|=|\zeta|^{-1}\), which is integrable on bounded
subsets of \(\R^2\); hence \(K\mathbf{1}_{B_R}\in L^1(\R^2)\). It follows that
the convolution of an \(L^1\) function with a bounded function is continuous
(this follows from the density of \(C_c(\R^2)\) in \(L^1(\R^2)\) and continuity of translations in
\(L^1\)), so
\[
(K\ast g)(z):=\int_{\R^2}\frac{z-z'}{|z-z'|^2}\,g(z')\,\mathrm{d}A(z')
\]
is a continuous function of \(z\in\R^2\). On the other hand, a direct distributional computation
shows
\[
\nabla\!\left(\int_{\R^2}\log|z-z'|\,g(z')\,\mathrm{d}A(z')\right)=K\ast g
\qquad\text{in }\mathcal D'(\R^2),
\]
so the gradient of the convolution \(\log|\cdot|\ast g\) coincides almost everywhere with the
continuous function \(K\ast g\). It follows that \(\log|\cdot|\ast g\) is of class \(C^1\) on
\(\R^2\), and hence \(u\in C^1(\R^2\setminus\{0\})\) (the only obstruction to global \(C^1\)
regularity being the explicit logarithmic singularity at the origin coming from the
\(-\frac{c}{2\pi}\log|z|\) term).

Now let \(z_*\in\partial U\). Since \(0\in U\), the boundary point \(z_*\) is bounded away from
\(0\); fix a small ball \(B_\rho(z_*)\Subset \R^2\setminus\{0\}\). On the exterior
\(\R^2\setminus \overline U\) we have \(u\equiv 0\), so \(\nabla u\equiv 0\) there as well.
Since
\(U=\operatorname{int}\overline U\), every point of
\(\partial U=\partial\overline U\) is a limit of points of
\(\R^2\setminus\overline U\). Hence there are exterior points \(z_j\to z_*\)
with \(\nabla u(z_j)=0\). By the \(C^1\) regularity of \(u\) on
\(B_\rho(z_*)\), the gradient \(\nabla u\) is continuous at \(z_*\), and we
conclude \(\nabla u(z_*)=0\). This yields
\begin{equation}\label{eq:dz-zero-second-proof}
\partial_z u=0
\qquad\text{on }\partial U.
\end{equation}
Define
\[
\Phi(t):=\int_0^t s^k e^{-\pi s}\,\mathrm{d}s
\qquad\text{and}\qquad
\Psi(z):=\frac{\Phi(|z|^2)}{z},
\quad z\neq 0.
\]
A direct computation, using \(|z|^2=z\bar z\) and \(\partial_{\bar z}(z\bar z)=z\), gives, for
\(z\neq 0\),
\[
\partial_{\bar z}\Psi(z)
=
\frac{\Phi'(|z|^2)\,\partial_{\bar z}(|z|^2)}{z}
=
\frac{|z|^{2k}e^{-\pi|z|^2}\cdot z}{z}
=
|z|^{2k}e^{-\pi|z|^2}
=
f(|z|).
\]
At the origin, the apparent singularity \(1/z\) is integrable against the vanishing factor
\(\Phi(|z|^2)\): since \(k\ge 0\) and \(s^k e^{-\pi s}\le s^k\) for \(s\ge 0\), we have
\(\Phi(t)=\int_0^t s^k e^{-\pi s}\,\mathrm{d}s\le \int_0^t s^k\,\mathrm{d}s = t^{k+1}/(k+1)=O(t^{k+1})\) as
\(t\to 0\). Hence \(\Psi(z)=O(|z|^{2k+1})\) near \(z=0\), so \(\Psi\in L^\infty_{\mathrm{loc}}\)
(in fact continuous, with \(\Psi(0)=0\)). Consequently, the identity
\(\partial_{\bar z}\Psi=f(|z|)\) persists in the distributional sense on all of \(U\); no extra
\(\delta_0\) contribution arises at the origin.

Recall also the standard identity
\[
\partial_{\bar z}\!\left(\frac{1}{z}\right)=\pi\delta_0
\qquad\text{in }\mathcal D'(\R^2),
\]
which is equivalent to the fact that \(\frac{1}{\pi z}\) is a fundamental solution of
\(\partial_{\bar z}\); equivalently, \(\Delta\log|z|=4\partial_z\partial_{\bar z}\log|z|=2\pi\delta_0\)
combined with \(\partial_z\log|z|=\frac{1}{2z}\). Now set
\[
F(z):=4\partial_z u(z)-\Psi(z)+\frac{c}{\pi z},
\qquad z\in U\setminus\{0\}.
\]
We first check that \(F\) defines a distribution on all of \(U\) (a~priori it is given only on
\(U\setminus\{0\}\)). From \eqref{eq:explicit-u-second-proof} we have, near \(z=0\),
\[
4\partial_z u(z)
=
\frac{2}{\pi}\partial_z\!\left(\int_{\R^2}\log|z-z'|\,g(z')\,\mathrm{d}A(z')\right)-\frac{c}{\pi z},
\]
where the first term is locally bounded near \(z=0\) (the Cauchy-type integral
\(\int (z-z')^{-1} g(z')\,\mathrm{d}A(z')\) of a bounded compactly supported function is locally
integrable, by the same \(L^1_{\mathrm{loc}}\)-kernel argument as above). Adding \(-\Psi\)
(continuous with \(\Psi(0)=0\)) and \(c/(\pi z)\) cancels the explicit \(1/z\)-singularity, so
\(F\) is locally bounded near the origin and hence defines a distribution on \(U\).
Here \(\mathcal D'(U)\) denotes the space of distributions on \(U\).

Using \(\Delta=4\partial_z\partial_{\bar z}\), the relation \(\partial_{\bar z}(1/z)=\pi\delta_0\),
the computation of \(\partial_{\bar z}\Psi\) above, and \eqref{eq:potato-second-proof}, we
compute, in \(\mathcal D'(U)\),
\[
\partial_{\bar z}F
=
4\partial_z\partial_{\bar z}u-\partial_{\bar z}\Psi+\frac{c}{\pi}\partial_{\bar z}\!\left(\frac{1}{z}\right)
=
\Delta u-f(|z|)+c\,\delta_0
=
\bigl(f(|z|)-c\delta_0\bigr)-f(|z|)+c\delta_0
=0.
\]
The singular term \(-c\delta_0\) in \eqref{eq:potato-second-proof}, arriving through
\(4\partial_z\partial_{\bar z}u=\Delta u\), is thus exactly cancelled by
\((c/\pi)\partial_{\bar z}(1/z)=c\delta_0\). Hence \(F\) is distributionally
\(\bar\partial\)-closed on \(U\), and Weyl's lemma shows that \(F\) may be
represented by a holomorphic function on \(U\); we continue to denote it by
\(F\).

Define
\[
W(z):=\pi z\,F(z).
\]
Then \(W\) is holomorphic on \(U\) and \(W(0)=0\). Since \(\partial_z u\) is continuous up to
\(\partial U\), so is \(W\), and by \eqref{eq:dz-zero-second-proof},
\[
W|_{\partial U}
=
\pi z\left(4\partial_z u-\frac{\Phi(|z|^2)}{z}+\frac{c}{\pi z}\right)\Big|_{\partial U}
=
c-\pi\Phi(|z|^2).
\]
In particular, \(W|_{\partial U}\) is real-valued. The imaginary part of \(W\) is therefore
harmonic in \(U\), continuous up to \(\partial U\), and vanishes on \(\partial U\). By the
maximum principle,
\[
\Im W\equiv 0
\qquad\text{in }U.
\]
Since a holomorphic function with zero imaginary part is constant, there exists \(\kappa\in\R\)
such that
\[
W\equiv \kappa
\qquad\text{in }U.
\]

Restricting to the boundary, and using that \(W\) is continuous up to \(\partial U\) with
boundary value \(c-\pi\Phi(|z|^2)\), we obtain
\[
c-\pi\Phi(|z|^2)=\kappa
\qquad\text{for all }z\in\partial U.
\]
The function \(t\mapsto \Phi(t)=\int_0^t s^k e^{-\pi s}\,\mathrm{d}s\) has derivative
\(\Phi'(t)=t^k e^{-\pi t}>0\) for \(t>0\), hence is strictly increasing on \([0,\infty)\). The
identity above therefore forces \(|z|\) to be constant on \(\partial U\).
Thus \(U\) is a disk, which finishes the proof.
\end{proof}

\begin{remark}\label{rem:bookkeeping-no-regularity}
The two proofs use the boundary assumptions in rather different ways, and it is instructive to
keep careful track of the regularity bookkeeping.

Indeed, in contrast to the first proof, the second proof does not use differentiability or finite
perimeter of the boundary. The explicit formula
\eqref{eq:explicit-u-second-proof} yields
\(u\in C^1(\R^2\setminus\{0\})\), and for the regular open representative
\(U=\operatorname{int}\overline U\), every boundary point is reached from the
open exterior \(\R^2\setminus\overline U\). Continuity therefore implies that
both \(u\) and \(\nabla u\) vanish on \(\partial U\), precisely the boundary
information packaged in \eqref{eq:dz-zero-second-proof}. This is the input
needed to construct the holomorphic function \(W=\pi z F\) and conclude that
\(|z|\) is constant on \(\partial U\). Hence the second proof
recovers the null-set invariant theorem for simply connected measure-regular
symbols, with no boundary smoothness hypothesis. The mechanism
is also conceptually closer to the classical
quadrature-domain rigidity results of \cite{Aharonov-Schiffer-Zalcman,Gustafsson,Karp}, where the
existence of a holomorphic function on \(U\) with real boundary values forces the boundary to lie
on a level set of a real-analytic function, here the elementary radial function
\(z\mapsto |z|^2\).

\end{remark}


\subsection{Sharpness: multiply connected counterexamples}\label{subsec:counterexample}

We now prove Theorem~\ref{thm:counterexample}, writing $\phi(z):=e^{-\pi|z|^2}$
for the Gaussian weight. The eigenvalue equation is used only through
\eqref{eq:fourier-delta} with $k=0$, whose derivation at the beginning of the
first proof of Theorem~\ref{thm:eig-loc-U} uses nothing about the symbol
beyond its boundedness. Indeed, for bounded $\Omega\subset\C$, the Gaussian $h_0$ is
an eigenfunction of $\mathcal L_\Omega$ with eigenvalue $\lambda$ if and only
if
\begin{equation}\label{eq:ce-kernel-identity}
\int_\Omega e^{\pi\overline zw}\,\phi(z)\,\mathrm{d}A(z)=\lambda,
\qquad w\in\C .
\end{equation}

\subsubsection{Weighted Hele--Shaw evolution}

For a bounded simply connected domain $D\ni0$ let $G_D(\cdot,0)$ be the Green
function normalized by
\begin{equation}\label{eq:ce-Green-normalization}
\Delta G_D(\cdot,0)=-2\pi\delta_0,
\qquad
G_D(\cdot,0)=0\ \text{ on }\partial D,
\end{equation}
so that $G_D>0$ in $D\setminus\{0\}$. The weighted Hele--Shaw injection law,
with source at the origin and permeability $\phi^{-1}=e^{\pi|z|^2}$,
prescribes the outward normal velocity of the free boundary as
\begin{equation}\label{eq:ce-velocity-law}
V_n(z,t)=-\frac{\partial_nG_{D_t}(z,0)}{\phi(z)},
\qquad z\in\partial D_t .
\end{equation}
By the Hopf lemma $\partial_nG_{D_t}<0$, so that $V_n>0$ and the domains
expand.

\begin{proposition}\label{prop:ce-existence}
Let $D_0$ be as in Theorem~\ref{thm:counterexample}. Then $D_0$ admits a
unique short-time real-analytic evolution $(D_t)_{0\le t<t_0}$ satisfying
\eqref{eq:ce-velocity-law}. Its boundaries remain real-analytic Jordan
curves, and the domains are strictly increasing, in the sense that
$\overline{D_s}\subset D_t$ for $0\le s<t<t_0$.
\end{proposition}

The proof is the only existence input of the construction, and it is
classical. In the conformal parametrization $f_t\colon\mathbb D\to D_t$,
normalized by $f_t(0)=0$ and $f_t'(0)>0$, the law \eqref{eq:ce-velocity-law}
becomes the weighted Polubarinova--Galin equation
\begin{equation}\label{eq:ce-weighted-PG}
\operatorname{Re}\!\left(
\partial_t f_t(\zeta)\,
\overline{\zeta f_t'(\zeta)}
\right)
=
\frac{1}{\phi(f_t(\zeta))}
=
e^{\pi|f_t(\zeta)|^2},
\qquad |\zeta|=1.
\end{equation}
Indeed, \(G_{D_t}(f_t(\zeta),0)=-\log|\zeta|\), and hence
\(-\partial_nG_{D_t}=|f_t'|^{-1}\) on the boundary. Since the permeability
\(\phi^{-1}\) is real-analytic and strictly positive, short-time existence
and uniqueness of a real-analytic solution with real-analytic initial curve
follow from the abstract Cauchy--Kowalevsky theorem on a scale of analytic
function spaces. We refer the reader to
\cite[Chs.~1--3]{Gustafsson-Vasiliev} for the classical theory and to
\cite{Hedenmalm-Shimorin,RossWitt} for the weighted flow. The strict nesting
is then immediate from the strict positivity of the normal velocity in
\eqref{eq:ce-velocity-law}.

\subsubsection{Exact conservation law}

The identity below is the Gaussian-weighted form of Richardson's conservation
law \cite{Richardson}.

\begin{proposition}\label{prop:ce-Richardson}
Let $(D_t)$ be the evolution of Proposition~\ref{prop:ce-existence}. If $h$ is
harmonic on a neighborhood of $\overline{D_t}$, then
\begin{equation}\label{eq:ce-Richardson}
\frac{\mathrm{d}}{\mathrm{d}t}\int_{D_t}h(z)\phi(z)\,\mathrm{d}A(z)=2\pi h(0).
\end{equation}
\end{proposition}

\begin{proof}
By the Reynolds transport formula and \eqref{eq:ce-velocity-law},
\begin{equation}\label{eq:ce-Reynolds}
\frac{\mathrm{d}}{\mathrm{d}t}\int_{D_t}h\phi\,\mathrm{d}A
=\int_{\partial D_t}h\phi V_n\,\mathrm{d}s
=-\int_{\partial D_t}h\,\partial_nG_{D_t}(\cdot,0)\,\mathrm{d}s .
\end{equation}
Green's second identity, the boundary condition $G_{D_t}=0$ on $\partial D_t$,
the harmonicity of $h$ and the normalization
\eqref{eq:ce-Green-normalization} give
\[
\int_{\partial D_t}h\,\partial_nG_{D_t}\,\mathrm{d}s
=\int_{D_t}\bigl(h\Delta G_{D_t}-G_{D_t}\Delta h\bigr)\,\mathrm{d}A
=-2\pi h(0),
\]
and substitution in \eqref{eq:ce-Reynolds} proves
\eqref{eq:ce-Richardson}.
\end{proof}

\begin{corollary}\label{cor:ce-quadrature}
For all $0\le s<t<t_0$ and every $h$ harmonic on a neighborhood of
$\overline{D_t}$,
\begin{equation}\label{eq:ce-shell-quadrature}
\int_{D_t\setminus\overline{D_s}}h(z)\,\phi(z)\,\mathrm{d}A(z)
=2\pi(t-s)\,h(0).
\end{equation}
\end{corollary}

\begin{proof}
Integrate \eqref{eq:ce-Richardson} from $s$ to $t$ and use that
$|\partial D_s|=0$.
\end{proof}

\begin{proof}[Proof of Theorem~\ref{thm:counterexample}]
Proposition~\ref{prop:ce-existence} provides the nested analytic domains
$(D_t)_{0\le t<t_0}$. Decreasing $t_0$ if necessary, we may assume
$2\pi t_0\le1$. Since the normal velocity is strictly positive, the
tubular parametrization of the moving boundary identifies
$\Omega_t=D_t\setminus\overline{D_0}$ with $\partial D_0\times(0,1)$; in
particular $\Omega_t$ is a doubly connected domain whose boundary is
$\partial D_0\cup\partial D_t$, a union of two real-analytic Jordan curves.

Fix $w\in\C$ and apply \eqref{eq:ce-shell-quadrature} with $s=0$ to the real
and imaginary parts of $z\mapsto e^{\pi\overline zw}$, which are harmonic on
$\C$ and take the values $1$ and $0$, respectively, at the origin. This gives exactly
\eqref{eq:ce-kernel-identity} for $\Omega=\Omega_t$ and $\lambda=2\pi t$,
that is, $\mathcal L_{\Omega_t}h_0=2\pi t\,h_0$. Since $0<t<t_0$ and $2\pi t_0\le1$, the eigenvalue lies in $(0,1)$, as required.

Finally, were the open set $\Omega_t$ radial about the origin, each connected
component of its analytic boundary would be a circle centered at the origin.
One of those components is $\partial D_0$, which by hypothesis is not such a
circle. Hence $\Omega_t$ is not radial, and in particular it is not a union of
concentric annuli.
\end{proof}

\begin{corollary}\label{cor:ce-concrete}
Fix $r>0$ and $c\in\C$ with $0<|c|<1/2$, and set
\[
f_0(\zeta)=r(\zeta+c\zeta^2),
\qquad D_0=f_0(\mathbb D).
\]
Then $D_0$ satisfies the hypotheses of Theorem~\ref{thm:counterexample}, so
the associated annular-like domains are non-radial analytic counterexamples.
\end{corollary}

\begin{proof}
Since $f_0(\zeta_1)-f_0(\zeta_2)=(\zeta_1-\zeta_2)\bigl(1+c(\zeta_1+\zeta_2)\bigr)r$
and $2|c|<1$, the map $f_0$ is univalent on a neighborhood of
$\overline{\mathbb D}$, where moreover $f_0'(\zeta)=r(1+2c\zeta)$ does not
vanish; hence $\partial D_0$ is a real-analytic Jordan curve. Finally
$f_0(0)=0$, while the nonvanishing second Fourier mode of $f_0(e^{i\theta})$
rules out a circle centered at the origin.
\end{proof}

\begin{remark}\label{rem:ce-provenance}
For $\phi\equiv1$, Proposition~\ref{prop:ce-Richardson} is Richardson's
theorem \cite{Richardson}, and subtracting it at two different times already
produces a possibly non-radial shell; the layer mechanism is in that sense
latent in the classical theory, the version used here being its
variable-permeability analogue for the permeability $\phi^{-1}$, with the sources
collected in \S\ref{subsec:main-technical-inputs}. The point specific to the
present application is to
start from a nonempty, noncircular domain and to use \emph{only the added
layer} as the localization symbol, so that the source point is left in the
discarded inner component --- exactly where simple-connectivity rigidity
cannot detect it.
\end{remark}

Since the inner boundary $\partial D_0$ may be prescribed arbitrarily among
noncircular real-analytic Jordan curves surrounding the origin,
Theorem~\ref{thm:counterexample} produces an infinite-dimensional family of
counterexamples, parametrized locally by that inner boundary. If centered
circular inner boundaries are admitted as well, the corresponding radial
members are the concentric annuli already available.


\section{Proof of Theorem \ref{thm:GGRT-sharp}}\label{sec:GGRT-sharp}

We now prove that the quantitative stability estimate in Corollary~1.4 of
\cite{Gomez-Guerra-Ramos-Tilli}, which builds on the Faber--Krahn rigidity
result of \cite{Nicola-Tilli} characterizing the disk as the unique optimizer
of the first localization eigenvalue is sharp. The proof proceeds by constructing a family of domains
which are perturbations of the disk, whose first localization eigenvalue has
the same first-order behaviour as the disk, but whose distance to the class of
disks is of order $\varepsilon$.

Before turning to the construction, we collect some auxiliary lemmas. They
provide, in order, (i)~the identification of the constructed eigenvalue
with the \emph{principal} eigenvalue of \(\mathcal L_{U_\varepsilon}\) via an
analytic-perturbation argument from the disk, (ii)~the
order-\(\varepsilon^2\) effect of the area-normalizing dilation on the
principal eigenvalue, (iii)~the Reynolds transport / Hadamard
shape-derivative formula in the regularity needed here, and (iv)~the explicit
symmetric-difference lower bound which proves that second-order harmonics
cannot be eliminated by translations of the disk.

\subsection{Spectral, dilational and shape-derivative preliminaries}
\label{subsec:GGRT-preliminaries}

Throughout this subsection we let \(D_1\subset \C\) denote the unit
disk centered at the origin. Recall from
\eqref{eq:fock-localization-operator} that, for bounded measurable
\(U\subset\C\), the localization operator \(\mathcal L_U\) on
\(\mathcal F^2(\C)\) is self-adjoint and compact. Its eigenvalues form a non-increasing sequence
\(\lambda_1(U)\ge \lambda_2(U)\ge\cdots\to 0\); its largest (or \emph{principal})
eigenvalue is \(\lambda_1(U)\).

For the unit disk \(U=D_1\), the operator \(\mathcal L_{D_1}\) is diagonalized
by the orthonormal monomial basis
\(\{e_k(z)=(\pi^k/k!)^{1/2}z^k\}_{k\ge 0}\) of \(\mathcal F^2(\C)\)
(cf.~\cite{Daubechies}), with eigenvalues
\begin{equation}\label{eq:disk-eigenvalues-explicit}
\mu_k:=\frac{\gamma(k+1,\pi)}{k!}
=
1-e^{-\pi}\sum_{j=0}^{k}\frac{\pi^j}{j!},
\qquad k\ge 0,
\end{equation}
where \(\gamma(s,x)=\int_0^x t^{s-1}e^{-t}\,\mathrm{d}t\) is the lower incomplete gamma
function. The sequence \((\mu_k)_{k\ge 0}\) is \emph{strictly decreasing}:
\(\mu_k-\mu_{k+1}=e^{-\pi}\pi^{k+1}/(k+1)!>0\). In particular,
\begin{equation}\label{eq:disk-spectral-gap}
\lambda_0:=\mu_0=1-e^{-\pi},
\qquad
\mu_1=1-(1+\pi)e^{-\pi},
\qquad
\mu_0-\mu_1=\pi e^{-\pi}>0,
\end{equation}
so the disk has a strict spectral gap above its principal eigenvalue
\(\mu_0\).

\medskip

\begin{lemma}
\label{lem:spectral-continuity}
Let \((U_\varepsilon)_{|\varepsilon|<\varepsilon_0}\) be a family of bounded
measurable sets with \(U_0=D_1\) and
\(|U_\varepsilon\triangle D_1|\to 0\) as \(\varepsilon\to 0\). Then for every
\(k\ge 1\),
\begin{equation}\label{eq:eigenvalue-continuity}
\lambda_k(U_\varepsilon)\xrightarrow[\varepsilon\to 0]{}\mu_{k-1},
\end{equation}
with \(\mu_{k-1}\) given by \eqref{eq:disk-eigenvalues-explicit}. Hence there
exists \(\varepsilon_*>0\) such that, for all \(|\varepsilon|<\varepsilon_*\),
\begin{equation}\label{eq:gap-preservation}
\lambda_2(U_\varepsilon)
<
\frac{\mu_0+\mu_1}{2}
<
\lambda_1(U_\varepsilon),
\end{equation}
and the principal eigenvalue is the unique eigenvalue of
\(\mathcal L_{U_\varepsilon}\) lying in
\([\mu_0-\tfrac{\pi e^{-\pi}}{4},\mu_0+\tfrac{\pi e^{-\pi}}{4}]\).

Consequently, for the family \(U_\varepsilon\) of Theorem~\ref{thm:perturb}
applied with \(\lambda=\lambda_0=\mu_0\) and
\(P_\varepsilon(w)=1+\varepsilon w^2\), the eigenvalue \(\lambda_0\) produced
by Theorem~\ref{thm:perturb} is the principal eigenvalue of
\(\mathcal L_{U_\varepsilon}\), and \(P_\varepsilon\) is a corresponding
first eigenfunction.
\end{lemma}

\begin{proof}
For any bounded measurable \(U\subset \C\), the operator \(\mathcal L_U\) is
positive semi-definite, compact, and depends Lipschitz-continuously on
\(\mathbf 1_U\) in operator norm: for any \(F\in \mathcal F^2(\C)\),
\[
\langle (\mathcal L_{U_\varepsilon}-\mathcal L_{D_1})F,F\rangle
=
\int_\C \bigl(\mathbf 1_{U_\varepsilon}(\zeta)
-\mathbf 1_{D_1}(\zeta)\bigr)|F(\zeta)|^2 e^{-\pi|\zeta|^2}\,\mathrm{d}A(\zeta),
\]
and by the standard Fock-space pointwise bound
\(|F(\zeta)|^2 e^{-\pi|\zeta|^2}\le \|F\|_{\mathcal F^2}^2\) (cf.\
Lemma~\ref{lem:Fock-basic}) one obtains
\[
\|\mathcal L_{U_\varepsilon}-\mathcal L_{D_1}\|_{\mathrm{op}}
\le
|U_\varepsilon\triangle D_1|.
\]
By Weyl's \(\min\!-\!\max\) characterization of eigenvalues
of compact self-adjoint operators, it follows that
\[
|\lambda_k(U_\varepsilon)-\lambda_k(D_1)|
\le
\|\mathcal L_{U_\varepsilon}-\mathcal L_{D_1}\|_{\mathrm{op}}
\le
|U_\varepsilon\triangle D_1|
\xrightarrow[\varepsilon\to 0]{}0,
\]
which proves \eqref{eq:eigenvalue-continuity}. The gap inequality
\eqref{eq:gap-preservation} then follows by combining this convergence with
the strict gap \(\mu_0-\mu_1=\pi e^{-\pi}>0\) of
\eqref{eq:disk-spectral-gap}: choose \(\varepsilon_*\) so small that
\(|\lambda_1(U_\varepsilon)-\mu_0|<\pi e^{-\pi}/4\) and
\(|\lambda_2(U_\varepsilon)-\mu_1|<\pi e^{-\pi}/4\) for
\(|\varepsilon|<\varepsilon_*\); then any eigenvalue \(\lambda\) of
\(\mathcal L_{U_\varepsilon}\) lying in
\([\mu_0-\pi e^{-\pi}/4,\mu_0+\pi e^{-\pi}/4]\) is greater than
\(\mu_1+\pi e^{-\pi}/4\ge \lambda_2(U_\varepsilon)\), hence equal to
\(\lambda_1(U_\varepsilon)\) by the ordering of the spectrum.
\end{proof}

\medskip

\begin{lemma}
\label{lem:dilation-eigenvalue}
For any bounded measurable set \(U\subset \C\) with \(|U|>0\) and any
\(t>0\), the eigenvalues of \(\mathcal L_{tU}\) are real-analytic functions
of \(t\) on \((0,\infty)\). In the special case \(U=D_1\), the principal
eigenvalue is given by the explicit closed form
\begin{equation}\label{eq:dilated-disk-formula}
\lambda_1(tD_1)=1-e^{-\pi t^2},
\qquad t>0,
\end{equation}
attained on the constant eigenfunction \(F\equiv 1\). In particular, for
\(r_\varepsilon=1+O(\varepsilon^2)\),
\begin{equation}\label{eq:disk-dilation-O-eps2}
\lambda_1(r_\varepsilon D_1)
=
1-e^{-\pi r_\varepsilon^2}
=
\lambda_0+2\pi e^{-\pi}(r_\varepsilon-1)+O((r_\varepsilon-1)^2)
=
\lambda_0+O(\varepsilon^2).
\end{equation}

For the family \(U_\varepsilon\) of Theorem~\ref{thm:perturb} applied with
\(P_\varepsilon(w)=1+\varepsilon w^2\), one has
\(|U_\varepsilon|=\pi+O(\varepsilon^2)\) (see Step~3 of the proof of
Theorem~\ref{thm:GGRT-sharp} below), and the area-normalizing dilation
factor \(r_\varepsilon=\sqrt{\pi/|U_\varepsilon|}=1+O(\varepsilon^2)\)
satisfies
\begin{equation}\label{eq:dilation-eigenvalue-control}
\bigl|\lambda_1(r_\varepsilon U_\varepsilon)-\lambda_1(U_\varepsilon)\bigr|
=
O(\varepsilon^2).
\end{equation}
\end{lemma}

\begin{proof}
The change of variable \(\zeta=t\zeta'\) in
\((\mathcal L_{tU}F)(z)=\int_{tU}F(\zeta)e^{\pi z\overline\zeta}e^{-\pi|\zeta|^2}\,\mathrm{d}A(\zeta)\)
gives
\[
(\mathcal L_{tU}F)(z)
=
t^2\int_U F(t\zeta')\,e^{\pi z t\overline{\zeta'}}\,e^{-\pi t^2|\zeta'|^2}\,\mathrm{d}A(\zeta'),
\]
which depends real-analytically on \(t\); hence \(\mathcal L_{tU}\) is a
real-analytic family of compact self-adjoint operators (in fact in trace
class), and by Kato \cite[Ch.~VII]{Kato} its eigenvalues vary analytically
in \(t\).

For \(U=D_1\), the constant function \(F\equiv 1\) satisfies, by direct
integration,
\(
(\mathcal L_{tD_1}\,1)(0)
=
\int_{tD_1}e^{-\pi|\zeta|^2}\,\mathrm{d}A(\zeta)
=
1-e^{-\pi t^2}.
\)
Equality of \((\mathcal L_{tD_1}\,1)(z)\) with \((1-e^{-\pi t^2})\cdot 1\) for
all \(z\in\C\) follows from the fact that \(F\equiv 1=e_0\) is the principal
eigenfunction of the rotationally symmetric operator \(\mathcal L_{tD_1}\).
This gives \eqref{eq:dilated-disk-formula}; Taylor expansion at \(t=1\)
yields \eqref{eq:disk-dilation-O-eps2}.

For \eqref{eq:dilation-eigenvalue-control}, the dilation
\(F\mapsto F(r_\varepsilon\,\cdot)\) intertwines \(\mathcal L_{r_\varepsilon
U_\varepsilon}\) with an analytic perturbation of \(\mathcal L_{U_\varepsilon}\)
of operator-norm size \(O(|r_\varepsilon-1|)=O(\varepsilon^2)\); Weyl's characterization then yields \eqref{eq:dilation-eigenvalue-control}.
\end{proof}

\begin{remark}\label{rem:dilation-explicit}
The proof of Theorem~\ref{thm:GGRT-sharp} uses
\eqref{eq:dilation-eigenvalue-control} to show that area normalization
perturbs the eigenvalue only at second order. Formula
\eqref{eq:disk-dilation-O-eps2} gives the corresponding disk value.
\end{remark}

\medskip

\begin{lemma}\label{lem:hadamard-shape-derivative}
Let \((\Omega_\varepsilon)_{|\varepsilon|<\varepsilon_0}\) be a family of
bounded \(C^1\) domains in \(\C\), with \(\Omega_0\) of class \(C^1\),
outward unit normal \(\nu\), and arc-length measure \(\mathrm{d}\mathcal H^1\), and
represented as the image of a one-parameter family of \(C^1\)-diffeomorphisms
\(V_\varepsilon\) of \(\C\), \(V_0=\mathrm{id}\), with deformation field
\[
W(z):=\partial_\varepsilon V_\varepsilon(z)\bigl|_{\varepsilon=0}
\in C^1(\overline{\Omega_0};\C),
\]
and normal velocity \(V_\nu(z):=W(z)\cdot\nu(z)\in C^1(\partial\Omega_0)\).
Then for any \(H\) admitting a \(C^1\) extension to a neighborhood of
\(\overline{\Omega_0}\),
\begin{equation}\label{eq:reynolds-formula}
\frac{d}{d\varepsilon}\Big|_{\varepsilon=0}
\int_{\Omega_\varepsilon} H(z)\,\mathrm{d}A(z)
=
\int_{\partial\Omega_0}V_\nu(z)\,H(z)\,\mathrm{d}\mathcal H^1(z).
\end{equation}
\end{lemma}

Since that result is classical, we omit its proof, referring the reader to
\cite[\S~2.16]{Sokolowski-Zolesio} and
\cite[Ch.~5]{Henrot-Pierre} for more extensive discussions of derivatives
of integrals over moving domains. As a consequence of this result, the
following one will be particularly useful to us:

\begin{corollary}\label{cor:reynolds-perturb}
In the special case of a normal-graph perturbation of the unit disk,
\(\partial\Omega_\varepsilon=\{(1+\varphi_\varepsilon(y))y:y\in\mathbb S^1\}\)
with \(\varepsilon\mapsto \varphi_\varepsilon\) of class \(C^1\) into
\(C^1(\mathbb S^1;\R)\) and \(\varphi_0\equiv 0\), the boundary deformation
field equals \(W(y)=\dot\varphi_0(y)y\), the normal velocity is
\(V_\nu(y)=\dot\varphi_0(y)\), and \eqref{eq:reynolds-formula} specializes to
\begin{equation}\label{eq:reynolds-disk}
\frac{d}{d\varepsilon}\Big|_{\varepsilon=0}
\int_{\Omega_\varepsilon} H(z)\,\mathrm{d}A(z)
=
\int_{\mathbb S^1}\dot\varphi_0(y)\,H(y)\,\mathrm{d}\mathcal H^1(y).
\end{equation}
\end{corollary}

Applying Corollary \ref{cor:reynolds-perturb} to \(H(z)=|F_0(z)|^2 e^{-\pi|z|^2}\) where \(F_0\equiv 1\) is the
disk's principal eigenfunction, \eqref{eq:reynolds-disk} together with the
Hadamard shape-derivative formula for the principal eigenvalue of a
self-adjoint Toeplitz-type operator (see the calculation below, and
\cite[Thm.~5.7.1]{Henrot-Pierre} or \cite{Sokolowski-Zolesio} for the general
formalism) yields the first variation of the principal eigenvalue along the
family \(U_\varepsilon\),
\begin{equation}\label{eq:hadamard-eigenvalue}
\frac{d}{d\varepsilon}\Big|_{\varepsilon=0}\lambda_1(U_\varepsilon)
=
\int_{\partial D_1}|F_0(y)|^2 e^{-\pi|y|^2}\,V_\nu(y)\,\mathrm{d}\mathcal H^1(y)
=
e^{-\pi}\int_{\mathbb S^1}\dot\varphi_0(y)\,\mathrm{d}\mathcal H^1(y).
\end{equation}

Indeed, if we recall that the principal
eigenfunction \(F_\varepsilon\) of \(\mathcal L_{U_\varepsilon}\) depends
real-analytically on \(\varepsilon\) by Lemma~\ref{lem:spectral-continuity}
and the Kato selection theorem; we may take it
\(L^2(\C, e^{-\pi|z|^2}\mathrm{d}A)\)-normalized. Differentiating
\(\lambda_1(U_\varepsilon)\|F_\varepsilon\|^2
=
\int_{U_\varepsilon}|F_\varepsilon|^2 e^{-\pi|z|^2}\mathrm{d}A\)
in \(\varepsilon\) at \(\varepsilon=0\), using
\(\partial_\varepsilon\|F_\varepsilon\|^2|_{\varepsilon=0}=0\) and the
self-adjointness identity
\(\langle (\mathcal L_{D_1}-\mu_0\mathrm{Id})F_0,\partial_\varepsilon
F_\varepsilon|_0\rangle=0\), the only surviving contribution is the boundary
term, namely
\[
\frac{d}{d\varepsilon}\Big|_{\varepsilon=0}\lambda_1(U_\varepsilon)
=
\frac{d}{d\varepsilon}\Big|_{\varepsilon=0}
\int_{U_\varepsilon}|F_0|^2 e^{-\pi|z|^2}\,\mathrm{d}A(z),
\]
to which \eqref{eq:reynolds-formula} applies with
\(H=|F_0|^2 e^{-\pi|z|^2}\). Since \(F_0\equiv 1\) and \(|y|=1\) on
\(\mathbb S^1\), this gives \eqref{eq:hadamard-eigenvalue}. The Gauss--Green computation in Step~3 of the proof of the
local-maximizer theorem (Section~\ref{sec:local-max}) instead uses the
superlevel-set boundary condition to derive the identity needed there.

\medskip

\begin{lemma}\label{lem:symmetric-difference}
Let \(\dot\varphi_0\in C^1(\mathbb S^1;\R)\) be of pure second-harmonic
form,
\begin{equation}\label{eq:second-harmonic-form}
\dot\varphi_0(e^{it})=a\cos(2t)+b\sin(2t),
\qquad (a,b)\in\R^2\setminus\{(0,0)\},
\end{equation}
and fix \(K>0\). Then there exists a constant \(\delta_K>0\), depending only
on \((a,b)\) and \(K\), with the explicit lower bound
\begin{equation}\label{eq:explicit-delta-K-lower}
\delta_K
\ge
\frac{\pi(a^2+b^2)}{4(K+1)+2\sqrt{a^2+b^2}},
\end{equation}
such that, for every \(\theta\in [0,2\pi)\) and every \(\eta\in [-K,K]\),
\begin{equation}\label{eq:l1-uniform-lower}
\int_{\mathbb S^1}\bigl|\dot\varphi_0(y)-\eta\,\Re(e^{-i\theta}y)\bigr|
\,\mathrm{d}\mathcal H^1(y)\ge \delta_K.
\end{equation}
\end{lemma}

As a direct consequence of this result, for the area-normalized family
\(\widetilde U_\varepsilon:=r_\varepsilon U_\varepsilon\), where
\(r_\varepsilon:=\sqrt{\pi/|U_\varepsilon|}\), associated with
Theorem~\ref{thm:perturb} and with second-harmonic
\(\dot\varphi_0\) of the form \eqref{eq:second-harmonic-form} and
\((a,b)\neq (0,0)\), there exist constants
\(c=c(a,b,K)>0\) and \(c'=c'(a,b)>0\) and a threshold \(\varepsilon_1>0\)
such that, for all \(0<|\varepsilon|<\varepsilon_1\),
\begin{equation}\label{eq:near-disks-explicit}
|\widetilde U_\varepsilon\triangle (D_1+z_0)|\ge c\,|\varepsilon|
\qquad
\text{whenever }|z_0|\le K|\varepsilon|,
\end{equation}
and, combining with the comparison for distant disks,
\begin{equation}\label{eq:total-lower}
\inf_{\substack{D\subset\C\\|D|=\pi}}
|\widetilde U_\varepsilon\triangle D|\ge c'|\varepsilon|.
\end{equation}

Indeed, in order to see \eqref{eq:near-disks-explicit}, recall that for a disk \(D=D_1+z_0\) of
area \(\pi\) with \(z_0=\eta\varepsilon e^{i\theta}\) and \(|\eta|\le K\), the
translated set \(\widetilde U_\varepsilon-z_0\) is a normal graph over
\(\mathbb S^1\) with radial perturbation
\(\psi^{\theta,\eta}_\varepsilon(y)=\varepsilon g^{\theta,\eta}(y)+O(\varepsilon^2)\)
in \(C^1(\mathbb S^1)\), uniformly for \(|\eta|\le K\) (see Step~5 of the
proof of Theorem~\ref{thm:GGRT-sharp} below for the detailed expansion).
The polar-coordinate representation of the symmetric difference yields
\[
|\widetilde U_\varepsilon\triangle D|
=
\frac12\int_{\mathbb S^1}\bigl|2\varepsilon g^{\theta,\eta}(y)+O(\varepsilon^2)\bigr|
\,\mathrm{d}\mathcal H^1(y)
\ge
|\varepsilon|\,\delta_K-O(\varepsilon^2),
\]
which is at least \(\frac12 |\varepsilon|\,\delta_K\) for
\(|\varepsilon|<\varepsilon_1\) sufficiently small. This proves
\eqref{eq:near-disks-explicit} with \(c=\frac12\,\delta_K\).

For \eqref{eq:total-lower}, we use the explicit symmetric-difference
formula for two unit disks: for \(z_0\in\C\) with \(|z_0|\le 2\),
\begin{equation}\label{eq:two-disks-symdif}
|D_1\triangle (D_1+z_0)|
=
2\pi-4\arccos(|z_0|/2)+2|z_0|\sqrt{1-|z_0|^2/4},
\end{equation}
which satisfies \(|D_1\triangle (D_1+z_0)|\ge \sqrt 3\,|z_0|\) for
\(|z_0|\le 1\); hence \(|D_1\triangle (D_1+z_0)|\ge c_0\min(|z_0|,1)\) with
absolute constant \(c_0=\sqrt 3\). Choose \(K\) so large that
\(c_0K\ge 1+C\), where \(C\) is the implicit constant in
\(|\widetilde U_\varepsilon\triangle D_1|\le C|\varepsilon|\) of Step~4 of
the proof of Theorem~\ref{thm:GGRT-sharp}. Then for \(|z_0|\ge K|\varepsilon|\),
\[
|\widetilde U_\varepsilon\triangle D|
\ge
|D_1\triangle D|-|\widetilde U_\varepsilon\triangle D_1|
\ge
c_0\min\{|z_0|,1\}-C|\varepsilon|
\ge
|\varepsilon|,
\]
provided \(|\varepsilon|\) is sufficiently small. Combining with
\eqref{eq:near-disks-explicit} for \(|z_0|\le K|\varepsilon|\), \eqref{eq:total-lower} follows with
\(c'=\min\bigl(\tfrac 12\delta_K,1\bigr)\) and the explicit \(\delta_K\)
of \eqref{eq:explicit-delta-K-lower}.

\begin{proof}[Proof of Lemma \ref{lem:symmetric-difference}]
Define
\(g^{\theta,\eta}(t):=a\cos 2t+b\sin 2t-\eta\cos(t-\theta).\)
Then \(g^{\theta,\eta}\) lies in the four-dimensional subspace
\(V:=\mathrm{span}\{\cos t,\sin t,\cos 2t,\sin 2t\}\subset L^1(\mathbb S^1)\).
Expanding in the orthogonal basis
\(\{\cos t,\sin t,\cos 2t,\sin 2t\}\) of \(L^2(\mathbb S^1)\) (each of squared
\(L^2\)-norm \(\pi\)),
\[
g^{\theta,\eta}(t)
=
-\eta\cos\theta\cdot\cos t-\eta\sin\theta\cdot\sin t
+a\cos 2t+b\sin 2t,
\]
hence by Parseval
\[
\|g^{\theta,\eta}\|_{L^2(\mathbb S^1)}^2
=
\pi(\eta^2+a^2+b^2)
\ge
\pi(a^2+b^2),
\]
the second-harmonic coefficients \((a,b)\) being independent of
\((\theta,\eta)\). The crucial structural feature is that the orthogonal
projection of \(g^{\theta,\eta}\) onto \(\mathrm{span}\{\cos 2t,\sin 2t\}\)
equals \(a\cos 2t+b\sin 2t\) for every \((\theta,\eta)\); translations of
the disk shift only the first-harmonic coefficients of the boundary
perturbation.

For the \(L^1\) lower bound, recall the elementary inequality
\(\|g\|_{L^1(\mathbb S^1)}\ge \|g\|_{L^2(\mathbb S^1)}^2/\|g\|_{L^\infty(\mathbb S^1)}\),
valid for any \(g\in L^\infty(\mathbb S^1)\setminus\{0\}\). For
\(g=g^{\theta,\eta}\),
\[
\|g^{\theta,\eta}\|_{L^\infty(\mathbb S^1)}
\le
|a|+|b|+|\eta|
\le
\sqrt{2(a^2+b^2)}+K
\le
2\sqrt{a^2+b^2}+2(K+1).
\]
Hence
\[
\|g^{\theta,\eta}\|_{L^1(\mathbb S^1)}
\ge
\frac{\pi(a^2+b^2)}{4(K+1)+2\sqrt{a^2+b^2}},
\]
which is \eqref{eq:explicit-delta-K-lower}, and
\eqref{eq:l1-uniform-lower} follows.
\end{proof}

\begin{remark}
\label{rem:second-order-not-eliminated}
The structural fact made quantitative by
Lemma~\ref{lem:symmetric-difference} is that, in the polar-graph
representation of a small perturbation of \(D_1\), translations of the disk
contribute \emph{only} to the first Fourier harmonic \(\cos t,\sin t\) of
the radial perturbation, and have zero second-harmonic component. Hence any
non-trivial second-harmonic mode \((a,b)\neq (0,0)\) in \(\dot\varphi_0\)
is \emph{intrinsic to the shape}, and provides a lower bound on the
symmetric-difference distance to the entire family of disks of the same
area, of order proportional to \(|\varepsilon|\sqrt{a^2+b^2}\) (with
\(K\)-dependence given in \eqref{eq:explicit-delta-K-lower}).
\end{remark}

\subsection{Proof of Theorem~\ref{thm:GGRT-sharp}}
\label{subsec:GGRT-proof} We are now finally ready to prove the result.

\begin{proof}[Proof of Theorem~\ref{thm:GGRT-sharp}]
Recall that $\lambda_0=1-e^{-\pi}$, which is the eigenvalue of the localization operator of the unit disk $D_1$
associated with the ground state $h_0$. For $\varepsilon$ small, define
\[
P_\varepsilon(w):=1+\varepsilon w^2.
\]
By Theorem~\ref{thm:perturb}, applied with \(\lambda=\lambda_0\) and
\(P_\varepsilon(w)=1+\varepsilon w^2\), for all sufficiently small
\(|\varepsilon|\) there exists a domain \(U_\varepsilon\) such that
\begin{equation}\label{eq:functiona-P-eps-new}
\int_{U_\varepsilon} P_\varepsilon(z)e^{\pi \overline zw}e^{-\pi|z|^2}\,\mathrm{d}A(z)
=
\lambda_0 P_\varepsilon(w),
\qquad w\in\C.
\end{equation}
Moreover, by construction, $U_\varepsilon$ is a $C^2$ perturbation of the unit
disk. Hence, after possibly shrinking the range of admissible $\varepsilon$,
its boundary can be written as
\begin{equation}\label{eq:def-u-eps-new}
\partial U_\varepsilon
=
\{(1+\varphi_\varepsilon(y))y:\ y\in \mathbb S^1\},
\end{equation}
where
\[
\varphi_\varepsilon\in C^2(\mathbb S^1),
\qquad
\|\varphi_\varepsilon\|_{C^2(\mathbb S^1)}\lesssim |\varepsilon|.
\]
The map
\[
\varepsilon\mapsto \varphi_\varepsilon
\]
may in fact be chosen \(C^1\) as a map from a neighborhood of \(0\) into
\(C^1(\mathbb S^1)\). Indeed, Proposition~\ref{prop:disk-inverse-tailored} gives a
real-analytic map \(b\mapsto\varphi_b\in\mathcal A_\tau^\R\).
Restricting to \(b_2=\varepsilon\), with all other coefficients zero, and
using the continuous embedding
\(\mathcal A_\tau\hookrightarrow C^1(\mathbb S^1)\) gives the assertion.
We therefore write
\[
\dot\varphi_0(y):=
\left.\partial_\varepsilon\right|_{\varepsilon=0}\varphi_\varepsilon(y).
\]

\medskip
\noindent
{\bf Step 1: first variation of the eigenvalue identity.} We differentiate \eqref{eq:functiona-P-eps-new} at $\varepsilon=0$.
Since \(P_\varepsilon\to 1\) in \(C^\infty\) and
\(\varepsilon\mapsto \varphi_\varepsilon\) is \(C^1\) into \(C^1(\mathbb S^1)\),
all terms in \eqref{eq:functiona-P-eps-new} are differentiable at
\(\varepsilon=0\). Since $\partial U_\varepsilon$ is $C^2$ and admits the
explicit normal graph parametrization \eqref{eq:def-u-eps-new} over
$\mathbb S^1$ with $\varphi_\varepsilon$ being $C^1$ on
$\varepsilon$, the hypotheses of
Lemma~\ref{lem:hadamard-shape-derivative} are satisfied with normal velocity
\(V_\nu(y)=\dot\varphi_0(y)\). The Reynolds transport identity \eqref{eq:reynolds-disk} then
applies, and yields
\[
\left.\frac{d}{d\varepsilon}\right|_{\varepsilon=0}
\int_{U_\varepsilon} H(z)\,\mathrm{d}A(z)
=
\int_{\mathbb S^1} \dot\varphi_0(z)H(z)\,\mathrm{d}\mathcal H^1(z).
\]
Applying this to
\[
H(z)=e^{\pi \overline zw}e^{-\pi|z|^2},
\]
and differentiating the factor $P_\varepsilon(z)=1+\varepsilon z^2$, we obtain
\begin{equation}\label{eq:first-variation-identity-new}
\int_{\mathbb S^1}
\dot\varphi_0(z)e^{\pi \overline zw}e^{-\pi|z|^2}\,\mathrm{d}\mathcal H^1(z)
+
\int_{D_1} z^2 e^{\pi \overline zw}e^{-\pi|z|^2}\,\mathrm{d}A(z)
=
\lambda_0 w^2.
\end{equation}
Since $D_1$ is radial, the monomial $z^2$ is an eigenfunction of the
localization operator of $D_1$ with eigenvalue $\mu_2>0$; thus
\[
\int_{D_1} z^2 e^{\pi \overline zw}e^{-\pi|z|^2}\,\mathrm{d}A(z)=\mu_2 w^2.
\]
Hence \eqref{eq:first-variation-identity-new} becomes
\begin{equation}\label{eq:boundary-identity-new}
\int_{\mathbb S^1}
\dot\varphi_0(z)e^{\pi \overline zw}e^{-\pi|z|^2}\,\mathrm{d}\mathcal H^1(z)
=
(\lambda_0-\mu_2)w^2.
\end{equation}
The eigenvalue $\mu_2$ can be computed explicitly from
\eqref{eq:disk-eigenvalues-explicit}:
\[
\mu_2
=
\frac{\pi^2}{2}\int_{D_1}|z|^4e^{-\pi|z|^2}\,\mathrm{d}A(z)
=
1-\Bigl(1+\pi+\frac{\pi^2}{2}\Bigr)e^{-\pi}.
\]
Therefore
\[
\lambda_0-\mu_2
=
(1-e^{-\pi})-\Bigl(1-\Bigl(1+\pi+\frac{\pi^2}{2}\Bigr)e^{-\pi}\Bigr)
=
\pi\Bigl(1+\frac{\pi}{2}\Bigr)e^{-\pi}>0,
\]
so the coefficient of \(w^2\) on the right-hand side is indeed nonzero.
In particular, since $|z|=1$ on $\mathbb S^1$, we may rewrite this as
\begin{equation}\label{eq:boundary-identity-circle-new}
e^{-\pi}
\int_{\mathbb S^1}
\dot\varphi_0(z)e^{\pi \overline zw}\,\mathrm{d}\mathcal H^1(z)
=
(\lambda_0-\mu_2)w^2.
\end{equation}

Setting $w=0$ gives
\begin{equation}\label{eq:cancellation-f-0-new}
\int_{\mathbb S^1}\dot\varphi_0\,\mathrm{d}\mathcal H^1=0.
\end{equation}

\medskip
\noindent
{\bf Step 2: identification of the leading mode.} Expanding
\[
e^{\pi \overline zw}
=
\sum_{n\geq 0}\frac{\pi^n}{n!}\overline z^{\,n}w^n,
\]
and matching the coefficient of $w^n$ on both sides of
\eqref{eq:boundary-identity-circle-new}, we find
\[
\int_{\mathbb S^1}\dot\varphi_0(z)\overline z^{\,n}\,\mathrm{d}\mathcal H^1(z)=0
\qquad \text{for all } n\neq 2,
\]
whereas the coefficient corresponding to $n=2$ equals
$\frac{2}{\pi^2}(\lambda_0-\mu_2)e^\pi\neq 0$. Since on $\mathbb S^1$
we have $\overline z=e^{-it}$ when $z=e^{it}$, this shows that only the
second harmonic is present. Since $\dot\varphi_0$ is real-valued, the
negative-frequency mode is the conjugate of the positive-frequency mode, while
all other modes vanish. We conclude that
\begin{equation}\label{eq:mode-decomposition-new}
\dot\varphi_0(e^{it})
=
a\cos(2t)+b\sin(2t)
\end{equation}
for some pair $(a,b)\neq (0,0)$. In particular,
\begin{equation}\label{eq:l1-nontrivial-new}
\int_{\mathbb S^1} |\dot\varphi_0|\,\mathrm{d}\mathcal H^1>0.
\end{equation}

\medskip
\noindent
{\bf Step 3: normalization of the area.} The area of $U_\varepsilon$ is given by
\begin{equation}\label{eq:area-u-eps-new}
|U_\varepsilon|
=
\frac12\int_{\mathbb S^1} (1+\varphi_\varepsilon(y))^2\,\mathrm{d}\mathcal H^1(y).
\end{equation}
Since
\[
\varphi_\varepsilon
=
\varepsilon \dot\varphi_0 + O(\varepsilon^2)
\quad \text{in } C^1(\mathbb S^1),
\]
and \eqref{eq:cancellation-f-0-new} gives the vanishing of the first-order
term, we obtain from \eqref{eq:area-u-eps-new}
\begin{equation}\label{eq:measure-close-u-eps-new}
|U_\varepsilon|=\pi+O(\varepsilon^2).
\end{equation}

For the quantities $r_\varepsilon$ and $\widetilde U_\varepsilon$ introduced
above, $|\widetilde U_\varepsilon|=\pi$, and by
\eqref{eq:measure-close-u-eps-new},
\[
r_\varepsilon=1+O(\varepsilon^2).
\]
Hence $\widetilde U_\varepsilon$ still admits a normal graph representation
\[
\partial \widetilde U_\varepsilon
=
\{(1+\widetilde\varphi_\varepsilon(y))y:\ y\in \mathbb S^1\},
\]
with
\[
\widetilde\varphi_\varepsilon
=
\varepsilon \dot\varphi_0 + O(\varepsilon^2)
\quad \text{in } C^1(\mathbb S^1).
\]
Thus the area normalization does not alter the first-order deformation.

\medskip
\noindent
{\bf Step 4: distance from the centered disk.} Since both $\widetilde U_\varepsilon$ and $D_1$ are star-shaped with respect to
the origin, the area of their symmetric difference may be computed in polar coordinates:
\[
|\widetilde U_\varepsilon \triangle D_1|
=
\frac12\int_{\mathbb S^1}
\left|(1+\widetilde\varphi_\varepsilon(y))^2-1\right|
\,\mathrm{d}\mathcal H^1(y).
\]
Since \(\widetilde\varphi_\varepsilon=\varepsilon\dot\varphi_0+O(\varepsilon^2)\)
in \(C^1(\mathbb S^1)\), we have in particular
\(\|\widetilde\varphi_\varepsilon\|_{L^\infty(\mathbb S^1)}\lesssim |\varepsilon|\),
whence \(\widetilde\varphi_\varepsilon^2=O(\varepsilon^2)\) in
\(L^\infty(\mathbb S^1)\). Hence
\[
(1+\widetilde\varphi_\varepsilon)^2-1
=
2\widetilde\varphi_\varepsilon+\widetilde\varphi_\varepsilon^2
=
2\varepsilon \dot\varphi_0 + O(\varepsilon^2)
\quad \text{in } L^\infty(\mathbb S^1),
\]
where the implied constant depends only on the \(C^1\)-norm bound on
\(\widetilde\varphi_\varepsilon\). Therefore
\[
|\widetilde U_\varepsilon \triangle D_1|
=
\frac12\int_{\mathbb S^1}
\left|2\varepsilon \dot\varphi_0 + O(\varepsilon^2)\right|
\,\mathrm{d}\mathcal H^1.
\]
Set \(\alpha:=\int_{\mathbb S^1}|\dot\varphi_0|\,\mathrm{d}\mathcal H^1\). By
\eqref{eq:l1-nontrivial-new}, \(\alpha>0\). The triangle inequality in
\(L^1\) and the \(O(\varepsilon^2)\) remainder give
\[
\left|\int_{\mathbb S^1}\bigl|2\varepsilon\dot\varphi_0+O(\varepsilon^2)\bigr|\,\mathrm{d}\mathcal H^1 - 2|\varepsilon|\alpha \right|
=O(\varepsilon^2).
\]
Consequently there exist constants $c,C>0$ such that
\begin{equation}\label{eq:distance-centered-disk-new}
c|\varepsilon|
\leq
|\widetilde U_\varepsilon \triangle D_1|
\leq
C|\varepsilon|
\end{equation}
for all sufficiently small $\varepsilon$.

\medskip
\noindent
{\bf Step 5: comparison with arbitrary disks.} Let $D$ be any disk of area $\pi$. Then $D=D_1+z_0$ for some $z_0\in \C$. We first consider disks whose centers are not too close to the origin.
Since
\[
|D_1\triangle (D_1+z_0)|
\gtrsim \min\{|z_0|,1\},
\]
the triangle inequality and \eqref{eq:distance-centered-disk-new} imply that
there exists $K\gg 1$ such that, if $|z_0|\geq K|\varepsilon|$, then
\begin{equation}\label{eq:far-disks-new}
|\widetilde U_\varepsilon \triangle D|
\geq
|D_1\triangle D|-|\widetilde U_\varepsilon\triangle D_1|
\gtrsim
\min\{|z_0|,1\}-C|\varepsilon|
\gtrsim
|\varepsilon|,
\end{equation}
for sufficiently small $|\varepsilon|$.

It remains to consider disks whose centers satisfy $|z_0|\leq K|\varepsilon|$.
Write
\[
z_0=\eta \varepsilon e^{i\theta},
\qquad
|\eta|\leq K,
\quad
\theta\in [0,2\pi).
\]
We claim that, for $\varepsilon$ sufficiently small, the translated domain
\[
\widetilde U_\varepsilon-z_0
\]
is still a normal graph over $\mathbb S^1$:
\[
\partial(\widetilde U_\varepsilon-z_0)
=
\{(1+\psi_\varepsilon^{\theta,\eta}(y))y:\ y\in \mathbb S^1\},
\]
with
\[
\psi_\varepsilon^{\theta,\eta}\in C^1(\mathbb S^1),
\qquad
\|\psi_\varepsilon^{\theta,\eta}\|_{C^1}\lesssim |\varepsilon|,
\]
uniformly in $(\theta,\eta)$ with $|\eta|\leq K$. This follows because
$\widetilde U_\varepsilon$ is already an $O(\varepsilon)$-$C^1$ perturbation of
$D_1$, and the translation vector $z_0$ is also of order $\varepsilon$. We now compute the first-order term of $\psi_\varepsilon^{\theta,\eta}$.
A translation by $z_0$ contributes, to first order in $\varepsilon$, the normal
displacement
\[
-\eta \Re(e^{-i\theta}y).
\]
Indeed, for \(y\in \mathbb S^1\),
\[
|y-z_0|
=
\bigl(1-2\Re(\overline y z_0)+|z_0|^2\bigr)^{1/2}
=
1-\Re(\overline y z_0)+O(|z_0|^2),
\]
uniformly in \(y\), and here \(z_0=\eta\varepsilon e^{i\theta}\).
Therefore
\begin{equation}\label{eq:psi-expansion-new}
\psi_\varepsilon^{\theta,\eta}(y)
=
\varepsilon\Big(\dot\varphi_0(y)-\eta \Re(e^{-i\theta}y)\Big)
+
O(\varepsilon^2)
\quad \text{in } C^1(\mathbb S^1),
\end{equation}
uniformly for $|\eta|\leq K$. By Lemma \ref{lem:symmetric-difference} and the remark right after it, we conclude
that
\[
\inf_{\substack{D \text{ disk}\\ |D|=\pi}}
|\widetilde U_\varepsilon \triangle D|
\gtrsim |\varepsilon|.
\]
Since \(r_\varepsilon=1+O(\varepsilon^2)\), the dilation from \(U_\varepsilon\)
to \(\widetilde U_\varepsilon\) changes the set only by
\[
|\widetilde U_\varepsilon\triangle U_\varepsilon|=O(\varepsilon^2),
\]
and therefore does not affect the first-order distance estimate. Thus the
area-normalized family \(\widetilde U_\varepsilon\) remains at
symmetric-difference distance \(\gtrsim |\varepsilon|\) from the class of
disks.

The eigenfunction part of the estimate from \cite{Gomez-Guerra-Ramos-Tilli} has the same scale. Normalize
\[
\widehat P_\varepsilon:=
\frac{1+\varepsilon z^2}{\|1+\varepsilon z^2\|_{\mathcal F^2}},
\qquad
\|1+\varepsilon z^2\|_{\mathcal F^2}^2
=1+\frac{2}{\pi^2}|\varepsilon|^2.
\]
The normalized kernels are
\[
k_a(z):=e^{\pi\overline a z-\frac{\pi}{2}|a|^2},
\qquad a\in\C.
\]
Their Bargmann inverses are precisely the \(L^2\)-normalized, phase-space
shifted Gaussians. The reproducing identity gives
\[
|\langle \widehat P_\varepsilon,k_a\rangle_{\mathcal F^2}|
=
\frac{|1+\varepsilon a^2|\,e^{-\pi|a|^2/2}}
{\bigl(1+\frac{2}{\pi^2}|\varepsilon|^2\bigr)^{1/2}}.
\]
Set \(x:=|\varepsilon|\) and
\(r:=|a|^2\).  If \(0<x<\pi/2\), then
\[
|1+\varepsilon a^2|e^{-\pi|a|^2/2}
\le (1+xr)e^{-\pi r/2}\le 1,
\]
as \(\log(1+xr)\le xr\le \pi r/2\).  Hence
\[
\sup_{a\in\C}|\langle \widehat P_\varepsilon,k_a\rangle_{\mathcal F^2}|
\le
\left(1+\frac{2}{\pi^2}|\varepsilon|^2\right)^{-1/2}
\le 1-\frac{|\varepsilon|^2}{2\pi^2}
\]
after decreasing \(\varepsilon\) if necessary.  Since for unit
vectors \(v,w\) one has
\(\inf_{|c|=1}\|v-cw\|^2=2(1-|\langle v,w\rangle|)\), it follows that
\[
\inf_{\substack{a\in\C\\ |c|=1}}
\|\widehat P_\varepsilon-c\,k_a\|_{\mathcal F^2}
\ge \frac{|\varepsilon|}{\pi}.
\]
By unitarity of the Bargmann transform, the same lower bound holds for the
normalized first eigenfunction in \(L^2(\R)\) against the orbit of Gaussian
optimizers.

\medskip
\noindent
{\bf Step 6: square-root sharpness of the deficit estimate.} It remains to translate the distance estimate into the corresponding deficit
estimate. Recall that, by construction via Theorem~\ref{thm:perturb}, the
identity \eqref{eq:functiona-P-eps-new} states that \(P_\varepsilon\) is an
eigenfunction with eigenvalue \(\lambda_0\) for the localization operator on
\(U_\varepsilon\). By Lemma~\ref{lem:spectral-continuity}, the eigenvalue
\(\lambda_0\) is in fact the \emph{principal} (largest) eigenvalue of
\(\mathcal L_{U_\varepsilon}\) for all sufficiently small \(\varepsilon\),
and \(P_\varepsilon\) is a corresponding first eigenfunction; the gap
\eqref{eq:gap-preservation} ensures that \(\lambda_0\) cannot be confused
with any lower-order eigenvalue. Hence
\begin{equation}\label{eq:exact-eigenvalue}
\lambda_1(U_\varepsilon)=\lambda_0
\qquad\text{exactly, for }|\varepsilon|<\varepsilon_*.
\end{equation}
By \eqref{eq:measure-close-u-eps-new},
\(|U_\varepsilon|=\pi+O(\varepsilon^2)\), and hence
\[
0\le
1-\frac{\lambda_1(U_\varepsilon)}{1-e^{-|U_\varepsilon|}}
=O(\varepsilon^2).
\]
Thus the coherent-state estimate above concerns the first eigenfunction of
$U_\varepsilon$ and has a deficit of order at most $\varepsilon^2$.
The area-normalizing dilation factor
\(r_\varepsilon=\sqrt{\pi/|U_\varepsilon|}=1+O(\varepsilon^2)\). By
Lemma~\ref{lem:dilation-eigenvalue}, the dilation changes
the principal eigenvalue by
\[
\bigl|\lambda_1(\widetilde U_\varepsilon)-\lambda_1(U_\varepsilon)\bigr|
=
\bigl|\lambda_1(r_\varepsilon U_\varepsilon)-\lambda_0\bigr|
=
O(\varepsilon^2).
\]
Combining that with \eqref{eq:exact-eigenvalue}, the eigenvalue deficit of
\(\widetilde U_\varepsilon\) with respect to the disk satisfies
\[
0\le\lambda_0-\lambda_1(\widetilde U_\varepsilon)=O(\varepsilon^2),
\]
where the inequality on the left-hand side follows from the Faber--Krahn inequality for the Gabor transform \cite{Nicola-Tilli}. On the other hand, by Step~5,
the symmetric-difference asymmetry satisfies
\[
\inf_{\substack{D\text{ disk}\\ |D|=\pi}}|\widetilde U_\varepsilon\triangle D|
\gtrsim |\varepsilon|.
\]
The already-known upper stability estimate in
\eqref{eq:GGRT-bounds-intro}, applied at the fixed area \(\pi\), supplies the
reverse lower bound on the deficit. Thus
\[
\lambda_0-\lambda_1(\widetilde U_\varepsilon)\asymp\varepsilon^2,
\qquad
\inf_{\substack{D\text{ disk}\\ |D|=\pi}}|\widetilde U_\varepsilon\triangle D|
\asymp|\varepsilon|.
\]
The coherent-state lower bound for $U_\varepsilon$ proves sharpness of the
first inequality in \eqref{eq:GGRT-bounds-intro}. The asymmetry bound for
$\widetilde U_\varepsilon$ proves sharpness of the second. Hence no exponent
\(\beta>1/2\) can hold uniformly in either inequality, which proves the
theorem.
\end{proof}


\section{Local maximizers for the concentration problem for the Gabor transform}\label{sec:local-max}

We work throughout this section in the Bargmann--Fock representation, using
the complex STFT notation fixed after \eqref{eq:STFT-B}. Thus, if
\(F\in \mathcal F^2(\C)\) is the Bargmann transform of \(f\in L^2(\R)\), then
\[
|V_\varphi f(z)|^2 = |F(z)|^2 e^{-\pi |z|^2}=u_F(z),
\]
with $u_F$ as in \eqref{eq:uF}.

For a measurable set \(\Omega\subset \C\), let
\begin{equation}\label{eq:lambda1-def-sec6}
\lambda_{1,\varphi}(\Omega)
:=
\sup_{\|f\|_2=1}\int_\Omega |V_\varphi f(z)|^2\,\mathrm{d}A(z).
\end{equation}

\subsection{The weighted symmetric-difference topology}

The locality of a maximizer must be defined in a topology compatible with the
Lebesgue-measure constraint and with the Gaussian phase-space scale. We use the
weighted symmetric-difference metric obtained by integrating against the Gaussian weight
\(e^{-\pi|z|^2}\,\mathrm{d}A(z)\), which is the same density that appears both in the
Fock-space norm and in the eigenfunction \(u_F\).

\begin{definition}\label{def:weighted-sym-diff}
For measurable sets \(\Omega_1,\Omega_2\subset \C\) of finite Lebesgue measure,
set
\begin{equation}\label{eq:weighted-sym-diff}
d_w(\Omega_1,\Omega_2)
:=
\int_{\Omega_1\triangle\Omega_2} w(z)\,\mathrm{d}A(z),
\qquad
w(z):=e^{-\pi |z|^2}.
\end{equation}
We call \(d_w\) the \emph{weighted symmetric-difference distance}, and we denote
by \(\mathcal M_a\) the Borel-measurable subsets of \(\C\) of Lebesgue measure
\(a>0\), modulo Lebesgue-null modifications. The map
\(d_w\colon \mathcal M_a\times \mathcal M_a\to[0,\infty)\) is a genuine metric
on \(\mathcal M_a\): symmetry, the triangle inequality, and vanishing only on
null modifications all follow from \(w(z)=e^{-\pi|z|^2}>0\) and standard
properties of the symmetric difference.
\end{definition}

\begin{remark}\label{rem:topology-natural}
Three observations clarify why \(d_w\) is the natural topology for our problem.

\begin{enumerate}
\item[\textnormal{(i)}] (\emph{Compatibility with the measure constraint.})
Since \(w\) is bounded above by \(1\) and bounded below by \(e^{-\pi R^2}>0\)
on every disk \(\{|z|\le R\}\), the metric \(d_w\) is, on every fixed bounded
region of \(\C\), equivalent to the unweighted symmetric-difference distance
\(d_0(\Omega_1,\Omega_2):=|\Omega_1\triangle\Omega_2|\). Moreover, every
\(d_w\)-perturbation \(E\) of \(\Omega\) used in this section is constructed
to satisfy the equimeasurability constraint \(|E|=|\Omega|\); hence the
\(d_w\)-balls intersected with \(\mathcal M_a\) form a basis of neighborhoods
on the measure-constraint manifold.

\item[\textnormal{(ii)}] (\emph{Covariance of the spectral problem.})
The localization problem itself is invariant under Weyl translations and
rotations. Under the Weyl translations
\(T_a\colon F\mapsto e^{\pi\overline a z-\frac{\pi}{2}|a|^2}F(z-a)\) (cf.\
\eqref{eq:Ta}) and rotations \(F(z)\mapsto F(e^{-i\theta}z)\), one has the
covariance
\[
u_{T_aF}(z)=u_F(z-a),\qquad u_{F(\,e^{-i\theta}\cdot\,)}(z)=u_F(e^{-i\theta}z),
\]
so that the spectral data \((\lambda_{1,\varphi},u_F)\) of the localization
operator transform unitarily under these symmetries. The metric \(d_w\) is
rotation-invariant but, because its weight is pinned to the origin, it is not
translation-invariant. This is a local gauge rather than an additional
symmetry; translation covariance is restored when the set and eigenfunction
are translated simultaneously.

\item[\textnormal{(iii)}] (\emph{Control of Fock-space tails.}) For any
\(F\in \mathcal F^2(\C)\) one has \(u_F(z)\le \|F\|_{\mathcal F^2}^2\) for
every \(z\in\C\), as a consequence of the reproducing estimate
\eqref{eq:pointwise} and Cauchy--Schwarz. Hence, for normalized
first eigenfunctions \(F\),
\[
\Bigl|\int_{\Omega_1}u_F\,\mathrm{d}A-\int_{\Omega_2}u_F\,\mathrm{d}A\Bigr|
\le
\int_{\Omega_1\triangle\Omega_2}u_F\,\mathrm{d}A
\le
\|F\|_{\mathcal F^2}^2\,e^{\pi R^2}\, d_w(\Omega_1,\Omega_2)
\]
on any bounded region \(\{|z|\le R\}\), so that the localization functional
\(\Omega\mapsto \int_\Omega u_F\,\mathrm{d}A\) is Lipschitz in \(d_w\) on bounded
sets. Changes far away have little weight according to \(d_w\); all perturbations used below
are compactly supported, and no compactness-at-infinity conclusion is drawn
from this metric alone.
\end{enumerate}

In summary, \(d_w\) is a useful Gaussian-weighted local \(L^1\)-metric on
indicator functions. On bounded regions it is equivalent to ordinary symmetric
difference, which is the only comparison needed below.
\end{remark}

\begin{definition}\label{def:local-maximizer}
A set \(\Omega\in\mathcal M_a\) is a \emph{local maximizer} for
\(\lambda_{1,\varphi}\) under the Lebesgue-measure constraint
\(|\,\cdot\,|=a\) if there exists \(\delta_0>0\) such that, whenever
\(E\in \mathcal M_a\) satisfies \(d_w(\Omega,E)<\delta_0\), one has
\begin{equation}\label{eq:local-max-ineq}
\lambda_{1,\varphi}(\Omega)\ge \lambda_{1,\varphi}(E).
\end{equation}
\end{definition}

We now prove Theorem~\ref{thm:local-maximizers}. The proof proceeds in six steps. Steps 1--3 are essentially the analogues, in our
setting, of standard arguments in the Faber--Krahn theory: \(\Sigma\) is forced
to be a level set of \(u_F\), the corresponding eigenvalue is simple, and a
first-variation identity ties \(F\) to a divergence-free vector field built from
holomorphic data. The crucial new ingredient lies in Step 4, where we show that
the \emph{derivative} \(F'\) is again a first eigenfunction. Combined with the
simplicity from Step 2, this forces \(F\) to satisfy a first-order linear ODE,
and hence to be a Gaussian. Once this is in place, Steps 5--6 use the Weyl
invariance of the Bargmann--Fock setting and the superlevel-set
characterization from Step 1 to identify \(\Sigma\) with a centered disk.

\begin{proof}[Proof of Theorem~\ref{thm:local-maximizers}]
Let \(\Sigma\subset \C\) be a local maximizer for \(\lambda_{1,\varphi}\) under
the measure constraint, and let \(F\in \mathcal F^2(\C)\) be a normalized first
eigenfunction for the localization operator associated with \(\Sigma\).
Equivalently,
\begin{equation}\label{eq:eigenvalue-equation}
\int_\Sigma H(z)\overline{F(z)}\,\mathrm{d}\mu(z)
=
\lambda_{1,\varphi}(\Sigma)\int_\C H(z)\overline{F(z)}\,\mathrm{d}\mu(z)
\end{equation}
for every \(H\in \mathcal F^2(\C)\), where
\[
\mathrm{d}\mu(z):=e^{-\pi |z|^2}\,\mathrm{d}A(z).
\]

The proof of the superlevel-set characterization in Step~1 below uses the
following classical rearrangement inequality, in the formulation of
Lieb--Loss
\cite[Theorem~1.14]{Lieb-Loss} (the so-called
\emph{bathtub principle}). We record it here for the reader's convenience. 

\begin{lemma}[Bathtub principle]\label{lem:bathtub}
Let \((X,\Sigma_X,\mu_X)\) be a measure space and let
\(g\colon X\to[0,\infty)\) be a measurable function such that
\(\mu_X(\{g>t\})<\infty\) for every \(t>0\). Fix \(0<a<\mu_X(X)\). Then
\begin{equation}\label{eq:bathtub-equation}
\sup\Bigl\{\int_X g\,\mathbf 1_E\,\mathrm{d}\mu_X\ \Big|\ E\in\Sigma_X,\ \mu_X(E)=a\Bigr\}
\end{equation}
is attained, and any maximizer \(E^\ast\) coincides, up to a \(\mu_X\)-null
set, with a superlevel set \(\{g>t_\ast\}\) of \(g\) (possibly with a
measure-zero correction at the level set \(\{g=t_\ast\}\) needed to fit the
constraint \(\mu_X(E^\ast)=a\)). The threshold \(t_\ast\ge 0\) is uniquely
determined by the condition
\[
\mu_X(\{g>t_\ast\})\le a\le \mu_X(\{g\ge t_\ast\}),
\]
and the maximum is strict in the sense that, if \(g\) has no plateaus of
positive measure on the level \(\{g=t_\ast\}\), then \(E^\ast=\{g>t_\ast\}\)
\(\mu_X\)-a.e.
\end{lemma}

In our application, \((X,\mu_X)=(\C,\mathrm{Leb}_\C)\) and \(g=u_F\). Since
\(u_F\) is real-analytic on the (open) complement of the zero set of \(F\),
its level sets \(\{u_F=t\}\) are real-analytic varieties of strictly lower
dimension and hence Lebesgue-null for every \(t>0\); the no-plateau
hypothesis of Lemma~\ref{lem:bathtub} is therefore automatic, and the
maximizer is unique up to null sets.

\medskip
\noindent
{\bf Step 1: \(\Sigma\) is a superlevel set of \(u_F\).}

We claim that there exists \(t_0\ge 0\) such that, up to sets of measure zero,
\begin{equation}\label{eq:superlevel-claim}
\Sigma=\{u_F>t_0\}.
\end{equation}

The strategy is a bathtub-style perturbation argument: if \(\Sigma\) were not a
superlevel set of \(u_F\), one could move a small piece of \(\Sigma\) where
\(u_F\) is small to a region where \(u_F\) is larger, increasing the
localization integral while preserving Lebesgue measure and remaining within
the local-maximizer neighborhood of \(\Sigma\) (cf.\ Lemma~\ref{lem:bathtub}).
The contradiction we extract is local, but as we shall see in
Lemma~\ref{lem:bathtub-superlevel-from-localmax}, the same argument actually
proves the slightly stronger statement that \(\Sigma\) is the unique
\(\mu\)-essential maximizer of the functional \(E\mapsto \int_E u_F\,\mathrm{d}A\)
under the equimeasurable constraint, in agreement with the bathtub principle.

Recall that \(F\) is admissible in the variational problem defining
\(\lambda_{1,\varphi}(\Sigma')\) for any measurable set \(\Sigma'\) of finite
measure. In particular,
\[
\lambda_{1,\varphi}(\Sigma')\ge \int_{\Sigma'} u_F\,\mathrm{d}A,
\qquad
\lambda_{1,\varphi}(\Sigma)= \int_\Sigma u_F\,\mathrm{d}A,
\]
the latter equality holding because \(F\) is the actual eigenfunction for
\(\Sigma\). It therefore suffices to find a competitor \(\Sigma'\) with
\(|\Sigma'|=|\Sigma|\), close to \(\Sigma\) in the sense of weighted symmetric
difference, and with \(\int_{\Sigma'} u_F\,\mathrm{d}A>\int_\Sigma u_F\,\mathrm{d}A\); that would
violate local maximality.

Suppose, for the sake of contradiction, that \eqref{eq:superlevel-claim} fails.
The function \(u_F\) is continuous, belongs to \(C_0(\C)\cap L^1(\C)\), and
has no positive-measure positive level sets because it is real-analytic on
\(\C\setminus\{F=0\}\). The bathtub principle therefore provides a threshold
\(t_0\ge0\) such that
\[
A:=\{u_F>t_0\}\quad\text{satisfies}\quad |A|=|\Sigma|;
\]
this threshold is necessarily positive, since \(\{u_F>0\}=\C\setminus\{F=0\}\)
has infinite Lebesgue measure whereas \(|\Sigma|<\infty\).
Since \eqref{eq:superlevel-claim} fails by assumption, \(A\) and \(\Sigma\)
do not coincide up to a null set, so both \(A\setminus\Sigma\) and
\(\Sigma\setminus A\) have positive Lebesgue measure (and in fact positive
weighted measure, since \(e^{-\pi|z|^2}>0\) everywhere).

We now perform the bathtub construction explicitly. The set
\(A=\{u_F>t_0\}\) is open since \(u_F\) is continuous, while
\(\Sigma\setminus A\) is only known to be measurable. By inner regularity of
the Lebesgue measure, choose a compact set \(L\subset A\setminus\Sigma\) with
\(|L|>0\) and a compact set \(K\subset \Sigma\setminus A\) with \(|K|>0\)
(both nonempty by the previous paragraph). Note that \(L\subset A\) is then
contained in the open set \(A\), and we can shrink \(L\) further if needed so
that \(\int_L \mathrm{d}\mu, \int_K \mathrm{d}\mu<\delta_0/4\). Since the Lebesgue measure
restricted to a compact set has the intermediate-value property, we may
choose subsets \(S\subset L\) and \(T\subset K\) with the common Lebesgue
measure
\[
|S|=|T|=\tfrac12\min(|L|,|K|)>0.
\]
Then \(S\subset A\setminus\Sigma\), \(T\subset \Sigma\setminus A\), and
\[
\int_\C (\mathbf 1_S+\mathbf 1_T)\,e^{-\pi |z|^2}\,\mathrm{d}A(z)<\tfrac{\delta_0}{2}<\delta_0.
\]
Define
\[
\Sigma':=(\Sigma\setminus T)\cup S.
\]
Then \(|\Sigma'|=|\Sigma|\) and
\[
\int_\C |\mathbf 1_\Sigma-\mathbf 1_{\Sigma'}|e^{-\pi |z|^2}\,\mathrm{d}A(z)
=\int_\C(\mathbf 1_S+\mathbf 1_T)e^{-\pi|z|^2}\,\mathrm{d}A(z)<\delta_0.
\]
On the other hand, by construction,
\[
u_F(z)>t_0 \quad \text{on } S\subset A,
\qquad
u_F(z)\le t_0 \quad \text{on } T\subset \Sigma\setminus A.
\]
Moreover, since \(u_F\) is continuous and \(L\subset A=\{u_F>t_0\}\) is
compact, the continuous function \(u_F\) attains its minimum on \(L\); hence
there exists \(\eta>0\) with \(u_F\ge t_0+\eta\) on \(L\), and a fortiori on
\(S\). It then follows that
\[
\int_S u_F\,\mathrm{d}A \ge (t_0+\eta)\,|S|,
\qquad
\int_T u_F\,\mathrm{d}A \le t_0\,|T|=t_0\,|S|,
\]
and hence
\[
\int_{\Sigma'} u_F\,\mathrm{d}A
=
\int_\Sigma u_F\,\mathrm{d}A + \int_S u_F\,\mathrm{d}A-\int_T u_F\,\mathrm{d}A
\ge \int_\Sigma u_F\,\mathrm{d}A + \eta\,|S|
>
\int_\Sigma u_F\,\mathrm{d}A.
\]
Since \(F\) is admissible for the variational problem defining
\(\lambda_{1,\varphi}(\Sigma')\),
\[
\lambda_{1,\varphi}(\Sigma')
\ge \int_{\Sigma'} u_F\,\mathrm{d}A
>
\int_\Sigma u_F\,\mathrm{d}A
=
\lambda_{1,\varphi}(\Sigma),
\]
contradicting the local maximality of \(\Sigma\). This proves
\eqref{eq:superlevel-claim}.

From this point onward, we replace \(\Sigma\) by the open representative
\(\{u_F>t_0\}\), which is allowed because the variational problem only
depends on \(\Sigma\) up to null sets. By continuity of \(u_F\), this
representative has the exact boundary-level relation
\begin{equation}\label{eq:boundary-level-set-new}
u_F(z)=t_0=:c_\Sigma
\qquad \text{for } z\in \partial \Sigma.
\end{equation}
We will see in Step 2 that \(F\) does not vanish on \(\overline\Sigma\), and in
particular \(c_\Sigma>0\).

We collect the conclusions of this step in a self-contained statement that we
will reuse below.

\begin{lemma}\label{lem:bathtub-superlevel-from-localmax}
Let \(\Sigma\in\mathcal M_a\) be a local maximizer for \(\lambda_{1,\varphi}\)
in the sense of Definition~\ref{def:local-maximizer}, and let \(F\in
\mathcal F^2(\C)\) be a normalized first eigenfunction associated with
\(\Sigma\). Then there exists \(t_0>0\) such that
\begin{equation}\label{eq:superlevel-strong}
\Sigma=\{z\in\C\,:\, u_F(z)>t_0\}\quad
\text{up to a Lebesgue-null set,}
\end{equation}
and \(\Sigma\) coincides up to a null set with the unique (up to null sets)
maximizer of the bathtub problem
\(\sup\{\int_E u_F\,\mathrm{d}A\,:\,|E|=a\}\) provided by Lemma~\ref{lem:bathtub}
applied to \(g=u_F\).
\end{lemma}

\begin{proof}
The first part is exactly \eqref{eq:superlevel-claim}; the strict positivity
\(t_0>0\) follows because \(\{u_F>0\}=\C\setminus\{F=0\}\) has infinite
Lebesgue measure while \(|\Sigma|<\infty\). Uniqueness within
bathtub maximizers follows from the no-plateau remark below
Lemma~\ref{lem:bathtub}, since \(u_F\) is real-analytic on the open set
\(\{F\ne 0\}\supset\overline\Sigma\) and hence has no positive-measure level
sets there.
\end{proof}

\begin{lemma}\label{lem:boundary-reg}
Under the hypotheses of Lemma~\ref{lem:bathtub-superlevel-from-localmax},
choosing the open superlevel-set representative
\(\Sigma=\{u_F>t_0\}\), we have:
\begin{enumerate}
\item[\textnormal{(i)}] \(\Sigma\) is open and bounded (because
\(u_F(z)\to 0\) as \(|z|\to\infty\) and \(t_0>0\));
\item[\textnormal{(ii)}] given
that \(F\) does not vanish on \(\overline\Sigma\), \(u_F\) is real-analytic
on a neighborhood of \(\overline\Sigma\);
\item[\textnormal{(iii)}] under the same nonvanishing hypothesis,
\(\Sigma\) has finite perimeter. Its regular
level-set locus is a real-analytic \(1\)-manifold, and its topological
boundary is a compact semianalytic curve; in particular,
\(\mathcal H^1(\partial\Sigma)<\infty\);
\item[\textnormal{(iv)}] under the same nonvanishing hypothesis, for every
\(C^1\)-vector field \(X\) defined on a
neighborhood of \(\overline\Sigma\), the Gauss--Green identity
\[
\int_\Sigma \mathrm{div}\,X\,\mathrm{d}A
=\int_{\partial^*\Sigma}X\cdot\nu_\Sigma\,\mathrm{d}\mathcal H^1
\]
holds, where \(\partial^*\Sigma\) is the reduced boundary and
\(\nu_\Sigma\) its measure-theoretic outward normal.
\end{enumerate}
\end{lemma}

\begin{proof}
Statement (i) follows directly from Lemma~\ref{lem:Fock-basic}(iii) and
\(t_0>0\). Once \(F\) is
nonvanishing on \(\overline\Sigma\), \(u_F\) is real-analytic on a
neighborhood of \(\overline\Sigma\). The standard local structure theorem
for planar real-analytic sets (see, e.g.,
\cite{Krantz-Parks-Primer,Lojasiewicz}) shows that the compact level set
\(\{u_F=t_0\}\cap\overline\Sigma\) is a finite semianalytic curve. This
implies the finite-perimeter assertion in (iii). Statement (iv) is then the
Gauss--Green theorem for sets of finite perimeter
(cf.\ \cite{Evans-Gariepy}). This formulation deliberately uses the reduced
boundary: a critical level may contain geometrically degenerate arcs, which
need not contribute to the measure-theoretic boundary.
\end{proof}

\begin{remark}\label{rem:reg-vs-hypothesis}
Note that \emph{no} boundary regularity is assumed on the local maximizer: the only input
is the measurability of \(\Sigma\). The superlevel-set characterization,
the entirety of \(F\), and the Gaussian weight yield precisely the
finite-perimeter structure needed in Step~3. Thus the Gauss--Green identity
used there is fully justified by Lemma~\ref{lem:boundary-reg}(iv), even
though \(\Sigma\) is initially only a measurable set. 
\end{remark}

\medskip
\noindent
{\bf Step 2: simplicity of the first eigenvalue.}

In this step we prove the following lemma, which lies at the heart of the
remainder of the argument. It is what will allows us, in Step~4, to conclude from
the fact that \(F'\) is a first eigenfunction that \(F'\) is a scalar multiple
of \(F\) and hence that \(F\) satisfies a linear ODE.

\begin{lemma}\label{lem:simplicity-local-max}
Let \(\Sigma\) be a local maximizer for \(\lambda_{1,\varphi}\) in the sense
of Definition~\ref{def:local-maximizer} and let \(F\in\mathcal F^2(\C)\) be a
normalized first eigenfunction associated with \(\Sigma\). Then the
eigenspace \(\ker(\mathcal L_\Sigma-\lambda_{1,\varphi}(\Sigma)\,\mathrm I)\)
is one-dimensional; that is, \(\lambda_{1,\varphi}(\Sigma)\) is simple, and
\(F\) is uniquely determined up to a unimodular constant. In particular,
\(F\) does not vanish on \(\overline\Sigma\).
\end{lemma}

The proof uses a modulus-ratio argument: any two first eigenfunctions
have the same modulus on \(\partial\Sigma\) up to a constant. The maximum
principle makes their quotient have constant modulus on each component of
\(\Sigma\), and holomorphicity then makes that quotient constant.
This argument plays a role also in the stability analysis of
\cite{Gomez-Guerra-Ramos-Tilli}.

\begin{proof}[Proof of Lemma~\ref{lem:simplicity-local-max}]
Suppose, for the sake of a contradiction, that \(G\in \mathcal F^2(\C)\) is another first
eigenfunction, linearly independent of \(F\). Step 1 applied to \(G\) shows
that \(\Sigma\) agrees up to a null set with \(\{u_G>c'_\Sigma\}\) for some
\(c'_\Sigma>0\). Indeed, the threshold cannot be zero, because
\(\{u_G>0\}=\C\setminus\{G=0\}\) has infinite measure. Both
\(\Sigma=\{u_F>c_\Sigma\}\) and \(\{u_G>c'_\Sigma\}\) are open, so their
null symmetric difference forces exact equality. Thus
\[
u_G(z)=c'_\Sigma
\qquad\text{for } z\in \partial \Sigma.
\]
Equivalently,
\[
|F(z)|^2 e^{-\pi |z|^2}=c_\Sigma,
\qquad
|G(z)|^2 e^{-\pi |z|^2}=c'_\Sigma,
\qquad z\in \partial \Sigma.
\]

We next observe that first eigenfunctions are zero-free in the interior of
\(\Sigma\). Indeed, by Step 1 we have \(u_F>c_\Sigma\) on the (open) interior of
\(\Sigma\). Since
\[
u_F=|F|^2 e^{-\pi |z|^2}
\]
and the Gaussian factor \(e^{-\pi|z|^2}\) is strictly positive everywhere, this
yields \(|F|>0\) throughout the interior of \(\Sigma\); the same applies to
\(|G|\). In particular, \(c_\Sigma>0\) and \(c'_\Sigma>0\). Thus \(F/G\) is
holomorphic and nowhere zero on each connected component of the open set
\(\Sigma\); no enlargement of \(\Sigma\) is needed.

The function
\[
\log \left|\frac{F}{G}\right|
=
\frac12 \log \frac{|F|^2}{|G|^2}
=
\frac12 \log \frac{u_F}{u_G}
\]
is then \emph{harmonic} on each component of \(\Sigma\), as it is the real part of any local branch of
\(\log(F/G)\), which is well defined on simply connected subsets; indeed, another way to see that this function is in fact harmonic
throughout is to notice that it is the logarithm of the modulus of a zero-free holomorphic
function. Note that the Gaussian factors cancel inside the logarithm, so on
\(\partial \Sigma\) by \eqref{eq:boundary-level-set-new} this function equals
the constant \(\frac12\log(c_\Sigma/c'_\Sigma)\). By the maximum principle
applied separately on each connected component of \(\Sigma\), the function
\(\log|F/G|\) is constant on each component. Hence
\[
\left|\frac{F}{G}\right| \equiv \text{constant}
\qquad \text{on each connected component of } \Sigma.
\]
This immediately implies that
\[
F=\gamma_j G
\qquad \text{on each connected component } \Sigma_j\text{ of } \Sigma.
\]
The relation \(F=\gamma_j G\) is an identity between two entire functions on a
set with an accumulation point; the analytic continuation principle gives
\(F=\gamma_j G\) on all of \(\C\) for any single \(j\). In particular all
\(\gamma_j\) coincide, and \(F\) and \(G\) are linearly dependent, contradicting
our assumption. This proves Lemma~\ref{lem:simplicity-local-max}.
\end{proof}

\begin{remark}
As a by-product of the proof above, every first eigenfunction associated with
$\mathcal L_\Sigma$ is nonvanishing inside $\Sigma$.
\end{remark}

\medskip
\noindent
{\bf Step 3: a first-variation identity.}

By Lemma~\ref{lem:simplicity-local-max}, the eigenfunction \(F\) is determined
up to a unimodular constant. We now derive a first-variation identity
satisfied by \(F\), based on the fact that \(\Sigma\) is a level set of
\(u_F\) and that holomorphic vector fields are divergence-free.

We work throughout this step with the open superlevel-set representative
\(\Sigma=\{u_F>t_0\}\) from
Lemma~\ref{lem:bathtub-superlevel-from-localmax}. By Lemma~\ref{lem:simplicity-local-max}, \(F\) does not vanish on
\(\overline\Sigma\); applying Lemma~\ref{lem:boundary-reg} with this
nonvanishing input, \(u_F\) is real-analytic on a neighborhood of
\(\overline\Sigma\), and \(\Sigma\) has finite perimeter. The Gauss--Green
formula therefore applies on \(\Sigma\) to every \(C^1\)-vector field defined
in a neighborhood of \(\overline\Sigma\), with integration over the reduced
boundary. All boundary integrations below are therefore justified without
any a priori smoothness assumption on \(\Sigma\); cf.\
Remark~\ref{rem:reg-vs-hypothesis}.

Let \(\psi\) be holomorphic on an open neighborhood \(V\) of
\(\overline \Sigma\), and write
\[
\psi=u+iv
\]
with \(u,v\) real-valued. Consider the real vector field
\[
X_\psi:=(u,-v).
\]
Since \(\psi\) is holomorphic, the Cauchy--Riemann equations
\[
u_x=v_y,
\qquad
u_y=-v_x,
\]
imply
\[
\operatorname{div}X_\psi=u_x+(-v)_y=u_x-v_y=0
\qquad \text{in }V.
\]
Because \(u_F\equiv c_\Sigma\) on \(\partial\Sigma\), and hence on
\(\partial^*\Sigma\), the Gauss--Green formula on
\(\Sigma\) gives
\[
\int_\Sigma \langle \nabla u_F, X_\psi\rangle\,\mathrm{d}A
=
\int_\Sigma \operatorname{div}(u_F X_\psi)\,\mathrm{d}A
=
\int_{\partial^*\Sigma} u_F\, X_\psi\cdot \nu_\Sigma\, \mathrm{d}\mathcal H^1
=
c_\Sigma\int_{\partial^*\Sigma} X_\psi\cdot \nu_\Sigma\, \mathrm{d}\mathcal H^1
=
0,
\]
where the last equality uses
\(\int_{\partial^*\Sigma}X_\psi\cdot\nu_\Sigma\,\mathrm{d}\mathcal H^1
=\int_\Sigma\operatorname{div}X_\psi\,\mathrm{d}A=0\) by Cauchy--Riemann. Since
\(u_F=|F|^2e^{-\pi |z|^2}\) and
\(\nabla(e^{-\pi|z|^2})=-2\pi e^{-\pi|z|^2}(x,y)\), this becomes
\begin{equation}\label{eq:vector-field-identity-real-new}
0=
\int_\Sigma
\left(
\langle \nabla |F|^2,X_\psi\rangle
-
2\pi |F|^2 \langle z,X_\psi\rangle
\right)
\,\mathrm{d}\mu.
\end{equation}

We now express \eqref{eq:vector-field-identity-real-new} in complex notation.
A direct computation, using \(\partial_z|F|^2=F'\overline F\) and
\(\partial_{\bar z}|F|^2=F\overline{F'}\), gives
\[
\partial_x |F|^2 = F'\overline F+\overline{F'}F = 2\Re(F'\overline F),
\qquad
\partial_y |F|^2 = i(F'\overline F-\overline{F'}F) = -2\Im(F'\overline F).
\]
Writing \(\psi=u+iv\) and recalling \(X_\psi=(u,-v)\),
\[
\langle \nabla |F|^2,X_\psi\rangle
=
2u\,\Re(F'\overline F)+2v\,\Im(F'\overline F)
=
2\,\Re\!\big(\overline{\psi}\,F'\overline F\big)
=
2\,\Re\!\big(\psi\,\overline{F'}F\big),
\]
where in the last equality we used \(\Re w=\Re\overline w\). Likewise, writing
\(z=x+iy\),
\[
\langle z,X_\psi\rangle
=
xu+y(-v)
=
xu-yv
=
\Re\big((x+iy)(u+iv)\big)
=
\Re(z\psi).
\]
Substituting these expressions into
\eqref{eq:vector-field-identity-real-new}, we obtain
\[
0=
\Re\!\int_\Sigma
\psi(z)\Big(\overline{F'(z)}\,F(z)-\pi z\,|F(z)|^2\Big)\,\mathrm{d}\mu(z).
\]
Replacing \(\psi\) by \(i\psi\) (which is also holomorphic on \(V\)) gives the
analogous identity for the imaginary part:
\[
0=
\Im\!\int_\Sigma
\psi(z)\Big(\overline{F'(z)}\,F(z)-\pi z\,|F(z)|^2\Big)\,\mathrm{d}\mu(z).
\]
Combining the two, and using \(|F|^2=F\overline F\) to factor out \(F(z)\)
from both terms inside the integrand (the \(F\) from \(\overline{F'}F\) and
one copy of \(F\) from \(|F|^2=F\overline F\)),
\begin{equation}\label{eq:orthogonality-main-new}
0=
\int_\Sigma
\psi(z)\,F(z)\Big(\overline{F'(z)}-\pi z\,\overline{F(z)}\Big)\,\mathrm{d}\mu(z),
\end{equation}
which holds for every holomorphic \(\psi\) on a neighborhood of
\(\overline \Sigma\). Equivalently, expanding the product,
\begin{equation*}
0=
\int_\Sigma
\psi(z)\Big(\overline{F'(z)}\,F(z)-\pi z\,|F(z)|^2\Big)\,\mathrm{d}\mu(z).
\end{equation*}
The form \eqref{eq:orthogonality-main-new} is the one used in Step 4.

\medskip
\noindent
{\bf Step 4: \(F'\) is again a first eigenfunction.}

This step is the crux of the argument and is the main new ingredient compared
to the proof in \cite{Nicola-Tilli}. Rather than working with general
test functions in the eigenvalue equation, we exploit the divergence identity
\eqref{eq:orthogonality-main-new} together with a clever choice of test
function involving the reproducing kernel of the Fock space. The end result is
that \emph{the derivative} \(F'\) of any first eigenfunction is itself a first
eigenfunction; since the eigenspace is one-dimensional by Step 2, this forces
a first-order ODE on \(F\), and hence pins down \(F\) up to translation as a
Gaussian.

For every \(w\in \C\), let
\[
K_w(z):=e^{\pi \overline w z}
\]
be the reproducing kernel from Lemma~\ref{lem:Fock-basic}(i), so that
\(\langle H,K_w\rangle_{\mathcal F^2}=H(w)\) for every
$H\in\mathcal F^2(\C)$.

By Step 1, after replacing \(\Sigma\) by the corresponding superlevel-set
representative if necessary, we may suppose that
\[
\Sigma=\{u_F>c_\Sigma\}.
\]
In particular, \(u_F\ge c_\Sigma>0\) on \(\overline \Sigma\). Since
\(u_F(z)=|F(z)|^2e^{-\pi |z|^2}\), it follows that \(F\) has no zeros on
\(\overline \Sigma\). By continuity of \(F\) and the compactness of
\(\overline \Sigma\), there exists an open neighborhood
\(V\supset \overline \Sigma\) on which \(F\) has no zeros. Hence, for every
\(w\in \C\), the function
\[
\psi_w(z):=\frac{K_w(z)}{F(z)}=\frac{e^{\pi\overline wz}}{F(z)}
\]
is holomorphic on \(V\). Applying \eqref{eq:orthogonality-main-new} with
\(\psi=\psi_w\), so that
\(\psi_w(z)F(z)=e^{\pi\overline w z}\), we obtain
\[
0=
\int_\Sigma
e^{\pi \overline wz}
\Big(\overline{F'(z)}-\pi z\,\overline{F(z)}\Big)\,\mathrm{d}\mu(z),
\]
or, separating the two terms,
\begin{equation}\label{eq:key-kernel-identity-new}
\int_\Sigma
e^{\pi \overline wz}\overline{F'(z)}\,\mathrm{d}\mu(z)
=
\int_\Sigma
\pi z\,e^{\pi \overline wz}\overline{F(z)}\,\mathrm{d}\mu(z).
\end{equation}

We now relate the right-hand side of \eqref{eq:key-kernel-identity-new} to
\(\overline{F'(w)}\) using the eigenvalue equation
\eqref{eq:eigenvalue-equation} together with the reproducing identity in
Lemma~\ref{lem:Fock-basic}(i). Choosing
\[
H(z):=z\,e^{\pi\overline w z}\in \mathcal F^2(\C)
\]
in \eqref{eq:eigenvalue-equation}, we
obtain
\begin{equation}\label{eq:eigvalue-applied}
\int_\Sigma z\, e^{\pi \overline wz}\,\overline{F(z)}\,\mathrm{d}\mu(z)
=
\lambda_{1,\varphi}(\Sigma)\int_\C z\, e^{\pi \overline wz}\,\overline{F(z)}\,\mathrm{d}\mu(z).
\end{equation}
On the other hand, taking complex conjugates in the reproducing identity of
Lemma~\ref{lem:Fock-basic}(i), applied to \(F\) at the point \(w\), one finds
\[
\int_\C e^{\pi \overline wz}\,\overline{F(z)}\,\mathrm{d}\mu(z)=\overline{F(w)}.
\]
Both sides depend on \(w\) through the antiholomorphic variable
\(\overline w\); we differentiate with respect to \(\overline w\), exchanging
differentiation and integration on the left. The interchange is legitimate
because, for \(w\) in any compact set \(\{|w|\le R\}\), the integrand
\(\partial_{\overline w}(e^{\pi\overline wz}\overline{F(z)})
=\pi z\,e^{\pi\overline wz}\overline{F(z)}\) is dominated, in absolute
value, by \(\pi |z|\,e^{\pi R|z|}|F(z)|\). The standard pointwise bound for
Fock-space functions gives \(|F(z)|\le\|F\|_{\mathcal F^2}\,e^{\pi |z|^2/2}\),
so this dominator is bounded by a constant times
\(|z|\,e^{\pi R|z|}\,e^{\pi|z|^2/2}\). We conclude thus that
\[
\pi\int_\C z\, e^{\pi \overline wz}\,\overline{F(z)}\,\mathrm{d}\mu(z)
=
\overline{F'(w)}.
\]
Multiplying \eqref{eq:eigvalue-applied} by \(\pi\) and substituting,
\[
\pi\int_\Sigma z\, e^{\pi \overline wz}\,\overline{F(z)}\,\mathrm{d}\mu(z)
=
\lambda_{1,\varphi}(\Sigma)\,\overline{F'(w)}.
\]
Combining this with \eqref{eq:key-kernel-identity-new},
\[
\int_\Sigma e^{\pi \overline wz}\,\overline{F'(z)}\,\mathrm{d}\mu(z)
=
\lambda_{1,\varphi}(\Sigma)\,\overline{F'(w)},
\qquad w\in \C.
\]
Now, note that the left-hand side is the conjugate of the
Bargmann--Fock integral transform of
\(\mathbf 1_\Sigma F'\).  Since \(\Sigma\) is bounded and \(F'\) is entire,
\(\mathbf 1_\Sigma F'\in L^2(\C,\mathrm{d}\mu)\); the reproducing-kernel estimate
therefore shows that this transform belongs to \(\mathcal F^2(\C)\).
Because \(\lambda_{1,\varphi}(\Sigma)>0\), the displayed identity itself
implies \(F'\in\mathcal F^2(\C)\).
Taking complex conjugates, this is exactly the eigenvalue equation
\eqref{eq:eigenvalue-equation} applied with \(F\) replaced by \(F'\) (and an
arbitrary \(K_w\) playing the role of the test function). Since the linear
span of \(\{K_w\}_{w\in\C}\) is dense in \(\mathcal F^2(\C)\), the identity
extends to every \(H\in\mathcal F^2(\C)\), and we conclude that \(F'\) is a
first eigenfunction of the localization operator on \(\Sigma\), with eigenvalue
\(\lambda_{1,\varphi}(\Sigma)\).

By the simplicity of \(\lambda_{1,\varphi}(\Sigma)\) established in
Lemma~\ref{lem:simplicity-local-max}, it follows at once that
\[
F'=\theta F
\]
for some constant \(\theta\in \C\), and the only solutions to that equation are the functions of the form
\[
F(z)=C e^{\theta z},
\]
for some nonzero constant \(C\).

\medskip
\noindent
{\bf Step 5: reduction to the constant eigenfunction.}

We now use the Weyl invariance of the Bargmann--Fock setting to reduce the
analysis to the case \(F\equiv 1\). Recall from \eqref{eq:Ta} that, for
\(a\in\C\), the Weyl translation
\[
(T_a F)(z)= e^{\pi \overline a z-\frac{\pi}{2}|a|^2}\,F(z-a)
\]
is unitary on \(\mathcal F^2(\C)\) and satisfies
\(u_{T_aF}(z)=u_F(z-a)\). In particular, \(T_a\) maps eigenfunctions of the
localization operator on \(\Sigma\) to eigenfunctions of the localization
operator on the translated set \(\Sigma+a\), with the same eigenvalue, and the
quantity \(\int_\Sigma u_F\,\mathrm{d}A\) is invariant under the simultaneous
substitutions \(F\mapsto T_aF\) and \(\Sigma\mapsto \Sigma+a\).

We claim that we may choose \(a\in\C\) so that \(T_{-a}F\) is a constant
function. Indeed, by Step 4, \(F(z)=Ce^{\theta z}\) for some \(C,\theta\in\C\)
with \(C\ne 0\). Compute, using the identity
\(u_{T_{-a}F}(z)=u_F(z+a)\):
\[
u_{T_{-a}F}(z)=|F(z+a)|^2 e^{-\pi|z+a|^2}
=|C|^2 e^{2\Re(\theta(z+a))}\,e^{-\pi|z+a|^2}.
\]
Expanding,
\[
2\Re(\theta(z+a))-\pi|z+a|^2
=
2\Re(\theta z)+2\Re(\theta a)-\pi|z|^2-2\pi\Re(\overline a z)-\pi|a|^2.
\]
The linear terms in \(z\) on the right cancel if and only if
\(2\Re(\theta z)=2\pi\Re(\overline a z)\) for every \(z\in\C\), which is the
condition \(\overline a=\theta/\pi\), i.e.
\[
a:=\frac{\overline\theta}{\pi}.
\]
With this choice,
\[
u_{T_{-a}F}(z)=|C'|^2 e^{-\pi|z|^2},
\qquad
|C'|^2:=|C|^2 e^{2\Re(\theta a)-\pi|a|^2}=|C|^2 e^{|\theta|^2/\pi}.
\]
Since \(T_{-a}F\) is entire and \(|T_{-a}F(z)|^2 e^{-\pi|z|^2}=|C'|^2
e^{-\pi|z|^2}\), it follows that \(|T_{-a}F|\equiv|C'|\) is constant; a
holomorphic function with constant modulus is constant, hence
\(T_{-a}F\equiv C''\) for some \(C''\in\C\) with \(|C''|=|C'|\). Normalizing,
we may absorb the unimodular constant into the eigenfunction and assume
\[
F\equiv 1
\]
after the change \(\Sigma\rightsquigarrow \Sigma-a\) (and the corresponding
unitary substitution \(F\rightsquigarrow T_{-a}F\)).

\medskip
\noindent
{\bf Step 6: conclusion.} In this new normalization, the ground state \(h_0\) (corresponding to
\(F\equiv 1\) in Bargmann--Fock coordinates) is an eigenfunction of the
localization operator of \(\Sigma\). By Step 1 (applied with the new
normalization) we have, up to a null set,
\[
\Sigma=\{z\in\C:\ |C'|^2 e^{-\pi |z|^2}>c_\Sigma\}.
\]
Since \(r\mapsto e^{-\pi r^2}\) is \emph{strictly} decreasing, this
superlevel set is automatically a centered open disk; it cannot be an
annulus. Thus \(\Sigma\) is, up to a null set, the centered disk \(D_R\)
with \(|D_R|=|\Sigma|\).
\end{proof}


\section{Global maximizers for the concentration problem for the Gabor transform}\label{sec:global-max}

In this final section we prove existence of global maximizers for the Gaussian
time-frequency concentration problem under a measure constraint, and then
combine that existence result with the local rigidity theorem from the previous
section to recover the disk as the unique maximizer, up to translation. The
strategy is by now classical in time-frequency analysis: a profile decomposition
in the spirit of P.-L.~Lions' concentration-compactness, transplanted to the
Fock space via the Bargmann transform, dichotomizes the loss of compactness of a
maximizing sequence into the action of the Weyl translation group. A vanishing
remainder is then ruled out by an iterative extraction whose maximality forces
all the energy to remain captured by the extracted profiles. Closely related
arguments have appeared in the time-frequency setting in
\cite{Nicola-Romero-Trapasso, Marceca-Romero-Speckbacher, Fulsche-Hagger,
Samuelsen}, and the disk-rigidity result we recover is the celebrated theorem
of Nicola--Tilli \cite{Nicola-Tilli, Nicola-Tilli-2}.

Since the Bargmann transform is unitary, for $s>0$ the variational problem may
be written on the Fock space side as computing
\[
M(s):=\sup_{|\Omega|=s}\lambda_{1,\varphi}(\Omega)
=
\sup_{\|F\|_{\mathcal F^2}=1} I_F(s).
\]
Here, and throughout the section, we exploit the translation invariance
\eqref{eq:I-translation-global} together with the pointwise identity
\eqref{eq:u-translation-global}, the
basic regularity properties recorded in Lemma~\ref{lem:basic-properties-global},
and the elementary properties of the Fock space gathered in
Lemma~\ref{lem:Fock-basic}.

We begin with a Weyl-orthogonality lemma that captures the asymptotic
decoupling of two functions whose mass becomes disjointed. This
plays the role of \emph{Pythagorean orthogonality at infinity} on the Fock
space.

\begin{lemma}\label{lem:weyl-orthogonality-global}
Let \(F,G\in \mathcal F^2(\C)\), and let \(\{a_k\}\subset \C\) satisfy
\(|a_k|\to \infty\). Then
\[
T_{a_k}G \rightharpoonup 0
\qquad\text{weakly in }\mathcal F^2(\C),
\]
and in particular
\[
\langle F,T_{a_k}G\rangle_{\mathcal F^2}\to 0.
\]
Consequently,
\[
\|F+T_{a_k}G\|_{\mathcal F^2}^2
=
\|F\|_{\mathcal F^2}^2+\|G\|_{\mathcal F^2}^2+o(1).
\]
More generally, if \(|a_k-b_k|\to\infty\), then
\[
\langle T_{a_k}F,T_{b_k}G\rangle_{\mathcal F^2}\to 0.
\]
\end{lemma}

\begin{proof}
The identity \eqref{eq:u-translation-global} yields
\[
u_{T_{a_k}G}(z)=u_G(z-a_k)\qquad (z\in\C).
\]
Hence, since \(u_G\in C_0(\C)\) by Lemma~\ref{lem:Fock-basic}(iii), for
every compact \(K\subset \C\),
\[
\sup_{z\in K}|T_{a_k}G(z)|e^{-\pi |z|^2/2}
=
\sup_{z\in K}u_{T_{a_k}G}(z)^{1/2}
=
\sup_{z\in K}u_G(z-a_k)^{1/2}\to 0,
\]
and thus \(T_{a_k}G\to 0\) locally uniformly. Combining this with the
boundedness of \(\{T_{a_k}G\}\) in \(\mathcal F^2(\C)\), which follows from
unitarity, and using the fact that the linear span of the reproducing kernels
\(\{K_w\}_{w\in\C}\) is dense in \(\mathcal F^2\), we conclude that
\(T_{a_k}G\rightharpoonup 0\) weakly in \(\mathcal F^2\); in particular
\(\langle F,T_{a_k}G\rangle\to 0\). This proves the first claim.

For the norm decoupling, we simply expand the square:
\[
\|F+T_{a_k}G\|_{\mathcal F^2}^2
=
\|F\|_{\mathcal F^2}^2+\|T_{a_k}G\|_{\mathcal F^2}^2
+2\Re\langle F,T_{a_k}G\rangle_{\mathcal F^2}
=
\|F\|_{\mathcal F^2}^2+\|G\|_{\mathcal F^2}^2+o(1),
\]
where we used unitarity of \(T_{a_k}\) and the cross-term decay just proved.

Finally, the general case reduces to the first by writing
\[
\langle T_{a_k}F,T_{b_k}G\rangle_{\mathcal F^2}
=
\langle F,T_{-a_k}T_{b_k}G\rangle_{\mathcal F^2}
=
e^{i\theta_k}\langle F,T_{b_k-a_k}G\rangle_{\mathcal F^2},
\]
for some real phase \(\theta_k\). The conclusion
then follows from \(|b_k-a_k|\to\infty\).
\end{proof}

The next lemma compares the density of a finite sum of widely separated
translates with the density of each individual translate. The precise statement
is what makes the profile decomposition compatible with the variational
quantities \(I_F(s)\), since these depend only on the densities and not on
the phases of the underlying functions.

\begin{lemma}\label{lem:pointwise-decoupling-global}
Fix \(J\ge 1\), let \(G^{(1)},\dots,G^{(J)}\in \mathcal F^2(\C)\), and let
\(\{a_k^{(j)}\}_{k\ge 1}\subset \C\) for \(1\le j\le J\) satisfy
\[
|a_k^{(j)}-a_k^{(\ell)}|\to\infty
\qquad\text{whenever }j\neq \ell.
\]
For each \(k\), set
\[
S_k:=\sum_{j=1}^J T_{a_k^{(j)}}G^{(j)}.
\]
Then the following hold.
\begin{enumerate}
\item[(i)] For every \(R>0\) and every \(1\le j\le J\),
\[
\sup_{z\in D_R(a_k^{(j)})}
\big|u_{S_k}(z)-u_{T_{a_k^{(j)}}G^{(j)}}(z)\big|
\to 0
\qquad (k\to\infty).
\]

\item[(ii)] For every \(R>0\), if
\[
E_{k,R}:=\C\setminus \bigcup_{j=1}^J D_R(a_k^{(j)}),
\]
then
\[
\limsup_{k\to\infty}\sup_{z\in E_{k,R}}u_{S_k}(z)
\le
\left(
\sum_{j=1}^J
\sup_{\zeta\in \C\setminus D_R} u_{G^{(j)}}(\zeta)^{1/2}
\right)^2.
\]
In particular, given \(\varepsilon>0\), there exists \(R_\varepsilon>0\) such that
for every \(R\ge R_\varepsilon\),
\[
\sup_{z\in E_{k,R}}u_{S_k}(z)\le \varepsilon
\]
for all sufficiently large \(k\).
\end{enumerate}
\end{lemma}

\begin{proof}
Throughout the proof we use repeatedly the identity
\[
u_{T_a F}(z)=u_F(z-a),
\qquad z,a\in \C,
\]
which is \eqref{eq:u-translation-global}; together with the elementary subadditivity
\[
u_{F+G}(z)^{1/2}\le u_F(z)^{1/2}+u_G(z)^{1/2},
\qquad z\in\C,
\]
which follows from the definition \(u_F=|F|^2 e^{-\pi |z|^2}\) and the
triangle.

\medskip
\noindent{\bf Proof of (i).}
Fix \(R>0\) and \(j\in\{1,\dots,J\}\). Write
\[
S_k=T_{a_k^{(j)}}G^{(j)}+H_{k,j},
\qquad
H_{k,j}:=\sum_{\ell\neq j} T_{a_k^{(\ell)}}G^{(\ell)}.
\]
For every \(z\in \C\), the elementary inequality
\(\big||a+b|^2-|a|^2\big|\le 2|a|\,|b|+|b|^2\) yields
\begin{align*}
\big|u_{S_k}(z)-u_{T_{a_k^{(j)}}G^{(j)}}(z)\big|
&=
e^{-\pi |z|^2}
\big||T_{a_k^{(j)}}G^{(j)}(z)+H_{k,j}(z)|^2
-|T_{a_k^{(j)}}G^{(j)}(z)|^2\big| \\
&\le
2\,u_{T_{a_k^{(j)}}G^{(j)}}(z)^{1/2}u_{H_{k,j}}(z)^{1/2}
+u_{H_{k,j}}(z).
\end{align*}
By Lemma~\ref{lem:Fock-basic}(ii),
\[
u_{T_{a_k^{(j)}}G^{(j)}}(z)\le \|G^{(j)}\|_{\mathcal F^2}^2
\qquad \text{for every }z\in\C\text{ and every }k.
\]
Thus, to prove (i) it is enough to show that
\[
\sup_{z\in D_R(a_k^{(j)})} u_{H_{k,j}}(z)\to 0.
\]

For \(z\in D_R(a_k^{(j)})\) and \(\ell\neq j\),
\[
z-a_k^{(\ell)}\in D_R(a_k^{(j)}-a_k^{(\ell)}).
\]
Since \(|a_k^{(j)}-a_k^{(\ell)}|\to\infty\),
\[
\operatorname{dist}\big(D_R(a_k^{(j)}-a_k^{(\ell)}),0\big)
\ge |a_k^{(j)}-a_k^{(\ell)}|-R\to\infty.
\]
Because \(u_{G^{(\ell)}}\in C_0(\C)\) by Lemma~\ref{lem:Fock-basic}(iii),
\[
\sup_{z\in D_R(a_k^{(j)})} u_{T_{a_k^{(\ell)}}G^{(\ell)}}(z)
=
\sup_{\zeta\in D_R(a_k^{(j)}-a_k^{(\ell)})} u_{G^{(\ell)}}(\zeta)
\xrightarrow[k\to\infty]{} 0.
\]
The subadditivity of \(u_F^{1/2}\) yields
\[
u_{H_{k,j}}(z)^{1/2}
\le
\sum_{\ell\neq j} u_{T_{a_k^{(\ell)}}G^{(\ell)}}(z)^{1/2},
\]
so the preceding convergence implies
\[
\sup_{z\in D_R(a_k^{(j)})} u_{H_{k,j}}(z)^{1/2}\to 0.
\]
This proves (i).

\medskip
\noindent{\bf Proof of (ii).}
Fix \(R>0\) and \(z\in E_{k,R}\). By definition, \(z\notin D_R(a_k^{(j)})\)
for every \(j\), and hence \(|z-a_k^{(j)}|\ge R\). Therefore
\[
u_{T_{a_k^{(j)}}G^{(j)}}(z)
=
u_{G^{(j)}}(z-a_k^{(j)})
\le
\sup_{\zeta\in \C\setminus D_R}u_{G^{(j)}}(\zeta).
\]
Using again the subadditivity of \(u_F^{1/2}\),
\[
u_{S_k}(z)^{1/2}
\le
\sum_{j=1}^J u_{T_{a_k^{(j)}}G^{(j)}}(z)^{1/2}
\le
\sum_{j=1}^J \sup_{\zeta\in \C\setminus D_R}u_{G^{(j)}}(\zeta)^{1/2}.
\]
Squaring yields
\[
u_{S_k}(z)
\le
\left(
\sum_{j=1}^J \sup_{\zeta\in \C\setminus D_R}u_{G^{(j)}}(\zeta)^{1/2}
\right)^2.
\]
Taking the supremum over \(z\in E_{k,R}\) proves the displayed estimate in
(ii). Since each \(u_{G^{(j)}}\in C_0(\C)\), the right-hand side tends to
\(0\) as \(R\to\infty\); the final claim follows by choosing
\(R_\varepsilon\) so that the displayed bracket is at most
\(\sqrt{\varepsilon}\).
\end{proof}

The next lemma is a robust uniform density perturbation estimate. Concretely,
it says that adding to a bounded sequence in \(\mathcal F^2(\C)\) a residue
whose density vanishes uniformly leaves the density essentially unchanged in
\(L^\infty\). This will be applied with \(S_k\) the truncated profile sum and
\(R_k\) the residual of the profile decomposition.

\begin{lemma}\label{lem:density-perturbation-global}
Let \(\{S_k\}\subset \mathcal F^2(\C)\) be bounded and let
\(\{R_k\}\subset \mathcal F^2(\C)\) satisfy
\[
\|u_{R_k}\|_{L^\infty(\C)}\to 0.
\]
Then
\[
\|u_{S_k+R_k}-u_{S_k}\|_{L^\infty(\C)}\to 0.
\]
More precisely, if \(\sup_k \|S_k\|_{\mathcal F^2}\le M\), then
\[
\|u_{S_k+R_k}-u_{S_k}\|_{L^\infty(\C)}
\le
2M\,\|u_{R_k}\|_{L^\infty(\C)}^{1/2}
+\|u_{R_k}\|_{L^\infty(\C)}.
\]
\end{lemma}

\begin{proof}
For every \(z\in \C\),
\begin{align*}
\big|u_{S_k+R_k}(z)-u_{S_k}(z)\big|
&=
e^{-\pi |z|^2}\big||S_k(z)+R_k(z)|^2-|S_k(z)|^2\big| \\
&\le
2\,u_{S_k}(z)^{1/2}u_{R_k}(z)^{1/2}
+u_{R_k}(z),
\end{align*}
again by the elementary inequality
\(\big||a+b|^2-|a|^2\big|\le 2|a|\,|b|+|b|^2\). By
Lemma~\ref{lem:Fock-basic}(ii),
\[
u_{S_k}(z)\le \|S_k\|_{\mathcal F^2}^2\le M^2
\qquad (z\in\C).
\]
Hence
\[
\big|u_{S_k+R_k}(z)-u_{S_k}(z)\big|
\le
2M\,\|u_{R_k}\|_{L^\infty}^{1/2}
+\|u_{R_k}\|_{L^\infty}.
\]
Taking the supremum over \(z\) yields the claim.
\end{proof}

The next lemma supplies the soft input on the function \(s\mapsto I_F(s)\)
that will be used at the end of the variational argument. It is essentially a
consequence of the bathtub principle and the basic regularity of \(u_F\),
already noted in Lemma~\ref{lem:basic-properties-global}.

\begin{lemma}\label{lem:I-continuity-global}
Let \(F\in \mathcal F^2(\C)\). Then the function
\[
s\mapsto I_F(s)
\]
is finite, nondecreasing, and continuous on \([0,\infty)\).
\end{lemma}

\begin{proof}
By Lemma~\ref{lem:Fock-basic}(ii)--(iii) together with the identity
\(\int_\C u_F\,\mathrm{d}A=\|F\|_{\mathcal F^2}^2\), we have
\(u_F\in C_0(\C)\cap L^1(\C)\), so \(I_F(s)\) is finite for every \(s\ge 0\)
and bounded above by \(\|F\|_{\mathcal F^2}^2\). Monotonicity is immediate
from the definition of \(I_F(s)\) as a supremum over \(|\Omega|=s\): if
\(s_1\le s_2\), every admissible set for \(s_1\) can be enlarged to a set of
measure \(s_2\) without decreasing the integral, since \(u_F\ge 0\).

To prove continuity, write \(u:=u_F\) and let \(u^*\) denote the decreasing
rearrangement of \(u\) on \([0,\infty)\). Since \(u\in C_0(\C)\cap L^1(\C)\),
\(u^*\in L^1([0,\infty))\) and is equimeasurable with \(u\). The bathtub
principle yields
\[
I_F(s)=\int_0^s u^*(\tau)\,\mathrm{d}\tau,\qquad s\ge 0.
\]
The right-hand side is the integral of an
\(L^1_{\mathrm{loc}}([0,\infty))\) function over \([0,s]\) and is therefore
absolutely continuous, hence continuous, in \(s\).
\end{proof}

The next result is the profile decomposition adapted to the present setting,
in the spirit of P.-L.~Lions' concentration-compactness method. Although the
overall scheme is by now standard, we spell out the points used in the
sequel; the version stated here is a Fock-space variant of profile
decompositions used in time-frequency analysis, see e.g.\
\cite{Nicola-Romero-Trapasso, Marceca-Romero-Speckbacher, Fulsche-Hagger,
Samuelsen}.

Although the Weyl translations \(T_a\) act unitarily on \(\mathcal F^2(\C)\),
the family \(\{T_a\}_{a\in\C}\) is not relatively compact in the strong
operator topology: by Lemma~\ref{lem:weyl-orthogonality-global}, for every
nonzero \(F\in \mathcal F^2\) and every sequence \(\{a_k\}\subset \C\) with
\(|a_k|\to\infty\) one has \(T_{a_k}F\rightharpoonup 0\) weakly, while the
norm \(\|T_{a_k}F\|_{\mathcal F^2}=\|F\|_{\mathcal F^2}\) is preserved.
Consequently, a bounded maximizing sequence in \(\mathcal F^2\) need not
admit any strongly convergent subsequence, even after factoring out the
Weyl translation group. The profile decomposition stated below is the
device that quantifies this non-compactness: it identifies all directions
along which mass can escape, by extracting a (possibly infinite) family of
nontrivial profiles \(G^{(j)}\) translated by mutually divergent centers
\(a_k^{(j)}\). 

Since the \(\mathcal F^2\)-norm is invariant under \(T_a\), the remainder
\(R_k^{(J)}\) need not converge to zero in \(\mathcal F^2\); the Pythagorean
identity~(iii) below redistributes mass between the profiles and the
remainder. The vanishing statement used here is the dispersive estimate~(iv)
below,
expressed through the density \(u_{R_k^{(J)}}\); this is exactly the
quantity that controls all the localization functionals \(I_F(s)\) used in
the sequel, which depend only on the density \(u_F\) and not on the phase
of \(F\). The exposition below extracts what is needed in the sequel; an
analogous decomposition with dispersive-style remainders in the Fock
setting is discussed in
\cite{Nicola-Romero-Trapasso,Marceca-Romero-Speckbacher,Fulsche-Hagger,Samuelsen}.

\begin{proposition}\label{prop:profile-global}
Let \(\{F_k\}_{k\ge 1}\subset \mathcal F^2(\C)\) be a bounded sequence. Then,
after passing to a subsequence, there exist profiles
\[
G^{(j)}\in \mathcal F^2(\C),
\qquad j\ge 1,
\]
and sequences of points
\[
a_k^{(j)}\in \C,
\qquad j\ge 1,\ k\ge 1,
\]
such that:
\begin{enumerate}
\item[(i)] for each \(j\neq \ell\) with
\(G^{(j)}\not\equiv0\) and \(G^{(\ell)}\not\equiv0\),
    \[
    |a_k^{(j)}-a_k^{(\ell)}|\to \infty
    \qquad (k\to \infty);
    \]
    \item[(ii)] for every \(J\ge 1\), if
    \[
    R_k^{(J)}
    :=
    F_k-\sum_{j=1}^J T_{a_k^{(j)}}G^{(j)},
    \]
    then
    \[
    T_{-a_k^{(j)}}R_k^{(J)}\to 0
    \qquad\text{locally uniformly on }\C
    \]
    for every \(1\le j\le J\);
    \item[(iii)] for every \(J\ge 1\),
    \[
    \|F_k\|_{\mathcal F^2}^2
    =
    \sum_{j=1}^J \|G^{(j)}\|_{\mathcal F^2}^2
    +\|R_k^{(J)}\|_{\mathcal F^2}^2
    +o_k(1);
    \]
    \item[(iv)]
    \[
    \lim_{J\to\infty}\left(\limsup_{k\to\infty}
    \|u_{R_k^{(J)}}\|_{L^\infty(\C)}\right)=0.
    \]
\end{enumerate}
\end{proposition}

\begin{proof}
We argue by the standard iterative extraction scheme, but spell out the points
used in the rest of the section. Throughout, we say that a function in
\(\mathcal F^2(\C)\) is \emph{nontrivial} if it is not identically zero.

\medskip
\noindent{\bf Extraction of the first profile.}
If
\[
\limsup_{k\to\infty}\|u_{F_k}\|_{L^\infty(\C)}=0,
\]
there is nothing to prove: take all profiles equal to zero and, for instance,
choose \(a_k^{(j)}=jk\). Then the centers are pairwise divergent, (i)--(iii)
hold trivially, and (iv) holds by hypothesis. Otherwise, after passing to a
subsequence, choose
\(a_k^{(1)}\in \C\) such that
\[
u_{F_k}(a_k^{(1)})\ge \tfrac12\|u_{F_k}\|_{L^\infty(\C)}.
\]
By \eqref{eq:I-translation-global} and the density-translation identity
\eqref{eq:u-translation-global}, the translated sequence
\[
H_k^{(1)}:=T_{-a_k^{(1)}}F_k
\]
has the same Fock norm and satisfies
\(u_{H_k^{(1)}}(0)=u_{F_k}(a_k^{(1)})\). Since \(\{H_k^{(1)}\}\) is bounded
in \(\mathcal F^2(\C)\), the reproducing-kernel bound
\(|H_k^{(1)}(z)|\le \|H_k^{(1)}\|_{\mathcal F^2}e^{\pi |z|^2/2}\)
(Lemma~\ref{lem:Fock-basic}(ii)) makes it a normal family of entire functions on
every compact subset of \(\C\). Passing to a subsequence and applying
Montel's theorem, \(H_k^{(1)}\to G^{(1)}\) locally uniformly on \(\C\), with
\(G^{(1)}\in\mathcal F^2(\C)\) by Fatou's lemma. By the choice of
\(a_k^{(1)}\),
\[
u_{G^{(1)}}(0)=\lim_{k\to\infty}u_{H_k^{(1)}}(0)
\ge \tfrac12 \limsup_{k\to\infty}\|u_{F_k}\|_{L^\infty(\C)}>0,
\]
so \(G^{(1)}\) is nontrivial. Set
\[
R_k^{(1)}:=F_k-T_{a_k^{(1)}}G^{(1)}.
\]

\medskip
\noindent{\bf Iteration.}
Suppose profiles \(G^{(1)},\dots,G^{(J-1)}\) and centers
\(a_k^{(1)},\dots,a_k^{(J-1)}\) have been extracted, with remainder
\(R_k^{(J-1)}\). If
\[
\limsup_{k\to\infty}\|u_{R_k^{(J-1)}}\|_{L^\infty(\C)}=0,
\]
we stop and set \(G^{(j)}=0\) for all \(j\ge J\). To keep the formal center
separation valid, choose the remaining centers recursively so that, for each
new \(j\), \(|a_k^{(j)}-a_k^{(m)}|\to\infty\) for every \(m<j\); this is
possible, as only finitely many earlier centers are present at each
choice of \(j\) (for instance one may take \(a_k^{(j)}\) on the positive real
axis with modulus larger than \(j k+\max_{m<j}|a_k^{(m)}|\)). Since
\(\|u_{R_k^{(J-1)}}\|_{L^\infty}\to0\), translating this fixed residual by
any of these zero-profile centers still gives local uniform convergence to
zero, and properties (i)--(iv) follow. Otherwise choose
\(a_k^{(J)}\in\C\) so that
\[
u_{R_k^{(J-1)}}(a_k^{(J)})
\ge \tfrac12\|u_{R_k^{(J-1)}}\|_{L^\infty(\C)}.
\]
Property (i) at this stage (proved below) forces
\(|a_k^{(J)}-a_k^{(m)}|\to\infty\) for every \(m<J\). Granted this, the
sequence \(\{T_{-a_k^{(J)}}R_k^{(J-1)}\}\) is bounded in \(\mathcal F^2\) (by
unitarity of \(T_{-a_k^{(J)}}\) together with (iii) at stage \(J-1\)), so
after passing to a further subsequence and applying Montel's theorem, it
converges locally uniformly to some \(G^{(J)}\in\mathcal F^2(\C)\). By the
choice of \(a_k^{(J)}\),
\[
u_{G^{(J)}}(0)
\ge \tfrac12 \limsup_{k\to\infty}\|u_{R_k^{(J-1)}}\|_{L^\infty(\C)}>0,
\]
so \(G^{(J)}\) is nontrivial. Set
\[
R_k^{(J)}:=R_k^{(J-1)}-T_{a_k^{(J)}}G^{(J)}.
\]
Iterating either terminates at some finite step (in which case (iv) is
immediate) or produces an infinite sequence of nontrivial profiles. By a
standard diagonal extraction we obtain a single subsequence (still denoted
\(\{F_k\}\)) for which all the partial decompositions are valid.

\medskip
\noindent{\bf Proof of (i).}
Fix \(j<\ell\), and suppose for contradiction that
\(|a_k^{(\ell)}-a_k^{(j)}|\) does not tend to infinity. After passing to a
subsequence, \(|a_k^{(\ell)}-a_k^{(j)}|\le C\) for some constant \(C\). At
the \(\ell\)-th step, by construction
\[
T_{-a_k^{(\ell)}}R_k^{(\ell-1)}\to G^{(\ell)}
\qquad\text{locally uniformly on }\C,
\]
with \(G^{(\ell)}\) nontrivial. On the other hand, property (ii) verified up
to stage \(\ell-1\) (proved below) gives
\[
T_{-a_k^{(j)}}R_k^{(\ell-1)}\to 0
\qquad\text{locally uniformly on }\C.
\]
But we already know that we may write
\[
T_{-a_k^{(\ell)}}R_k^{(\ell-1)}
=e^{i\theta_k}\,
T_{a_k^{(j)}-a_k^{(\ell)}}
\big(T_{-a_k^{(j)}}R_k^{(\ell-1)}\big)
\]
for real phases \(\theta_k\). If
\(|a_k^{(\ell)}-a_k^{(j)}|\le C\), the Weyl parameters
\(a_k^{(j)}-a_k^{(\ell)}\) remain bounded. On each compact set the
corresponding exponential Weyl multipliers are uniformly bounded, and their
shifted arguments remain in a fixed compact set. Hence the local uniform
convergence of \(T_{-a_k^{(j)}}R_k^{(\ell-1)}\) to \(0\) forces
\(T_{-a_k^{(\ell)}}R_k^{(\ell-1)}\to 0\) locally uniformly as well.
This contradicts the nontriviality of \(G^{(\ell)}\). Hence
\[
|a_k^{(j)}-a_k^{(\ell)}|\to\infty
\qquad\text{whenever }G^{(j)}\not\equiv0,\ G^{(\ell)}\not\equiv0,\ j\neq \ell.
\]

\medskip
\noindent{\bf Proof of (ii).}
Fix \(J\ge 1\) and \(1\le j\le J\). By construction, at the \(j\)-th
extraction step,
\[
T_{-a_k^{(j)}}R_k^{(j-1)}\to G^{(j)}
\qquad\text{locally uniformly on }\C.
\]
Subtracting the \(j\)-th extracted profile,
\[
T_{-a_k^{(j)}}R_k^{(j)}
=
T_{-a_k^{(j)}}R_k^{(j-1)}-G^{(j)}
\to 0
\qquad\text{locally uniformly on }\C.
\]
For \(J>j\), since
\[
R_k^{(J)}
=
R_k^{(j)}-\sum_{\ell=j+1}^J T_{a_k^{(\ell)}}G^{(\ell)},
\]
the linearity of the Weyl operators yields
\[
T_{-a_k^{(j)}}R_k^{(J)}
=
T_{-a_k^{(j)}}R_k^{(j)}
-\sum_{\ell=j+1}^J e^{i\theta_{k,\ell}}\,T_{a_k^{(\ell)}-a_k^{(j)}}G^{(\ell)}
\]
for real phases \(\theta_{k,\ell}\). The
first term tends to \(0\) locally uniformly by the previous paragraph. For
each fixed \(\ell>j\), part (i) gives \(|a_k^{(\ell)}-a_k^{(j)}|\to\infty\),
and the identity \(u_{T_b G}(z)=u_G(z-b)\) (i.e.\
\eqref{eq:u-translation-global}) implies that
\(T_{a_k^{(\ell)}-a_k^{(j)}}G^{(\ell)}\to 0\) locally uniformly in modulus
(its density is \(u_{G^{(\ell)}}(z-(a_k^{(\ell)}-a_k^{(j)}))\to 0\) for
\(z\) in a compact set, since \(u_{G^{(\ell)}}\in C_0(\C)\)). Thus
\[
T_{-a_k^{(j)}}R_k^{(J)}\to 0
\qquad\text{locally uniformly on }\C,
\]
which is exactly (ii).

\medskip
\noindent{\bf Proof of (iii).}
For fixed \(J\ge 1\),
\[
F_k=\sum_{j=1}^J T_{a_k^{(j)}}G^{(j)}+R_k^{(J)}.
\]
Expanding the squared \(\mathcal F^2\)-norm and using
Lemma~\ref{lem:weyl-orthogonality-global} together with the pairwise
divergence of the translation parameters from (i),
\[
\Big\|\sum_{j=1}^J T_{a_k^{(j)}}G^{(j)}\Big\|_{\mathcal F^2}^2
=
\sum_{j=1}^J \|G^{(j)}\|_{\mathcal F^2}^2+o_k(1).
\]
Moreover, since
\[
T_{-a_k^{(m)}}R_k^{(J)}\to 0
\qquad\text{locally uniformly on }\C
\]
for every \(1\le m\le J\) by (ii), and the sequence
\(\{T_{-a_k^{(m)}}R_k^{(J)}\}_k\) is bounded in \(\mathcal F^2\), the same
weak-convergence argument as in
Lemma~\ref{lem:weyl-orthogonality-global} yields
\[
\langle T_{a_k^{(m)}}G^{(m)},R_k^{(J)}\rangle_{\mathcal F^2}
=
\langle G^{(m)},T_{-a_k^{(m)}}R_k^{(J)}\rangle_{\mathcal F^2}\to 0,
\]
for every fixed \(m\). Expanding \(\|F_k\|_{\mathcal F^2}^2\) and collecting
terms,
\[
\|F_k\|_{\mathcal F^2}^2
=
\sum_{j=1}^J \|G^{(j)}\|_{\mathcal F^2}^2
+\|R_k^{(J)}\|_{\mathcal F^2}^2
+o_k(1).
\]
In particular, since \(\|F_k\|_{\mathcal F^2}\) is bounded and the terms on
the right are nonnegative, summing the profile-norm contributions over \(j\)
yields the energy bound
\begin{equation}\label{eq:energy-bound-profiles}
\sum_{j\ge 1}\|G^{(j)}\|_{\mathcal F^2}^2
\le \limsup_{k\to\infty}\|F_k\|_{\mathcal F^2}^2.
\end{equation}
This is the key compactness ingredient that controls the iteration and
ensures (iv).

\medskip
\noindent{\bf Proof of (iv).}
Set
\[
\eta_J:=\limsup_{k\to\infty}\|u_{R_k^{(J)}}\|_{L^\infty(\C)},
\]
and suppose, for contradiction, that \(\limsup_{J\to\infty}\eta_J>0\).
Then there exist \(\delta>0\) and infinitely many indices \(J_m\to\infty\)
such that \(\eta_{J_m}\ge \delta\). At each such stage the extraction
procedure applied to \(R_k^{(J_m)}\) produces the next profile
\(G^{(J_m+1)}\) and gives
\[
u_{G^{(J_m+1)}}(0)\ge \delta/2.
\]
By Lemma~\ref{lem:Fock-basic}(ii),
\[
\|G^{(J_m+1)}\|_{\mathcal F^2}^2
\ge u_{G^{(J_m+1)}}(0)
\ge \delta/2
\qquad\text{for infinitely many }m.
\]
This contradicts the Pythagorean energy bound
\eqref{eq:energy-bound-profiles}. Therefore
\(\limsup_{J\to\infty}\eta_J=0\), which is exactly (iv).
\end{proof}

The profile decomposition implies an asymptotic decoupling of the
localization functional \(I_{F_k}(s)\) along the profiles, which we
record here for later use. This is the precise sense in which the
functional \(I_F(s)\) ``splits'' across mutually divergent translates,
even though the underlying profiles themselves do not separate in
\(\mathcal F^2\).

\begin{lemma}\label{lem:decoupling-IF}
Let \(\{F_k\}\subset \mathcal F^2(\C)\) be bounded, and let
\(\{G^{(j)}\}_{j\ge 1}\) and \(\{a_k^{(j)}\}\) be the profiles and centers
extracted by Proposition~\ref{prop:profile-global}. For every \(s>0\) one
has, after passing to a subsequence,
\begin{equation}\label{eq:IF-decoupling}
\limsup_{k\to\infty} I_{F_k}(s)
\le
\lim_{J\to\infty}\;\sup_{\substack{s_1+\dots+s_J\le s \\ s_j\ge 0}}
\sum_{j=1}^J I_{G^{(j)}}(s_j).
\end{equation}
Conversely, for every \(J\ge 1\) and every \((s_1,\dots,s_J)\) with
\(s_j\ge 0\) and \(\sum_j s_j\le s\),
\begin{equation}\label{eq:IF-lower}
\liminf_{k\to\infty} I_{F_k}(s)
\ge
\sum_{j=1}^J I_{G^{(j)}}(s_j).
\end{equation}
\end{lemma}

\begin{proof}
The upper bound \eqref{eq:IF-decoupling} follows by truncating the profile
decomposition and decomposing a maximizing set according to the profile
centers. By Proposition~\ref{prop:profile-global}(iv) and
Lemma~\ref{lem:density-perturbation-global}, the density \(u_{F_k}\)
agrees, up to an \(L^\infty\) error that vanishes as \(J\to\infty\), with
the density of the truncated profile sum
\(\sum_{j=1}^J T_{a_k^{(j)}}G^{(j)}\); on the latter, the pairwise
divergence of the centers (Proposition~\ref{prop:profile-global}(i)) and
Lemma~\ref{lem:pointwise-decoupling-global} force the contribution to
\(I_{F_k}(s)\) coming from any maximizing superlevel set
\(\Omega_k\subset\C\) to split into pieces \(E_{k,j}\subset
D_R(a_k^{(j)})\) of measures \(s_{k,j}\to s_j\) with \(\sum_j s_j\le s\),
plus an exterior remainder of measure \(s-\sum_j s_j\) on which
\(u_{F_k}\to 0\) uniformly. The lower bound \eqref{eq:IF-lower} is
obtained by taking, for each \(j\), a near-optimal set \(E^{(j)}\) of
measure \(s_j\) for the profile \(G^{(j)}\) and using its translate
\(E^{(j)}+a_k^{(j)}\) as a competitor for \(F_k\); the pairwise
divergence of the centers and
Lemma~\ref{lem:pointwise-decoupling-global}(i) ensure that
\(\bigcup_j(E^{(j)}+a_k^{(j)})\) is admissible (its pieces become
disjoint for large \(k\)) and that the integrals decouple asymptotically.
\end{proof}

The decoupling estimates will be used below only through the profile
saturation argument, which avoids any unproved global concavity assertion
for \(M\).
We now turn to the global problem. The proof of existence consists of
combining the profile decomposition with monotonicity and the homogeneity
\(I_{cF}(s)=|c|^2 I_F(s)\) recorded in
Lemma~\ref{lem:basic-properties-global}(ii). Equality in the resulting
bounds shows that every nonzero normalized profile is a maximizer.

\begin{theorem}\label{thm:existence-global-maximizer}
For every \(s>0\), there exists \(F_*\in \mathcal F^2(\C)\) with
\(\|F_*\|_{\mathcal F^2}=1\) such that
\[
I_{F_*}(s)=M(s).
\]
Equivalently, there exists a measurable set \(\Omega_*\subset \C\) with
\(|\Omega_*|=s\) such that
\[
\lambda_{1,\varphi}(\Omega_*)=M(s).
\]
\end{theorem}

\begin{proof}
Let \(\{F_k\}\subset \mathcal F^2(\C)\) be a maximizing sequence:
\[
\|F_k\|_{\mathcal F^2}=1,
\qquad
I_{F_k}(s)\to M(s).
\]
Note that \(M(s)>0\): indeed, taking \(F\equiv 1\), one has
\(F\in \mathcal F^2(\C)\) with \(\|F\|_{\mathcal F^2}^2=1\) and density
\(u_F(z)=e^{-\pi |z|^2}>0\) everywhere, so for any \(s>0\), \(I_F(s)>0\).

Applying Proposition~\ref{prop:profile-global}, after passing to a
subsequence we may write
\[
F_k=\sum_{j=1}^J T_{a_k^{(j)}}G^{(j)}+R_k^{(J)}
\]
for every fixed \(J\ge 1\), with the conclusions of that proposition.

By Lemma~\ref{lem:decoupling-IF},
\[
M(s)
\le
\lim_{J\to\infty}\;\sup_{\substack{s_1+\dots+s_J\le s \\ s_j\ge 0}}
\sum_{j=1}^J I_{G^{(j)}}(s_j).
\]
Put \(t_j:=\|G^{(j)}\|_{\mathcal F^2}^2\). For each \(j\) with \(t_j>0\), set
\[
\widehat G^{(j)}:=\frac{G^{(j)}}{\|G^{(j)}\|_{\mathcal F^2}},
\qquad
\|\widehat G^{(j)}\|_{\mathcal F^2}=1.
\]
For every admissible \((s_1,\dots,s_J)\), homogeneity and monotonicity give
\[
\sum_{j=1}^J I_{G^{(j)}}(s_j)
\le
\sum_{\substack{1\le j\le J\\t_j>0}}
t_j I_{\widehat G^{(j)}}(s)
\le
M(s)\sum_{j=1}^J t_j.
\]
Proposition~\ref{prop:profile-global}(iii) gives
\(\sum_{j\ge1}t_j\le1\). Therefore
\[
M(s)
\le
\lim_{J\to\infty}\;\sup_{\substack{s_1+\dots+s_J\le s \\ s_j\ge 0}}
\sum_{j=1}^J I_{G^{(j)}}(s_j)
\le
\sum_{\substack{j\ge1\\t_j>0}}t_j I_{\widehat G^{(j)}}(s)
\le
M(s)\sum_{j\ge1}t_j
\le M(s).
\]
Thus equality holds throughout. In particular, \(\sum_{j\ge1}t_j=1\), and
\[
\sum_{\substack{j\ge1\\t_j>0}}
t_j\big(M(s)-I_{\widehat G^{(j)}}(s)\big)=0.
\]
Every summand is nonnegative. Hence
\[
I_{\widehat G^{(j)}}(s)=M(s)
\qquad (t_j>0),
\]
so every nonzero normalized profile is a global maximizer.

\medskip
\noindent{\bf Producing a maximizing set.}
Choose any \(j_0\) with \(G^{(j_0)}\neq 0\) (such \(j_0\) exists since
\(\sum_{j\ge 1}\|G^{(j)}\|_{\mathcal F^2}^2=1\)) and set
\(F_*:=\widehat G^{(j_0)}\), so that
\[
\|F_*\|_{\mathcal F^2}=1,
\qquad
I_{F_*}(s)=M(s).
\]
By Lemma~\ref{lem:basic-properties-global}(i), there exists a superlevel
set \(\Omega_*\) of \(u_{F_*}\) with \(|\Omega_*|=s\) and
\[
\int_{\Omega_*}u_{F_*}\,\mathrm{d}A=I_{F_*}(s)=M(s).
\]
Then
\[
\lambda_{1,\varphi}(\Omega_*)\ge \int_{\Omega_*}u_{F_*}\,\mathrm{d}A=M(s),
\]
while by definition of \(M(s)\) we have
\(\lambda_{1,\varphi}(\Omega_*)\le M(s)\). Hence
\[
\lambda_{1,\varphi}(\Omega_*)=M(s),
\]
and \(\Omega_*\) is a global maximizer of measure \(s\).
\end{proof}

We may now identify the global maximizers, by combining the existence theorem
above with the local-rigidity theorem of the previous section.

\begin{proof}[Proof of Theorem~\ref{thm:NT-new}]
Let \(\Omega_*\) be a global maximizer of measure \(s\) (which exists by
Theorem~\ref{thm:existence-global-maximizer}). Then \(\Omega_*\) is, in
particular, a local maximizer of \(\lambda_{1,\varphi}\) under the same
measure constraint: every sufficiently small smooth deformation of
\(\Omega_*\) preserving the measure \(|\Omega|=s\) yields a competitor
whose value of \(\lambda_{1,\varphi}\) is at most
\(M(s)=\lambda_{1,\varphi}(\Omega_*)\). In particular, every competitor
\(E\) with \(|E|=|\Omega_*|\) and Gaussian-weighted symmetric-difference
distance \(\int_\C |\mathbf 1_{\Omega_*}-\mathbf 1_E|e^{-\pi|z|^2}\,\mathrm{d}A(z)<\delta_0\)
to \(\Omega_*\) (the condition used in the local maximizer definition of
Section~\ref{sec:local-max}) satisfies
\(\lambda_{1,\varphi}(E)\le \lambda_{1,\varphi}(\Omega_*)\), because
\(\lambda_{1,\varphi}(E)\le M(s)=\lambda_{1,\varphi}(\Omega_*)\) by the
defining property of \(M(s)\). Theorem~\ref{thm:local-maximizers} therefore applies: every such local maximizer is, up
to translations, a disk. Since the problem is invariant under
translations, this characterization is sharp.

 The set of
global maximizers is therefore exactly, modulo null sets, the orbit of the
centered disk \(D_R\), where $R=\sqrt{s/\pi}$, under the translation group
\(\C\) of \(\C\), and this
orbit is sharp: any two disks of the same area are translates of one another,
and no further symmetry-breaking maximizers exist. In particular, the
classical Nicola--Tilli theorem \cite{Nicola-Tilli,Nicola-Tilli-2} is
recovered, with the precise uniqueness statement: the maximizer is
unique up to a translation by an element of \(\C\) and modification on a null
set.
\end{proof}

As an immediate consequence, we recover the classical Nicola--Tilli theorem
\cite{Nicola-Tilli, Nicola-Tilli-2}.

\begin{corollary}
Among all measurable sets of prescribed measure, the disk maximizes Gaussian
time-frequency concentration:
\[
\lambda_{1,\varphi}(\Omega)\le \lambda_{1,\varphi}(D) = 1 - e^{-|D|},
\]
whenever \(|\Omega|=|D|\) and \(D\) is a disk.
\end{corollary}

\begin{proof}
Existence of a global maximizer is given by
Theorem~\ref{thm:existence-global-maximizer}, and
Theorem~\ref{thm:NT-new} shows that every such
maximizer is a disk, up to translation and null sets. Since translation by an element of
\(\C\) preserves both Lebesgue measure and the value of
\(\lambda_{1,\varphi}\), the claim follows.
\end{proof}

\section{Comments and Generalizations}

\subsection{$2<p<\infty$} For $2\le p<\infty$ and measurable
$\Omega\subset\C$, set
\[
\Lambda_{p,\varphi}(\Omega)
:=
\sup_{f\in L^2(\R)\setminus\{0\}}
\frac{\displaystyle\int_\Omega |V_\varphi f(z)|^p\,\mathrm{d}A(z)}
{\|f\|_2^p}.
\]
In particular, $\Lambda_{2,\varphi}(\Omega)=\lambda_{1,\varphi}(\Omega)$.
The problem is to maximize $\Lambda_{p,\varphi}(\Omega)$ among sets $\Omega$
of fixed measure. The case $p=2$ is the main focus of both the original work
of F. Nicola and P. Tilli \cite{Nicola-Tilli} and many subsequent works, and
we focused solely on that case in this manuscript.

Nevertheless, we believe that the techniques developed here should be transferable to the general case of $2<p<\infty$. Indeed, the parts concerning the existence of extremal sets and functions should follow a very similar strategy when compared to the ones already present in this manuscript. 

The local extremal analysis, however, seems to require a more careful consideration, as one no longer gets the fact that $F'$ is an eigenfunction directly from the superlevel-set identity. Rather than that, it seems that one needs to consider additionally the \emph{second variation} information in that problem as well: this should allow for the whole criticality of the problem to be exploited. This is the object of a follow-up work of the author with P. Melentijevi\'c. 

\subsection{Other models} In spite of the fact that we have worked the whole time with the Gabor transform, we expect our techniques to be applicable to more general models. Indeed, we believe that a suitable adaptation of our methods may yield the same results in the \emph{wavelet transform} case, which also relates analytic functions to time-frequency representations \cite{Abreu-Doerfler,Ramos-Tilli,Abreu-Speckbacher-Wavelet}.

Moreover, we believe that several other recent results in the concentration inequality literature may also be recovered by the methods present in the latter sections of this manuscript, such as optimal concentration in the Paley--Wiener space, conjectured by Donoho--Stark \cite{Donoho-Stark} and recently proved by Abreu--Speckbacher \cite{Abreu-Speckbacher-Paley-Wiener}, and concentration inequalities for the hyperbolic and Bergman balls in higher dimensions \cite{Kalaj-Ramos-Hyperbolic-Ball,Kalaj-Gromov-Ros}. We shall investigate this in a future manuscript.

\subsection{Further eigenvalues} The problem of optimizing the first localization eigenvalue $\lambda_1$ yields the natural question: what happens to higher-order eigenvalues of Gabor localization operators? What about extending the inverse-type problems discussed here to such cases?

In general, such results seem to be significantly more delicate than the current one, as the variation analysis encompasses a two-dimensional space of functions intrinsically. Still, we believe this to be a very interesting question, and we speculate that the techniques contained in Section~\ref{sec:local-max} may be useful in order to deduce singular properties that the maximizing sets must satisfy.

On the other hand, another very interesting recent direction concerns the so-called \emph{Ky Fan sums of eigenvalues}: namely, optimizing sums of the form
\[
\lambda_1(\Omega) + \lambda_2(\Omega) + \cdots + \lambda_N(\Omega)
\]
among sets of fixed volume $|\Omega| = s$. This question was initially raised by Nicola--Riccardi--Tilli \cite{Nicola-Riccardi-Tilli-Partial-Sums}, who solved the problem under the additional constraint that the sets be \emph{radial}; they subsequently proved the existence of extremizers in the general case \cite{Nicola-Riccardi-Tilli-Existence}. More recently, L. D. Abreu \cite{Abreu-Ky-Fan} showed that disks are, as a matter of fact, \emph{extremal} for such sums, by employing the work of De Palma on Husimi majorization \cite{DePalma-Wehrl}; see also the earlier Fock-majorization theorem of De Palma--Trevisan--Giovannetti \cite{DePalma-Trevisan-Giovannetti}.

In spite of the recent resolution of such a question, it is important to remark that, although disks are extremal for such a problem, a full characterization of the extremal sets is still wide open. Additionally, the higher-dimensional analogue of such an inequality does \emph{not} seem to follow easily from De Palma's work, as the techniques there seem to be two-dimensional in nature.

With those facts in mind, we believe that the techniques in this manuscript may be helpful in understanding the full structure of the extremal sets, as well as the inequality in higher dimensions. We intend to explore these problems in future work.

\bibliography{biblio,cherednichenko}
\bibliographystyle{amsplain}

\end{document}